\documentclass[a4paper,10pt]{amsart}
\usepackage{amsmath}
\usepackage{amsthm}
\usepackage{amssymb}
\usepackage{amsthm}
\usepackage[arrow, matrix, curve]{xy}
\usepackage{lmodern,dsfont}
\usepackage{enumerate}
\usepackage{url}
\usepackage{enumitem}

\newtheorem{theorem}{Theorem}[section]
\newtheorem{proposition}[theorem]{Proposition}
\newtheorem{corollary}[theorem]{Corollary}
\newtheorem{lemma}[theorem]{Lemma}

\theoremstyle{definition}
\newtheorem{definition}[theorem]{Definition}
\newtheorem{example}[theorem]{Example}
\newtheorem{remark}[theorem]{Remark}
\newtheorem*{numbering}{Numbering of the Roots}

\usepackage{color}

\title[On General Extension Fields for the Classical Groups]{On General Extension Fields for the Classical Groups in Differential Galois Theory}
\author{Matthias Seiss}
\email{matthias.seiss@mathematik.uni-kassel.de}
\address{Universit\"at Kassel\\ 
Fachbereich 10 \\ 
Heinrich Plett Str. 40 \\ 34132 Kassel \\ Germany.}
\begin{document}

\begin{abstract}
Let $G$ be one of the classical groups of Lie rank $l$. We make a similar 
construction of a general extension field in differential Galois theory 
for $G$ as E. Noether did in classical Galois theory for finite groups. More precisely, we build a differential 
field $E$ of differential transcendence degree $l$ over the constants 
on which the group $G$ acts and show that it is a Picard-Vessiot extension of the field of invariants $E^G$.
The field $E^G$ is differentially generated by $l$ differential 
polynomials which are differentially algebraically independent over the constants. They are the coefficients of the defining equation of the extension. 
Finally we show that our construction satisfies generic properties for a specific kind of 
$G$-primitive Picard-Vessiot extensions.
\end{abstract}

\maketitle

\section{Introduction}

In classical Galois theory there is a well-known construction of the general 
equation with Galois group the symmetric group $S_n$. As a starting point 
one takes $n$ indeterminates $\boldsymbol{T}=(T_1,\dots,T_n)$ 
and considers the rational function field $\mathbb{Q}(\boldsymbol{T})$.
The group $S_n$ acts on 
$\mathbb{Q}(\boldsymbol{T})$ by permuting the indeterminates and one can show  
that $\mathbb{Q}(\boldsymbol{T})$ is a Galois extension of the field of invariants 
$\mathbb{Q}(\boldsymbol{T})^{S_n}$ for a polynomial equation of degree $n$ whose coefficients 
are the elementary symmetric polynomials. The field of invariants is 
generated by these polynomials and they are algebraically independent over $\mathbb{Q}$. 
A generalization of this idea leads to 
the so called Noether problem
which would solve the inverse problem over $\mathbb{Q}$ for finite groups.
Unfortunately, the answer to the Noether problem is not affirmative.

In differential Galois theory one can construct in a similar way a general equation with 
differential Galois group the full general linear group $\mathrm{GL}_n(C)$. 
For an algebraically closed field $C$ of characteristic zero and $n$ differential 
indeterminates $\boldsymbol{y}=(y_1,\dots,y_n)$ one begins with the 
differential field $C\langle \boldsymbol{y} \rangle$, that is the field 
differentially generated by the indeterminates $\boldsymbol{y}$ over $C$. 
The action of the group $\mathrm{GL}_n(C)$ 
on $C\langle \boldsymbol{y} \rangle$ is induced by linear transformations 
from the right on the wronskian matrix in $\boldsymbol{y}$.
As in classical Galois theory one can show that $C\langle \boldsymbol{y} \rangle$ 
is a Picard-Vessiot extension of the field of invariants 
$C\langle \boldsymbol{y} \rangle^{\mathrm{GL}_n(C)}$ 
with differential Galois group $\mathrm{GL}_n(C)$.
As in the case of the symmetric group this extension is defined by a linear 
differential equation whose coefficients are differentially algebraically independent over $C$
and differentially generate the field of invariants.
This idea was generalized in \cite{Gold} by L.~Goldman and he applied the techniques 
developed there to some connected subgroups 
$G \subset \mathrm{GL}_n(C)$ obtaining explicit equations. 

In the present work we perform a similar construction in differential Galois theory 
for the classical groups. For this purpose let $G$ be one of these groups and denote by 
$l$ its Lie rank.
As in the above example we start our construction 
with a differential field $C\langle \boldsymbol{\eta} \rangle$ which is differentially 
generated by $l$ differential 
indeterminates $\boldsymbol{\eta}=(\eta_1,\dots,\eta_l)$ over the constants $C$. We  
use this purely differential transcendental extension to build our final general extension field 
$E \supset C\langle \boldsymbol{\eta} \rangle$ by taking into 
account the structure of $G$-primitive Picard-Vessiot extensions.  
These are Picard-Vessiot extensions with  
differential Galois group $G$ whose fundamental solution matrices satisfy the algebraic relations defining 
the group $G$.
The structural information of these extensions is obtained by connecting results from the theory 
of reductive groups with differential Galois theory.
More precisely, the Bruhat decomposition provides a normal form of elements of $G$ 
parameterized by a fixed Borel subgroup $B$ and the Weyl group. It turns out that 
the Bruhat decomposition of a fundamental matrix $\bar{Y}$ of a $G$-primitive Picard-Vessiot extension
is with respect to the maximal Weyl group element $\bar{w}$, stemming from the fact that $B\bar{w}B$ 
is a rational variety. Using the normal form $\bar{Y}=u_1 \bar{w} t u_2$, where $u_1$, $u_2$ are 
elements of the maximal unipotend subgroup of $B$ and $t$ lies in the torus, we found out that the entries of 
$t$ and $u_2$ generate a Liouvillian extension over the fixed field under the Borel subgroup, which on his 
part is generated by the entries of $u_1$.
In a specific case when the defining 
matrix has shape as constructed in \cite{Seiss}, the fixed field is 
differentially generated by Lie rank many elements $\boldsymbol{x}=(x_1,\dots,x_l)$ and the Liouvillian extension 
can be parameterized by these 
elements, that is, it is generated by $l$ exponentials, whose logarithmic derivatives are $\boldsymbol{x}$, and 
successive integrals of these exponentials. As a consequence the whole Picard-Vessiot extension is determined by the $l$ 
elements $\boldsymbol{x}$. According to these observations we construct our general extension field $E$. We use the indeterminates $\boldsymbol{\eta}$ to parameterize a general element of the 
Cartan subalgebra and this element together with the basis elements of root spaces belonging to the negative simple roots 
define the Liouvillian extension $E$ of $C\langle \boldsymbol{\eta} \rangle$ with 
fundamental matrix $t u_2$ which lies in the Borel group.
We show that we can build an element $u_1$ in the maximal unipotent subgroup with entries in 
$C\langle \boldsymbol{\eta} \rangle$ such that the logarithmic derivative of 
$Y=u_1 \bar{w} t u_2$ has shape as the matrix equations constructed in \cite{Seiss}.  
Similar as in the case of $\mathrm{GL}_n(C)$ the group $G$ acts on $E$ by right multiplication on $Y$, but now we take the Bruhat decomposition to obtain the effect on the generators of $E$. 
We prove that $E$ is a Picard-Vessiot extension of the 
field of invariants $E^G$ with differential Galois group $G$ and that  
$E^G$ is generated by $l$ invariants 
$\boldsymbol{h}=(h_1(\boldsymbol{\eta}),\dots,h_l(\boldsymbol{\eta}))$ which 
are differentially algebraically independent over the constants. The coefficients of the 
linear differential equation defining our extension   
are differential polynomials in these invariants. 
In opposition to the above approach we started in \cite{Seiss} with a differential 
field $C\langle \boldsymbol{t} \rangle$ generated by 
$l$ differential indeterminates $\boldsymbol{t}=(t_1,\dots,t_l)$ as our differential base field.
Under the consideration of the geometrical structure 
of the group $G$, we defined a parameter differential equation
$\partial(\boldsymbol{y})=A_G(\boldsymbol{t})\boldsymbol{y}$ and 
showed that it defines a Picard-Vessiot extension of $C\langle \boldsymbol{t} \rangle$ 
with differential Galois group $G$. 
Our extension field $E$ is build in such away that its defining equation 
is actually $\partial(\boldsymbol{y})=A_G(\boldsymbol{h})\boldsymbol{y}$, that is the 
indeterminates $\boldsymbol{t}$ can be taken to be the differentially algebraically 
independent invariants $\boldsymbol{h}$. In fact we have shown that for the 
parameterized equations of \cite{Seiss} we can perform a similar construction as
E.~Noether did in classical Galois theory. 

In contrast to the above example of the general equation 
for $\mathrm{GL}_n(C)$, where the extension field is generated by $n$ differential 
indeterminates over the constants, the differential transcendence degree of the extension 
field of our construction is equal to the Lie rank. 
In the first case the extension and equation are clearly generic. We can substitute any $n$
linearly independent solutions of any scalar linear differential equation of degree $n$ 
with differential Galois group a subgroup of $\mathrm{GL}_n(C)$ into $\boldsymbol{y}$ and 
the general equation specializes to it. At a first view
our construction is only generic for $G$-primitive Picard-Vessiot 
extensions of a differential field $F$ stemming from differential equations of shape 
$\partial(\boldsymbol{y})=A_G(\boldsymbol{f})\boldsymbol{y}$ with $\boldsymbol{f} \in F^l$
in the sense that we can substitute the $l$ elements $\boldsymbol{x}$ determining the extension 
into $\boldsymbol{\eta}$ which has the effect that $A_G(\boldsymbol{h})$ specializes to 
$A_G(\boldsymbol{f})$. In order to be generic for a broader range of Picard-Vessiot extension 
we consider extensions whose defining matrix is gauge equivalent to our equation. 
In \cite{Seiss} we have seen that every matrix of an open subset of the direct sum of the Lie algebra of $B^-$ and the root spaces corresponding to the positive simple roots 
is gauge equivalent to a matrix of shape $A_G(\boldsymbol{f})$. For such a defining matrix 
the logarithmic derivative 
of the Liouvillian part $u_2$ has the shape of a principal nilpotent 
matrix in normal form (for a definition see 
Chapter~\ref{The Structure of the Classical Groups}). 
It turns out that an arbitrary defining matrix is gauge equivalent to a 
matrix of shape $A_G(\boldsymbol{f})$ if and only if the defining matrix of 
its Liouvillian part  is gauge equivalent by specific 
transformations to principal nilpotent 
matrix in normal form. We will call such  extensions $G$-primitive with normalisable 
unipotent part and our general extension field will be generic for such extensions.

The paper is organized in the following way. 
In Chapter \ref{The Structure of the Classical Groups} we recapitulate 
some basic facts about the 
structure of the classical groups and their Lie algebras and introduce the 
corresponding notation. Chapter~\ref{The Gauge Transformation} shortly deals with 
the decomposition of the gauge transformation into the adjoint action and the 
logarithmic derivative.
In Chapter \ref{Connecting the structure of reductive groups with Picard-Vessiot extensions} 
the connection between 
the group structure and Picard-Vessiot extensions is established.   
The normal form decomposition of a fundamental matrix of a 
$G$-primitive Picard-Vessiot extension 
and the resulting normal form coefficients which generate the extension are given in Proposition \ref{general_bruhat of fundamental matrix}. 
Using the 
results we show in Proposition~\ref{Liouvillian_part_general} 
that the normal form coefficients which belong to the torus and the second 
unipotent group in the decomposition of the fundamental matrix generate 
a Liouvillian extension and that these coefficients are exponentials and successive integrals respectively.
In Chapter~\ref{The construction of the general extension field}
we construct our general extension field. We prove in Theorem~\ref{theorem_noetherstyl} that it
is a Picard-Vessiot extension of the differential field generated by the 
coefficients of the logarithmic derivative of the matrix $Y$ from above and that the 
differential Galois
group of the extension is $G$. Reinterpreting the result in Remark~\ref{corollary_noetherstyl} we obtain that 
$E$ with the action of $G$ has the above outlined properties.
Every step in the construction 
comes along with an example for the group $\mathrm{SL}_4$. At the end of the 
chapter we consider the case of the group $\mathrm{G}_2$. Chapter \ref{Full $G$-Primitive Picard-Vessiot Extensions and the General Extension field} deals with the generic 
properties of our general extension field. For a full $G$-primitive Picard-Vessiot extension 
we give in Theorem~\ref{proposition_transf_principal nilpotent} a criteria for gauge equivalence 
of its defining matrix to a matrix of shape $A_G(\boldsymbol{f})$ in terms of its Liouvillian 
part. Our finial results are stated in Theorem~\ref{main_theorem_generic}.

\section{The structure of the classical groups} 
\label{The Structure of the Classical Groups}
Let $G$ be one of the classical groups of Lie rank $l$ over $C$. 
Let $\Phi$ be the root system of $G$ of the respective type and let $
\Delta=\{ \bar{\alpha}_1, \dots , \bar{\alpha}_l \}$ be a basis of $\Phi$. 
We write $\Phi^+$ for the set of positive and $\Phi^-$ for the set of negative roots and we set 
$m:= \mid \Phi^+ \mid = \mid \Phi^- \mid $. 
Every root $\alpha$ can be written uniquely as 
\begin{equation*}
 \alpha =   k_1 \bar{\alpha}_1 + \dots + k_l \bar{\alpha}_l
\end{equation*}
with integer coefficients $k_i  $ where all $k_i$ are nonpositive or nonnegative. 
By the integer $\mathrm{ht}(\alpha)= k_1 + \dots + k_l$ we mean the height of the root $\alpha$.

We denote by $\mathcal{W}$ the Weyl group of $\Phi$, that is the finite group 
which consists of the orthogonal reflections 
\[
w_{\alpha}(\beta) = \beta - \langle \beta , \alpha \rangle  \alpha 
\]
for $\alpha, \ \beta \in \Phi$ 
where  
$\langle \beta , \alpha \rangle=2(\beta,\alpha)(\alpha,\alpha)^{-1}$ and $(\cdot,\cdot)$ denotes the inner product of the corresponding ambient real vector space of $\Phi$. The Weyl group is generated by the reflections $w_{\bar{\alpha}_i}$
for the simple roots $\bar{\alpha}_i \in \Delta$ and every $w_{\alpha} \in \mathcal{W}$
can be written as 
\[
w_{\alpha}=w_{\bar{\alpha}_{i_1}} \cdots w_{\bar{\alpha}_{i_k}} 
\]
with $k$ minimal. The integer $k$ is called the length of $w_{\alpha}$. 
We denote by $\bar{w}$ the unique element of maximal length, that is the element
which maps $\Phi^+$ to $\Phi^-$ and vice versa. 

Let $\mathfrak{h}$ be a maximal Cartan subalgebra of the Lie algebra 
$\mathfrak{g}$ of $G$. Then with respect to $\mathfrak{h}$ we obtain a root space decomposition
\begin{equation*}
 \mathfrak{g}= \mathfrak{h} \oplus \bigoplus_{\alpha \in \Phi} \mathfrak{g}_{\alpha}
\end{equation*}
where we denote by $\mathfrak{g}_{\alpha}$ the one-dimensional root spaces of $\mathfrak{g}$.   
With respect to this root space decomposition we can choose a Chevalley basis 
\begin{equation*}
 \{H_{ i} \mid \bar{\alpha}_i \in \Delta \} \cup \{ X_{\alpha} \mid \alpha \in \Phi \}
\end{equation*}
where $\mathfrak{h}= \langle H_1 , \dots , H_l \rangle$ and 
$\mathfrak{g}_{\alpha}= \langle X_{\alpha} \rangle$. Further we consider in the following the 
maximal nilpotent subalgebras $\mathfrak{u}= \sum_{\alpha \in \Phi^+} \mathfrak{g}_{\alpha}$
and $\mathfrak{u}^-= \sum_{\alpha \in \Phi^-} \mathfrak{g}_{\alpha}$
defined by all positive and all negative roots respectively. 
Moreover, let $\mathfrak{b}^+ = \mathfrak{h} + \mathfrak{u}^+$ 
be the maximal solvable subalgebra of $\mathfrak{g}$ which contains the positive 
maximal nilpotent subalgebra $\mathfrak{u}^+$  
and the Cartan subalgebra $\mathfrak{h}$. Analogously let 
$\mathfrak{b}^- = \mathfrak{h} + \mathfrak{u}^-$. 

For $X, \ Y \in \mathfrak{g}$ we denote by $[X,Y]$ the 
usual bracket product and we write $\mathrm{ad}(X)$ for the endomorphism 
$\mathrm{ad}(X): Y \mapsto [X,Y]$. For an element $g \in G$ we denote by 
\begin{equation*}
\mathrm{Ad}(g): \mathfrak{g} \rightarrow \mathfrak{g}, \ X \mapsto gXg^{-1} 
\end{equation*}
the adjoint action of $G$ on $\mathfrak{g}$.

For $l$ nonzero elements $\boldsymbol{s}=(s_1,\dots,s_l)$ of $C$ let
\begin{equation*}
  A_0^+(\boldsymbol{s}):=  \sum_{i=1}^l s_i X_{\bar{\alpha}_i}   \quad \mathrm{and} \quad  
  A_0^-(\boldsymbol{s}):=  \sum_{i=1}^l s_i X_{-\bar{\alpha}_i}     
\end{equation*}
be the sum of all basis elements which correspond to the positive simple roots 
(resp. negative simple roots) with coefficients $\boldsymbol{s}$. 
We call $A_0^-(\boldsymbol{s})$ a {\bf principal nilpotent matrix in normal form} and 
to shorten notation we write $A_0^+$ and $A_0^-$ in case of $\boldsymbol{s}=(1,\dots,1)$. 

\begin{numbering}
 We order the negative roots $\Phi^- = \{\beta_1,\dots, \beta_m\}$ such that 
$\mathrm{ht}(\beta_i) \geq \mathrm{ht}(\beta_{i+1})$ for all $i=1,\dots, m-1$. 
Then the indices of all roots of a given height form an unbroken string $i_1,\dots,i_2$. 
Let $\gamma_1, \dots, \gamma_l$ be complementary roots of $\Phi^-$ (see \cite[Lemma 6.4 and Definition 6.5]{Seiss}). 
For every height we reorder the roots such that the complementary roots have  the greatest indices, 
that is, the indices of the non-complementary roots of a given height 
form an unbroken string $i_1,\dots i_2'$. 
\end{numbering}

We rename the basis elements 
$\{ X_{\beta_i} \mid \beta_i \in \Phi^- \}$ of $\mathfrak{u}^-$ into $X_i$.
One can decompose $\mathfrak{g}$ into a direct sum of subspaces 
$\mathfrak{g}^{(j)}$ and $\mathfrak{h}$ where each $\mathfrak{g}^{(j)}$ is the direct sum of root spaces corresponding to roots of height $j$ (see \cite[Lemma 6.1]{Seiss}). We denote
by $r(i)$ the negative number such that 
$X_i \in \mathfrak{g}^{(r(i))}$, that is $r(i)=\mathrm{ht}(\beta_i)$. Then the basis is ordered such 
that $r(i) \geq r(i+1)$ for all $i=1,\dots ,m-1$. For each $X_i$ we define 
\[
W_i := [X_i, A_0^+ ].
\]
Then $W_i \in \mathfrak{g}^{(r(i)+1)}$ and the set $\{ W_i \mid i=1, \dots, m \}$ is a basis 
of $\mathrm{ad}(A_0^+)(\mathfrak{u}^-)$. Further by \cite[Lemma 6.4]{Seiss} the set 
\begin{equation*}
  \{ W_i \mid i=1,\dots,m\} \cup \{ X_{\gamma_k} \mid k=1,\dots,l \} 
\end{equation*}
is a basis of $\mathfrak{b}^-$.

We denote by $T$ the maximal torus of $G$ whose Lie algebra is $\mathfrak{h}$. Further by $U^+$ 
and $U^-$ we mean the maximal unipotent groups whose Lie algebras are $\mathfrak{u}^+$ 
and $\mathfrak{u}^-$ respectively. Finally we write $B^+$ and $B^-$ for the Borel 
subgroups of $G$ with respective Lie algebras $\mathfrak{b}^+$ and $\mathfrak{b}^-$. 

For a root $\alpha \in \Phi$ we denote by $U_{\alpha}$ the one-dimensional 
root group of $G$ whose Lie algebra 
coincide with the root space $\mathfrak{g}_{\alpha}$. There is an isomorphism 
$\mathbb{G}_a \rightarrow U_{\alpha}$ and for $x \in C$ we denote by $u_{\alpha}(x)$ 
the image of $x$ in $U_{\alpha}$. For the negative roots $\beta_1,\dots,\beta_m$, which 
we ordered in a specific way,
we rename the root groups $U_{\beta_i}$ into $U_i$ and their elements into $u_i(x)$.
For $i=1,\dots,l$ let $T_i$ be the one-dimensional 
subtorus of $T$ whose Lie algebra coincides with $\langle H_i \rangle \subset \mathfrak{h}$. 
Then there is an isomorphism $\mathbb{G}_m \rightarrow T_i$ and  
for $z \in C^{\times}$ we denote by $t_i(z)$ the image of $z$ in $T_i$.
For $m$ elements $\boldsymbol{x}=(x_1, \dots, x_m)$ of $C$ we denote in the 
following by $\boldsymbol{u}(\boldsymbol{x})$ the product 
\[
 \boldsymbol{u}(\boldsymbol{x}):=u_1(x_1) \cdots u_m(x_m)  \in U^- 
\]
and for $l$ elements $\boldsymbol{z}=(z_1, \dots, z_l)$ of $C^{\times}$ we write 
$\boldsymbol{t}(\boldsymbol{z})$ for the product 
\[
\boldsymbol{t}(\boldsymbol{z}) :=t_1(z_1) \cdots t_l(z_l) \in T.
\]

An important result in the structure theory of reductive groups is the Bruhat decomposition. It gives 
a normal form for elements of $G$ parametrized by a fixed Borel subgroup and the Weyl group.
\begin{theorem}[Bruhat Decomposition] Fix a Borel subgroup of $G$. Then 
 we have $G= \bigcup_{w \in \mathcal{W}} B w B$ (disjoint union) with
 $B w B = B \tilde{w} B$ if and only if $w = \tilde{w}$ in $\mathcal{W}$.
\end{theorem}
\begin{proof}
 See \cite[Theorem 28.3]{HumGroups}.
\end{proof}
For $w \in \mathcal{W}$ we fix a representative $n(w)$ of $w$ in the 
normalizer $N_G(T)$ of $T$ in $G$. Following the discussion of \cite[Chapter  28.1]{HumGroups} 
we obtain a $T$-stable subgroup $U'_{w}=U \cap n(w)U^- n(w)^{-1}$ of $U$ for $w$. 
A normal form for elements of $G$ can now be made unique.
\begin{theorem}\label{Bruhat1}
 For each $w \in \mathcal{W}$ fix a coset representative $n(w) \in N_G(T)$. 
 Then each element $x \in G$ can be written in the form $x=u'n(w)t u $, 
 where $w \in \mathcal{W}$, $t \in T$, $u  \in U$ and $u' \in U'_{w}$
 are all determined uniquely by $x$. 
\end{theorem}
\begin{proof}
 See \cite[Theorem 28.4]{HumGroups}.
\end{proof}
Borel subgroups are maximal connected solvable subgroups of $G$ and so by 
\cite[Theorem 19.3 and Exercise 17.7]{HumGroups} 
there is 
a descending chain of closed subgroups of $B^-$ where each group is a normal subgroup of its 
predecessor and their quotient is isomorphic to $\mathbb{G}_a$ or $\mathbb{G}_m$.  
For $ i = 0, \dots , l-1$ define $\bar{T}_i:=T_{i+1} \cdots T_{l}$ and denote by $\bar{B}^{-}_i$ the subgroup 
$\bar{T}_i \ltimes U^-$ of $B^-$. Moreover for $i= 0 , \dots ,m-1$ denote by $\bar{U}^{-}_i$ the subgroup 
$U_{i+1} \cdots U_{m}$ and let $\bar{U}^{-}_m=\{e\}$. 
Following \cite[Theorem 5.3.3]{Carter} these groups form a descending chain of closed subgroups 
\begin{equation}  \label{chainofsubgroups}
 B^- =\bar{B}^{-}_0 \supset \bar{B}^{-}_1 \supset \dots \supset \bar{B}^{-}_{l-1} \supset 
 \bar{U}^{-}_0 \supset \bar{U}^{-}_1 \supset \dots \bar{U}^{-}_{m-1} \supset \bar{U}^{-}_m=\{e\}
\end{equation}
where each group is a normal subgroup of its predecessor and their quotient is isomorphic to 
$\mathbb{G}_a$ or $\mathbb{G}_m$. More precisely, for $ i = 0, \dots , l-2$ the quotients 
$\bar{B}^{-}_i/\bar{B}^{-}_{i+1}$ and $\bar{B}^{-}_{l-1}/\bar{U}^{-}_0$ are isomorphic to $\mathbb{G}_m$ and for 
$i = 0 , \dots ,m-1$ the quotients $\bar{U}^{-}_i/\bar{U}^{-}_{i+1}$ are isomorphic to $\mathbb{G}_a$. 
We denote the positive roots by $\Phi^+ = \{\alpha_1 , \dots , \alpha_m \}$ and we order 
them in the same way as the negative roots. If we define as in the case of $B^-$ and $U^-$ 
the subgroups $\bar{B}^{+}_i\subset B^+$ and $\bar{U}^{+}_i \subset U^+$, then these subgroups form a 
chain with the same properties.

\section{The gauge transformation}
\label{The Gauge Transformation} 
Let $F$ be a differential field with derivation $\partial$ and field of constants $C$. 
In the following we mean by 
$G(F)$ and $\mathfrak{g}(F)$ the group and Lie algebra of $F$-rational points.
\begin{definition}
 The map 
 \begin{equation*}
  \ell \delta: \mathrm{GL}_n(F) \rightarrow F^{n \times n}, g \mapsto \partial(g)g^{-1}
 \end{equation*}
is called the logarithmic derivative.
\end{definition}
A gauge transformation of an element $A \in \mathfrak{g}(F)$ by an element $g$ of $G(F)$
is defined as   
\begin{equation*}
 \mathrm{Ad}(g)(A) + \ell \delta(g) .   
\end{equation*}
We can decompose the gauge transformation of $A$ into  
the image $\mathrm{Ad}(g)(A)$ of $A$ under the adjoint action by $g$ and into 
the image $\ell \delta(g)$ 
of $g$ under the logarithmic derivative. We will mostly look at the two images 
separately and then afterwards combine the results. This makes it possible to use the 
root structure of $G$ and $\mathfrak{g}$ to describe the gauge transformation.
 
Proposition~\ref{log_derivate} below shows that the image of an element of a linear algebraic group 
under the logarithmic derivative lies in its Lie algebra.

\begin{proposition}\label{log_derivate}
 Let $G \subset \mathrm{GL}_n$ be a linear algebraic group. Then the restriction of $\ell\delta$ to $G$
 maps $G(F)$ to its Lie algebra $\mathfrak{g}(F)$, that is
 \begin{equation*}
  \ell \delta \mid_{G}: G(F) \rightarrow \mathfrak{g}(F)
 \end{equation*} 
\end{proposition}
\begin{proof}
 A proof can be found in \cite{Kov}.
\end{proof}
In order to understand better the adjoint action we will use Remark~\ref{remark_adjoint_action}
below. It gives a description of the image of elements of a Chevalley basis under the adjoint 
action of a root group element by the root system.
\begin{remark}\label{remark_adjoint_action}
  For $\alpha$, $\beta \in \Phi$ linearly independent let
 $\alpha - r \beta, \dots ,\alpha + q \beta$ be the $\beta$-string through $\alpha$ 
 for $r, q \in \mathbb{N}$ and let $\langle \alpha, \beta \rangle$ be the Cartan integer. We have
  \begin{eqnarray*}
   \mathrm{Ad}(u_{\beta}(x))(X_{\alpha}) & =& \sum\nolimits_{i=0}^q c_{\beta, \alpha,i} x^i X_{\alpha + i \beta}, \\
   \mathrm{Ad}(u_{\beta}(x))(H_{\alpha}) &=& H_{\alpha} - \langle \alpha, \beta \rangle x X_{\beta} , \\
 \mathrm{Ad}(u_{\beta}(x))(X_{-\beta}) &=& X_{-\beta} + x H_{\beta} - x^2  X_{\beta} 
\end{eqnarray*}
where $c_{\beta,\alpha,0}=1$ and $c_{\beta,\alpha,i}= \pm \binom{r+i}{i}$.
\end{remark}

\section{Connecting the structure of the classical groups\\ with Picard-Vessiot extensions}
\label{Connecting the structure of reductive groups with Picard-Vessiot extensions}
In this setion we will establish a connection between the geometrical structure of a 
classical group $G$ and a $G$-primitive Picard-Vessiot extension 
(see Definition~\ref{def:GPrimitivePVE})
of the differential field $F$. 
This link will be obtained by applying the Bruhat decomposition to a fundamental 
solution matrix of the extension. 
In the following we denote by $C[\mathrm{GL}_n]$ the ring 
$C[X_{ij},\mathrm{det}(X_{ij})^{-1}]$ in the indeterminates $X_{ij}$ with $1 \leq i,j \leq n$ and
for a linear algebraic group $H \subseteq \mathrm{GL}_n$ we write $C[H]$ for the coordinate ring of $H$, that is the quotient ring of $C[\mathrm{GL}_n]$ by the defining ideal of $H$.
Further we denote by 
$\overline{X}_{ij}$ the image of $X_{ij}$ in $C[H]$ and we write shortly $X$ and $\overline{X}$
for the matrices $(X_{ij})$ and $(\overline{X}_{ij})$ respectively.
Finally, if $H$ is connected we 
mean by $C(H)$ the field of fractions of the coordinate ring $C[H]$.


\begin{lemma}\label{decomp_for_matrix_of_B}
There are algebraically independent elements $\boldsymbol{z}=(z_1, \dots,z_l)$ and 
$\boldsymbol{y}=(y_1, \dots,y_m)$ of $C[B^-]$ such that 
\begin{equation*}
\overline{X} =  \boldsymbol{t}(\boldsymbol{z})
\boldsymbol{u}(\boldsymbol{y}).
\end{equation*}
\end{lemma}

\begin{proof}
We shortly write $\boldsymbol{U}$ for the direct product $U_1 \times \dots \times U_m$ and  
analogously $\boldsymbol{T}$ for the direct product $T_1 \times \dots \times T_l$.
The product maps $\boldsymbol{U} \rightarrow U^-$ and $\boldsymbol{T} \rightarrow T$ are 
isomorphisms of varieties and since $T_i$ and $U_i$ are isomorphic to $\mathbb{G}_m$ 
and $\mathbb{G}_a$, we identify the coordinate ring of $T_i$ with 
$C[\bar{z}_i,\bar{z}_i^{-1}]$ and the coordinate ring of $U_i$ 
with $C[\bar{y}_i]$ for some new indeterminates $\bar{z}_i$ and $\bar{y}_i$ respectively. 
Together we obtain a $C$-algebra isomorphism 
\begin{equation*}
C[\boldsymbol{T} \times \boldsymbol{U} ] \rightarrow 
C[\bar{z}_1,\bar{z}_1^{-1},\dots,\bar{z}_l,\bar{z}_l^{-1} , \bar{y}_1,\dots,  \bar{y}_m]  
\end{equation*}
where $\boldsymbol{\bar{z}}=(\bar{z}_1, \dots ,\bar{z}_l)$ and 
$\boldsymbol{\bar{y}}=(\bar{y}_1, \dots ,\bar{y}_m)$ are algebraically independent.
From \cite[Theorem 10.6 (4)]{Borel} we obtain that the product map from 
$T \times U^-$ to $B^-$ is an isomorphism and combined with the above product maps 
it follows that the map 
\begin{equation*}
\varphi:\boldsymbol{T} \times \boldsymbol{U} \rightarrow B^-, \
(t_1,\dots,t_l,u_1,\dots,u_m) \mapsto  t_1\cdots  t_l   u_1 \cdots u_m 
\end{equation*}
is an isomorphism of varieties. Thus its comorphism 
\begin{equation*}
\varphi^*: C[B^-] \rightarrow C[\boldsymbol{T} \times \boldsymbol{U}], \ \overline{X}_{ij} 
\mapsto (\boldsymbol{t}(\boldsymbol{\bar{z}})  \boldsymbol{u}(\boldsymbol{\bar{y}}))_{ij}
\end{equation*}
is an isomorphism of $C$-algebras and its inverse sends the algebraic independent elements 
$\boldsymbol{\bar{z}}$ and $\boldsymbol{\bar{y}}$ to algebraic independent elements 
$\boldsymbol{z}=(z_1, \dots ,z_l)$ and $\boldsymbol{y}=(y_1, \dots ,y_m)$ of the coordinate 
ring of $B^-$.  Putting everything together we have  
\begin{equation*}
\overline{X}=\varphi^{*-1}(\boldsymbol{t}(\boldsymbol{\bar{z}})  
\boldsymbol{u}(\boldsymbol{\bar{y}})) 
=\boldsymbol{t}( \boldsymbol{z})  \boldsymbol{u}( \boldsymbol{y})
\end{equation*}
where $\boldsymbol{z}$ and $\boldsymbol{y}$ have the required property.
\end{proof}

\begin{lemma} \label{decomp_for_matrix_of_G}
There are algebraically independent elements $\boldsymbol{x}=(x_1, \dots,x_m)$,
$\boldsymbol{z}=(z_1, \dots,z_l)$ and $\boldsymbol{y}=(y_1, \dots,y_m)$ of the 
field of fractions $C(G)$ of $C[G]$ such that 
\begin{equation*}
\overline{X} = \boldsymbol{u}(\boldsymbol{x})n(\bar{w}) \boldsymbol{t}(\boldsymbol{z})
\boldsymbol{u}(\boldsymbol{y}).
\end{equation*}
\end{lemma}
\begin{proof}
Let $\boldsymbol{U}$ and $\boldsymbol{T}$ be as in the proof of Lemma~\ref{decomp_for_matrix_of_B}. 
We show that the map $\varphi: \boldsymbol{U} \times \boldsymbol{T} \times \boldsymbol{U} \rightarrow G $ 
defined by
\begin{equation*}
(u_1,\dots,u_m,t_1,\dots,t_l,\bar{u}_1,\dots,\bar{u}_m)
\mapsto u_1\cdots u_m n(\bar{w}) t_1 \cdots t_l \bar{u}_1 \cdots \bar{u}_m
\end{equation*}
is an isomorphism onto an open subset of $G$, that is, we prove that 
$\boldsymbol{U} \times \boldsymbol{T} \times \boldsymbol{U}$ and $G$ are birational equivalent. 
The image of $\varphi$ is the variety $U^-n(\bar{w})B^- \subset G$ and we have to show 
that it is an open subset of $G$. By \cite[Chapter 14.14]{Borel} the subset 
\begin{equation*}
 n(\bar{w})U^-n(\bar{w})B^-=U^+ B^-
\end{equation*}
is open in $G$. It is the isomorphic image of $U^-n(\bar{w})B^-$ under left multiplication with 
$n(\bar{w})$. Since left multiplication is a continuous morphism, we conclude that 
$U^-n(\bar{w})B^-$ is an open subset of $G$. Hence, $\varphi$ is a birational morphism.

As in the proof of Lemma~\ref{decomp_for_matrix_of_B} we identify the coordinate ring 
$C[\boldsymbol{U}\times \boldsymbol{T}\times \boldsymbol{U}]$ with 
\begin{gather*}
C[\bar{x}_1,\dots,  \bar{x}_m,\bar{z}_1,\bar{z}_1^{-1},\dots,\bar{z}_l,\bar{z}_l^{-1} , 
\bar{y}_1,\dots,  \bar{y}_m]
\end{gather*}
for algebraically independent elements $\boldsymbol{\bar{x}}=(\bar{x}_1, \dots, \bar{x}_m)$, 
$\boldsymbol{\bar{z}}=(\bar{z}_1, \dots, \bar{z}_l)$ and 
$\boldsymbol{\bar{y}}=(\bar{y}_1, \dots, \bar{y}_m)$ which represent the coordinates of 
the first, second and third factor of $\boldsymbol{U}\times \boldsymbol{T}\times \boldsymbol{U}$ 
respectively. Since $\varphi$ is a birational morphism, its comorphism 
\begin{equation*}
\varphi^{*} : C(G) \rightarrow C(\boldsymbol{U}\times \boldsymbol{T}\times \boldsymbol{U}), \  
\overline{X}_{ij} \mapsto
(\boldsymbol{u}(\boldsymbol{\bar{x}})n(\bar{w})\boldsymbol{t}(\boldsymbol{\bar{z}})
\boldsymbol{u}(\boldsymbol{\bar{y}}))_{ij}
\end{equation*}
is an isomorphism of fields. Its inverse sends the algebraically independent elements 
$\boldsymbol{\bar{x}}$, $\boldsymbol{\bar{z}}$ and 
$\boldsymbol{\bar{x}}$ to algebraically independent elements $\boldsymbol{x}=(x_1, \dots,x_m)$,
$\boldsymbol{z}=(z_1, \dots,z_l)$ and $\boldsymbol{y}=(y_1, \dots,y_m)$ of $C(G)$. 
This yields  
\begin{equation*}
\overline{X}=\varphi^{*-1}( \boldsymbol{u}(\boldsymbol{\bar{x}}) n(\bar{w}) 
\boldsymbol{t}(\boldsymbol{\bar{z}}) \boldsymbol{u}(\boldsymbol{\bar{y}}))=
\boldsymbol{u}( \boldsymbol{x})n(\bar{w}) \boldsymbol{t}( \boldsymbol{z}) \boldsymbol{u}( \boldsymbol{y})
\end{equation*}
with the required properties.
\end{proof}

\begin{corollary}\label{corollary_decomp_for_matrix_of_G}
Let 
 \begin{equation*}
\overline{X}= 
\boldsymbol{u}(\boldsymbol{\bar{x}})n(\bar{w}) \boldsymbol{t}(\boldsymbol{\bar{z}})
\boldsymbol{u}(\boldsymbol{\bar{y}}) 
\end{equation*}
 be the normal form as in Lemma~\ref{decomp_for_matrix_of_G}.
 \begin{enumerate}[label=(\alph*),ref=\thetheorem(\alph*)]
   \item \label{cor_decomp_for_matrix_of_G_point_a} Let $g$ of $G(C)$. 
     Then there are over $C$ algebraically independent elements 
     $\boldsymbol{x}=(x_1, \dots,x_m)$, $\boldsymbol{z}=(z_1, \dots,z_l)$ and 
     $\boldsymbol{y}=(y_1, \dots,y_m)$ of $C(G)$ such that 
     \begin{equation*}
     \overline{X} g = \boldsymbol{u}(\boldsymbol{x})n(\bar{w}) \boldsymbol{t}(\boldsymbol{z})
          \boldsymbol{u}(\boldsymbol{y}). 
     \end{equation*}
  \item \label{cor_decomp_for_matrix_of_G_point_b} Let $i =0,\dots,l-1$ and $b \in \bar{B}^-_{i}(C)$. 
     Then there are over $C(\boldsymbol{\bar{x}},\bar{z}_1,\dots,\bar{z}_{i})$ 
     algebraically independent elements  
     $z_{i+1},\dots,z_l$ and $\boldsymbol{y}=(y_1, \dots,y_m)$ of $C(G)$ such that 
     \begin{equation*}
      \overline{X} b = \boldsymbol{u}(\boldsymbol{\bar{x}})n(\bar{w}) t_1(\bar{z}_1) \cdots 
      t_i(\bar{z}_i)t_{i+1}(z_{i+1})\cdots t_l(z_l) \boldsymbol{u}(\boldsymbol{y}). 
     \end{equation*}
  \item \label{cor_decomp_for_matrix_of_G_point_c} Let $i =0,\dots,m-1$ and $u \in \bar{U}^{-}_i$. 
  Then there are over 
  $C(\boldsymbol{\bar{x}},\boldsymbol{\bar{z}},\bar{y}_1,\dots,\bar{y}_{i})$ 
  algebraically independent elements $y_{i+1},\dots,y_m$ of $C(G)$ such that 
     \begin{equation*}
      \overline{X} u = \boldsymbol{u}(\boldsymbol{\bar{x}})n(\bar{w}) 
      \boldsymbol{t}(\boldsymbol{\bar{z}})
      u_1(\bar{y}_1) \cdots u_i(\bar{y}_i)u_{i+1}(y_{i+1})\cdots u_m(y_m) .
     \end{equation*}
\end{enumerate}
\end{corollary}
\begin{proof}
 (a) Let $g \in G(C)$. The translation map 
 $\psi: G(C) \rightarrow G(C), \ x \mapsto xg$ 
 is an isomorphism of varieties and so its comorphism is an $C$-isomorphism of fields
 \begin{equation*}
  \psi^*: C(G)\rightarrow C(G), \ \overline{X}_{ij} \mapsto (\overline{X}g)_{ij}.
 \end{equation*}
 Applying $\psi^*$ to the normal form of $\overline{X}$ we obtain 
 \begin{equation*}
  \overline{X}g = \psi^*(\boldsymbol{u}(\boldsymbol{\bar{x}})n(\bar{w}) 
  \boldsymbol{t}(\boldsymbol{\bar{z}}) \boldsymbol{u}(\boldsymbol{\bar{y}}))=
  \boldsymbol{u}(\boldsymbol{x})n(\bar{w}) \boldsymbol{t}(\boldsymbol{z})
  \boldsymbol{u}(\boldsymbol{y})
 \end{equation*}
where the images $\boldsymbol{x}=(x_1, \dots,x_m)$, $\boldsymbol{z}=(z_1, \dots,z_l)$ and 
$\boldsymbol{y}=(y_1, \dots,y_m)$ of $\boldsymbol{\bar{x}}$, $\boldsymbol{\bar{z}}$ and 
$\boldsymbol{\bar{y}}$ under $\psi^*$ are algebraically independent over $C$.

(b) Let $b \in \bar{B}^-_i(C)$ and represent it as a 
product $t_{i+1}\cdots t_l u$ with $u \in U^-(C)$ and $t_{i+1} \in T_{i+1}(C), \dots, t_l \in T_l(C)$. 
Using that $U^-(C(G)) $ is normal in $\bar{B}^-_i(C(G)) $ we compute
\begin{gather*}
\boldsymbol{t}(\boldsymbol{\bar{z}})\boldsymbol{u}(\boldsymbol{\bar{y}}) t_{i+1} \cdots t_l u =
\boldsymbol{t}(\boldsymbol{\bar{z}})t_{i+1}\cdots t_l \boldsymbol{u}(\boldsymbol{y}) =  
t_1(\bar{z}_1) \cdots 
t_i(\bar{z}_i)t_{i+1}(z_{i+1})\cdots t_l(z_l)
\boldsymbol{u}(\boldsymbol{y})
\end{gather*}
where $z_{i+1},\dots ,z_l$ and $\boldsymbol{y}=(y_1,\dots,y_m) $ 
are elements of $C(G)$. From the last equation we obtain the normal form 
\begin{gather*}
\boldsymbol{u}(\boldsymbol{\bar{x}})n(\bar{w}) \boldsymbol{t}(\boldsymbol{\bar{z}})
\boldsymbol{u}(\boldsymbol{\bar{y}}) b = \boldsymbol{u}(\boldsymbol{\bar{x}})n(\bar{w})
t_1(\bar{z}_1) \cdots t_i(\bar{z}_i)t_{i+1}(z_{i+1})\cdots t_l(z_l)
\boldsymbol{u}(\boldsymbol{y}).
\end{gather*}
Since the elements in the normal form are uniquely determined, 
it follows with (a) that $\boldsymbol{\bar{x}}$,  
$(\bar{z}_1,\dots,\bar{z}_{i}, z_{i+1},\dots,z_l)$
and $\boldsymbol{y}$ are algebraically independent over $C$.

(c) For $u \in \bar{U}^{-}_i(C)$ the subgroup structure 
gives the identity  
\begin{equation*}
\boldsymbol{u}(\boldsymbol{\bar{y}})u = u_1(\bar{y}_1) \cdots 
u_i(\bar{y}_i)u_{i+1}(y_{i+1})\cdots u_m(y_m)
\end{equation*}
with $y_{i+1},\dots ,y_m$ in $C(G)$. 
The rest of the proof works as in (b).
\end{proof}

The next step is to connect our results with Picard-Vessiot theory. 
We will mainly consider a specific type of Picard-Vessiot extensions. 

\begin{definition}\label{def:GPrimitivePVE}
 A Picard-Vessiot extension $E/F$ is called a full $G$-primitive extension of $F$ if 
 the differential Galois group of $E/F$ is $G(C)$ and if there is a matrix $Y \in G(E)$ 
 whose entries generate $E$ over $F$ and which satisfies 
 $\ell \delta (Y)=A \in \mathfrak{g}(F)$. 
\end{definition}

\begin{proposition}\label{general_bruhat of fundamental matrix} 
Let $E/F$ be a full $G$-primitive Picard-Vessiot extension with matrix $Y \in G(E)$.
 Then there are over $F$ algebraically independent elements  
 $\boldsymbol{x}=(x_{1}, \dots, x_m)$, $\boldsymbol{z}=( z_{1}, \dots , z_l)$ and 
 $\boldsymbol{y}=(y_{1}, \dots ,y_m)$ of $E$ such that 
 \begin{equation*}
 Y= \boldsymbol{u}(\boldsymbol{x})   n(\bar{w}) 
  \boldsymbol{t}(\boldsymbol{z})  
  \boldsymbol{u}(\boldsymbol{y}) .
 \end{equation*}  
\end{proposition}
\begin{proof}
The matrix $A=\ell \delta(Y)  \in \mathfrak{g}(F)$ is a defining 
matrix for the Picard-Vessiot extension $E$ of $F$. We use the standard construction method 
to built a differentially isomorphic Picard-Vessiot extension for $A$. 
Since $A  \in \mathfrak{g}(F)$ and the differential 
Galois group for $A$ is $G(C)$ the defining ideal of 
$G$ in $C[\mathrm{GL}_n ]$ extends to a maximal differential ideal of $F[\mathrm{GL}_n] $. 
Its quotient ring $S$ is then a Picard-Vessiot ring  for $A$ and the matrix 
$\overline{X} \in G(S)$ is by construction a fundamental solution matrix. 
Denote by $L$ the field of fractions of $S$. We have the embedding $C(G) \hookrightarrow L$ 
and so the normal form of $\overline{X}$ in $G(C(G))$ which we obtain from 
Lemma \ref{decomp_for_matrix_of_G}
yields the normal form of 
the fundamental matrix $\overline{X} \in G(L)$.

Since $E$ and $L$ are both Picard-Vessiot extensions for $A$, 
there is a differential $F$-algebra isomorphism 
\begin{equation*}
\varphi: L \rightarrow E, \ \overline{X} \mapsto Y g
\end{equation*}
where $g \in \mathrm{GL}_n(C)$. Because the matrix $\overline{X}$ satisfies the defining 
conditions for $G$, its image also does and so $Yg$ lies in $G(E)$. We conclude that $g$ also lies 
in $G(C)$. Applying Corollary~\ref{corollary_decomp_for_matrix_of_G} to 
$\overline{X}g^{-1}$ we obtain that there are over $C$ algebraically independent elements 
$\boldsymbol{\bar{x}}=(\bar{x}_1, \dots,\bar{x}_m)$, 
$\boldsymbol{\bar{z}}=(\bar{z}_1, \dots,\bar{z}_l)$ and 
$\boldsymbol{\bar{y}}=(\bar{y}_1, \dots,\bar{y}_m)$ of $C(G)$ such that 
\begin{equation*}
\overline{X} g^{-1}=\boldsymbol{u}(\boldsymbol{\bar{x}})n(\bar{w}) 
\boldsymbol{t}(\boldsymbol{\bar{z}}) \boldsymbol{u}(\boldsymbol{\bar{y}}). 
\end{equation*}
Since $\varphi$ is an $F$-algebra isomorphism there are over $F$ algebraically independent 
elements $\boldsymbol{x}=(x_1, \dots,x_m)$, $\boldsymbol{z}=(z_1, \dots,z_l)$ and 
$\boldsymbol{y}=(y_1, \dots,y_m)$ of $E$ such that 
\begin{equation*}
Y= \varphi(\boldsymbol{u}(\boldsymbol{\bar{x}})n(\bar{w}) 
\boldsymbol{t}(\boldsymbol{\bar{z}}) \boldsymbol{u}(\boldsymbol{\bar{y}}))
=\boldsymbol{u}(\boldsymbol{x})n(\bar{w})\boldsymbol{t}(\boldsymbol{z})
\boldsymbol{u}(\boldsymbol{y}). 
\end{equation*}
\end{proof}

\begin{definition}
 For a full $G$-primitive extension $E/F$ with fundamental matrix $Y \in G(E)$ we call the 
 elements $\boldsymbol{x}=(x_1,\dots, x_m)$, $\boldsymbol{z}=(z_1,\dots, z_l)$ and 
 $\boldsymbol{y}=(y_1,\dots, y_m)$ of Proposition~\ref{general_bruhat of fundamental matrix} the coefficients of the normal form of $Y$. 
\end{definition}

The Borel subgroups, which are by definition maximal solvable subgroups, play 
an important role in the structure theory of reductive groups. In differential Galois
theory Picard-Vessiot extensions with solvable Galois group form a special class of 
extensions. We will see now how these two concepts are linked in our setting.    


\begin{proposition}\label{Liouvillian_part_general}
Let $\boldsymbol{x}$, $\boldsymbol{z}$ and $\boldsymbol{y}$ be the coefficients of a 
normal form of full $G$-primitive extension $E/F$.
Then $L=F(\boldsymbol{x})$ is a differential field with constants $C$. 
Moreover, $E/L$ is a Picard-Vessiot extension with differential Galois group 
$B^-(C)$ and $E$ is generated as a field by $\boldsymbol{z}$ and $\boldsymbol{y}$ over $L$. 
Further $E/L$ is a Liouvillian extension with tower of differential fields
\[
 L(\boldsymbol{z},  y_{1},\dots, y_{m-1})(y_{m})
 \supset \dots \supset L(\boldsymbol{z}) (y_{1}) \supset  
 L(z_1,\dots,z_{l-1})(z_l) \supset \cdots  \supset 
 L(z_1) \supset L  
\]
where the elements $\boldsymbol{z}$ are exponentials and the elements $\boldsymbol{y}$ 
are integrals.
\end{proposition}

\begin{proof}
Applying the Galois correspondence we obtain that $E$ is a 
Picard-Vessiot extension of the fixed field $E^{B^-}$ with differential Galois group $B^-(C)$. 
Since $B^-$ is solvable, the extension $E/E^{B^-}$ is a Liouvillian 
extension by \cite[Theorem 1.43]{P/S}. The field $E$ is generated 
by $\boldsymbol{z}$ and $\boldsymbol{y}$ over $F( \boldsymbol{x} )$, because we have 
$E=F( \boldsymbol{x}, \boldsymbol{z},\boldsymbol{y})$.
We show that $E^{B^-}=F( \boldsymbol{x} )$. This will complete the proof of  
the lemma except for the statement with the tower of fields. 

It follows from Corollary~\ref{cor_decomp_for_matrix_of_G_point_b}, 
applied with $i=0$, that the elements $\boldsymbol{x}$ are left invariant by all $b \in B^-$ 
and so the differential field $E^{B^-}$ contains $F ( \boldsymbol{x} )$.
We prove that the two fields are actually equal by comparing their transcendence degrees.
 Because $G$ is connected and of dimension $2m+l$, the field extension $E$ of $F$ is a  
 transcendental extension of degree $2m+l$. As an intermediate field $E^{B^-}$ is also a transcendental extension of $F$.
 Since $B^-$ is connected and 
 has dimension $l+m$, the extension $E$ of $E^{B^-}$ is a transcendental extension of degree $l+m$.
 We conclude that the transcendence degree of $E^{B^-}$ over $F$ is $m$. 
 Since $\boldsymbol{x}=(x_1, \dots, x_m)$ 
 are algebraically independent over $F$, the transcendence degree of $F( \boldsymbol{x} )$ is 
 also $m$. It follows that $E^{B^-}=F( \boldsymbol{x} )$ and so $F( \boldsymbol{x})$ has the required properties.
 
 It is left to show that there is a tower of differential fields as stated and that the elements 
 $\boldsymbol{z}$ are exponentials and the elements $\boldsymbol{y}$ are integrals. 
Following the notation of Chapter~\ref{The Structure of the Classical Groups} we have a 
descending chain of subgroups as in \eqref{chainofsubgroups}.

For $i=0, \dots, l-2$ the quotients of $\bar{B}_i^-$ by $\bar{B}_{i+1}^-$ and the quotient 
of $\bar{B}_{l-1}^-$ by $\bar{U}_0^-$ are isomorphic to $\mathbb{G}_m$ and so the fixed fields
\[
E^{\bar{B}_{i+1}^-} \supset  E^{\bar{B}_{i}^-} \quad \mathrm{and} \quad  E^{\bar{U}_0^-} \supset E^{\bar{B}_{l-1}^-} 
\]
are Picard-Vessiot extensions of transcendence degree one and they are generated by an exponential. 
From above we already know that $F(\boldsymbol{x}) = E^{\bar{B}_0^-(C)}$ and from 
Corollary~\ref{cor_decomp_for_matrix_of_G_point_b} and \ref{cor_decomp_for_matrix_of_G_point_c}, applied with $i=0$, we obtain for $i=1, \dots, l-1$ that 
\[
F(\boldsymbol{x}, z_1, \dots , z_i) \subseteq E^{\bar{B}_i^-(C)}
\quad \mathrm{and} \quad 
F(\boldsymbol{x}, z_1, \dots , z_l) \subseteq E^{\bar{U}_0^-(C)}
\]
One proves now inductively that the inclusions are actually equalities by comparing the transcendence 
degrees of the corresponding extensions. Summing up we have that for $i=1,\dots,l$ the extensions
$F(\boldsymbol{x}, z_1, \dots , z_{i-1}) \subset F(\boldsymbol{x}, z_1, \dots , z_{i})$ are 
Liouvillian and $z_i$ is an exponential.

For $i=0,\dots m-1$ we have that the quotient of $\bar{U}_i^-$ by $\bar{U}_{i+1}^-$ is isomorphic to $\mathbb{G}_a$. 
Thus the fixed fields 
\[
E^{\bar{U}_{i+1}^-} \supset E^{\bar{U}_{i}^-}
\]
define Picard-Vessiot extensions of transcendence degree one which are obtained by adjoining an integral.
From above we know that $F(\boldsymbol{x}, z_1, \dots , z_l) = E^{\bar{U}_0^-(C)}$ and 
trivially $F(\boldsymbol{x}, \boldsymbol{z}, y_1, \dots , y_m) = E^{\bar{U}_m^-(C)}$. Moreover, 
Corollary~\ref{cor_decomp_for_matrix_of_G_point_c} 
yields for $i=1,\dots , m-1$ that 
\[
F(\boldsymbol{x}, \boldsymbol{z}, y_1, \dots , y_i) \subseteq E^{\bar{U}_i^-(C)}.
\]
As above if we compare the transcendence degrees of the corresponding extensions we can show 
that the inclusions are actually equalities. We conclude that for $i=1,\dots,m$ the extension 
$F(\boldsymbol{x}, \boldsymbol{z}, y_1, \dots , y_{i})$ of $F(\boldsymbol{x}, \boldsymbol{z}, y_1, \dots , y_{i-1})$ 
is Liouvillian and that $y_i$ is an integral.
\end{proof}

\begin{remark}\label{remark_eq_for_Liouv}
In the notation of Proposition~\ref{Liouvillian_part_general} the matrix 
$\ell \delta (Y) \in \mathfrak{g}(F)$ is a defining matrix for the 
Picard-Vessiot extension $E$ of $L$. 
It is actually very easy to produce a defining matrix for this extension 
which lies in the Lie algebra $\mathfrak{b}^-(L)$. 
Indeed, the inverse of the matrix $\boldsymbol{u}(\boldsymbol{x})n(\bar{w})$ 
is an element of $G(L)$ and so the gauge transformation of $\ell \delta (Y)$ 
with this matrix is defined over $L$ and gives an equivalent defining matrix for the 
extension. 
The gauge equivalent defining matrix is the 
logarithmic derivative of 
\begin{equation*}
(\boldsymbol{u}(\boldsymbol{x})n(\bar{w}))^{-1} Y 
=  \boldsymbol{t}(\boldsymbol{z})\boldsymbol{u}(\boldsymbol{y}) 
\end{equation*}
 which lies in $\mathfrak{b}^- (L)$ by Proposition~\ref{log_derivate}.
\end{remark}

\section{The construction of the general extension field}
\label{The construction of the general extension field}
In this chapter we construct our  general extension field $E$. 
It will be a Picard-Vessiot extension with differential Galois group $G$ of
a differential field which has differential transcendence degree $l$ 
over $C$. The differential structure of $E$ will be associated to a 
defining matrix of specific form. In \cite{Seiss} we 
constructed a matrix differential equation depending on $l$ parameters 
$\boldsymbol{t}=(t_1, \dots, t_l)$ which defines a Picard-Vessiot extension 
of the differential field generated over $C$ by the differential indeterminates $\boldsymbol{t}$
with Galois group $G(C)$. More precisely, for $l$ 
complementary roots $\gamma_1, \dots,\gamma_l$ of $\Phi^-$ 
and for $A_0^+$ the equation is defined by the matrix 
 \begin{equation*}
  A_G(\boldsymbol{t})=A_0^+ + \sum_{i=1}^l t_i X_{\gamma_i} .
 \end{equation*}
The construction of our general extension field will be based on the differential field 
$C\langle \boldsymbol{\eta} \rangle$ where $\boldsymbol{\eta}=(\eta_1,\dots,\eta_l)$
are $l$ differential indeterminates over $C$. It will be a Liouvillian extension of 
$C\langle \boldsymbol{\eta} \rangle$ whose  differential Galois group will be $B^-(C)$. 
The defining matrix of the Liouvillian extension will be 
the sum of a parametrization  of the Cartan subalgebra by $\boldsymbol{\eta}$ and 
$A_0^-$. It will turn out that a fundamental matrix of this extension is 
$\boldsymbol{t}(\boldsymbol{z}) \boldsymbol{u}(\boldsymbol{y})$ where 
$\boldsymbol{z}=(z_1, \dots, z_l)$ are exponentials and 
$\boldsymbol{y}=(y_1, \dots, y_l)$ are integrals. 
Using the root structure we will show that there are $m-l$ differential polynomials 
\[
f_{l+1}(\boldsymbol{\eta}), \dots, f_{m}(\boldsymbol{\eta}) 
\]
of $C\{\boldsymbol{\eta} \}$ with the property that the logarithmic derivative of 
\begin{equation*}
\boldsymbol{u}\big(\eta_1, \dots, \eta_l,
f_{l+1}(\boldsymbol{\eta}), \dots, f_{m}(\boldsymbol{\eta} ) \big) n(\bar{w}) 
\boldsymbol{t}(\boldsymbol{z}) \boldsymbol{u}(\boldsymbol{y})
\end{equation*}
is $A_G(\boldsymbol{h})$ where 
$\boldsymbol{h}=(h_1(\boldsymbol{\eta}) , \dots, h_l(\boldsymbol{\eta}))$ are  
elements of $C\{ \boldsymbol{\eta} \}$.
Finally we will prove that $\boldsymbol{h}$ are differentially algebraically 
independent over $C$ and that $E$ is a Picard-Vessiot extension of 
$C\langle \boldsymbol{h} \rangle$ for $A_G(\boldsymbol{h})$ 
with differential Galois group $G(C)$. 
The degree of freedom of our general extension field $E$ lies in the $l$ 
differential indeterminates $\boldsymbol{\eta}=(\eta_1,\dots,\eta_l)$.

 For technical reasons we start the construction of $E$ with a tuple of $m$
 differential indeterminates $\boldsymbol{\eta}_m =(\eta_1, \dots, \eta_m)$
 which we will reduce later to the tuple of the first $l$ indeterminates 
 $\boldsymbol{\eta} =(\eta_1, \dots, \eta_l)$. To shorten notation 
 we will in the following denote by 
 $\boldsymbol{u}_i(\boldsymbol{\eta}_m)$ the product 
 $u_1(\eta_1) \cdots u_i(\eta_i)$  for $i=2, \dots m-1$.

\begin{lemma} 
\label{relations_coefficients1} 
The image of $\boldsymbol{u}(\boldsymbol{\eta}_m)$ 
under the logarithmic derivative is 
 \begin{equation*}
 \ell\delta( \boldsymbol{u}(\boldsymbol{\eta}_m))= \sum_{i=1}^l \eta_i'  X_{i} +
 \sum_{i=l+1}^m (\eta_i' + v_i(\boldsymbol{\eta}_m)) X_{i} 
 \end{equation*}
 where $v_i(\boldsymbol{\eta}_m)$ lies in 
 $C\{ \eta_1 , \dots,\eta_{s_2} \}$ 
 with $s_2$ maximal such that $r(s_2) =r(i)+1$ and all its terms are of
 order one and of degree greater than one.
 \end{lemma}

 \begin{proof}
With the product rule and the 
definition of the adjoint action 
we obtain for the image of $\boldsymbol{u}(\boldsymbol{\eta}_m)$
under the logarithmic derivative the sum 
\[
\begin{gathered}
\ell\delta( \boldsymbol{u}(\boldsymbol{\eta}_m))  = 
\ell\delta(u_{1}(\eta_1))  + \\
\mathrm{Ad}(u_{1}(\eta_1)) ( \ell \delta(u_{2}(\eta_2))) + \dots +    
\mathrm{Ad}( \boldsymbol{u}_{m-1}(\boldsymbol{\eta}_m))(\ell\delta(u_{m}(\eta_m))).
\end{gathered}
\]
We are going to determine the logarithmic derivatives $\ell \delta(u_{i}(\eta_i ))$ 
for $i=1,\dots,m$ and the images of $\ell \delta (u_{k+1}(\eta_{k+1} )) $ under 
$\mathrm{Ad}( \boldsymbol{u}_{k}(\boldsymbol{\eta}_m))$ for $k=1,\dots,m-1$. 
Using the results the statement of the lemma will follow from the last equation. 
 
The logarithmic derivative of $u_i(\eta_i)$ is $\eta_i' X_i$ for $i=1,\dots,m$.
Indeed, Proposition~\ref{log_derivate} yields that $\ell \delta (u_i(\eta_i))$ 
lies in the root space $\mathfrak{g}_i$. 
Moreover, the root group element $u_i(\eta_i)$ is the image of $\eta_i X_i $  under the exponential map where $X_i$ 
nilpotent. Using this representation of $u_i(\eta_i)$ 
one easily checks that the only contribution to $X_i$ in the corresponding
product of power series is $\eta_i' X_i$.

For $j = 2, \dots, m$ let $X$ be a linear combination of basis elements 
$X_i$ with $r(i) \leq r(j)$ whose the coefficients lie in $C\{\eta_1, \dots, \eta_{j_2}\}$ with $j_2$ maximal such that 
$r(j_2)=r(j)$. We claim that for $k = 1, \dots, j-1$
the image of $X$ under the adjoint action of 
$u_k(\eta_k)$ can be represented as a linear combination of the same basis elements and that the coefficients of those $X_i$ with $r(i)=r(j)$ 
are the same as in $X$ and of those with $r(i)<r(j)$ lie again in $C\{\eta_1, \dots, \eta_{j_2}\}$. Let  $i$ be an index with $r(i) \leq r(j)$
and let $v(\boldsymbol{\eta}_m)$ be a nonzero element of 
$C\{\eta_1, \dots, \eta_{j_2}\}$. 
From the first formula in Remark~\ref{remark_adjoint_action} we obtain
\begin{equation}\label{eq:relations_coefficients_eq2}
\mathrm{Ad}(u_k(\eta_k))( v(\boldsymbol{\eta}_m) X_i )
= \sum_{s\geq 0}  c_{\beta_k,\beta_i,s} \ \eta_k^s \ v(\boldsymbol{\eta}_m) \  X_{\beta_i + s \beta_k}
\end{equation}
where the coefficient 
$c_{\beta_k,\beta_i,s}$ is nonzero
if $\beta_i + s \beta_k$ is a root and zero otherwise.
Clearly if $\beta_i + s \beta_k$ is a root, then it is equal to $\beta_i$ (case $s=0$) or it is of height less than $\beta_i$ (case $s \geq 1$). Thus the right hand side of \eqref{eq:relations_coefficients_eq2} is a linear combination of 
basis elements corresponding to roots of height less or equal than $r(i)$. For $s=0$ the coefficient of $X_i$ is $v(\boldsymbol{\eta}_m)$ and 
because the index $k$ is less than $j$,  
the coefficients 
in the sum of the right hand side of \eqref{eq:relations_coefficients_eq2}
lie in $C\{\eta_1, \dots, \eta_{j_2}\}$ for $s \geq 1$. The claim now follows from the 
linearity of the adjoint action.

For $k=1, \dots,m-1 $ we apply the claim iteratively to 
\begin{equation*} 
\mathrm{Ad}( \boldsymbol{u}_{k}(\boldsymbol{\eta}_m))(\ell \delta (u_{k+1}(\eta_{k+1}))
=\mathrm{Ad}(u_1(\eta_1))( \dots ( \mathrm{Ad}(u_k(\eta_k))(\eta_{k+1}' X_{k+1})) \dots ).
\end{equation*}
Representing the result as a linear combination of $X_1,\dots,X_m$
we obtain that the coefficients of $X_i$ with $r(i) \geq r(k+1)$
 and $i \neq k+1$ are zero, that the coefficient of   
$X_{k+1}$ is $\eta_{k+1}'$ and that the coefficients of $X_i$ with $r(i)<r(k+1)$ are in  $C\{ \eta_1, \dots, \eta_{s}\}$ where $s$ is maximal such that $r(s)=r(k+1)$. 
In case of the latter coefficients we can adapt the definition of the ring 
containing them such that it becomes independent of $k$. Let $s_2$ be maximal such that 
$r(s_2)=r(i)+1$. For basis elements with $r(i)+1= r(k+1)$ the ring 
$C\{ \eta_1, \dots, \eta_{s_{2}}\}$ coincides with $C\{ \eta_1, \dots, \eta_{s}\}$ and for basis elements with $r(i)+1< r(k+1)$ the ring 
$C\{ \eta_1, \dots, \eta_{s_{2}}\}$ contains the ring $C\{ \eta_1, \dots, \eta_{s}\}$.
\end{proof}

\begin{example}\label{ex:SL4logderivative}
We consider the group $\mathrm{SL}_4(C)$.
In \cite[Chapter 7]{Seiss} we worked out the 
root system of type $A_3$ and presented a corresponding explicit Chevalley basis of the Lie algebra $\mathfrak{sl}_4(C)$. Our computations 
will be with respect to this basis. We denote and number the negative roots  
according to Chapter~\ref{The Structure of the Classical Groups} as 
\begin{gather*}
\beta_1 = - \bar{\alpha}_1 , \ \beta_2 = - \bar{\alpha}_2 , \ \beta_3 = - \bar{\alpha}_3 , \ 
\beta_4 = - \bar{\alpha}_1 - \bar{\alpha}_2, \ \beta_5 = - \bar{\alpha}_2 - \bar{\alpha}_3 , \\   
\beta_6 = - \bar{\alpha}_1 - \bar{\alpha}_2 - \bar{\alpha}_3
\end{gather*}
so that $r(i)=-1$ for $i=1,2,3$, $r(i)=-2$ for $i=4,5$ and $r(i)=-3$ for $i=6$.
We write for the basis element $X_{\beta_i}$ and for a root 
group element $u_{\beta_i}$ shortly $X_i$ and $u_i$ respectively.
From the exponential map we obtain that for a 
 negative root $\beta_i$ the element $u_i(\eta_i)$ of the root group $U_i$ is the matrix 
 $ u_i(\eta_i)  = E + \eta_i X_i$ where $E$ denotes 
 the $4\times 4$ unit matrix. It follows that $\ell \delta (u_i(\eta_i))  =  \eta_i' X_i$.
We have 
\begin{gather*}
     \ell \delta (\boldsymbol{u}(\boldsymbol{\eta}_6)) = \eta_1' X_1 +
     \mathrm{Ad}(u_1(\eta_1))(\eta_2' X_2) + 
     \mathrm{Ad}(u_1(\eta_1)u_2(\eta_2))(\eta_3' X_3) + \\
     \mathrm{Ad}(u_1(\eta_1)\cdots u_3(\eta_3))(\eta_4' X_4) +
     \mathrm{Ad}(u_1(\eta_1)\cdots u_4(\eta_4))(\eta_5' X_5) + \\
     \mathrm{Ad}(u_1(\eta_1)\cdots u_5(\eta_5))(\eta_6' X_6)  .
\end{gather*}
We determine the terms in the coefficients of the linear representation of the 
logarithmic derivative of $\boldsymbol{u}(\boldsymbol{\eta}_6)$ with respect to the 
basis $\{ X_i \mid i=1,\dots,6 \}$. Obviously for all
$i=1,\dots,6$ the coefficient of $X_i$ contains the term $\eta_i'$. Since 
$\beta_2 + \beta_1=\beta_4$ the second summand contributes the term $\eta_1\eta_2'$
to the coefficient of $X_4$. From the third summand the coefficients of 
$X_5$ and $X_6$ obtain the terms $\eta_2\eta_3'$ and $\eta_1 \eta_2\eta_3'$ respectively, 
since $\beta_3 + \beta_2= \beta_5$ and $\beta_5 + \beta_1=\beta_6$. The fourth summand 
contributes the term $\eta_3 \eta_4'$ to the coefficient of $X_6$, since 
$\beta_4 +\beta_3=\beta_6$. From the next summand the coefficient of $X_6$ obtains the 
term $\eta_1\eta_5'$, since $\beta_5 +\beta_1=\beta_6$. Finally the last summand 
gives no contribution, because the sum of $\beta_6$ and any negative root is not a root. 
Putting our results together we obtain 
  \begin{gather*}
 \ell \delta (\boldsymbol{u}(\boldsymbol{\eta}_6)) = 
  \eta_1' X_1 + \eta_2' X_2 + \eta_3' X_3  \\ +  
 ( \eta_4' + v_4(\boldsymbol{\eta}_6) ) X_4 + (\eta_5' + v_5(\boldsymbol{\eta}_6)) X_5 +  
  (\eta_6' + v_6(\boldsymbol{\eta}_6)) X_6 
 \end{gather*}
 with 
 \[
 v_4(\boldsymbol{\eta}_6)= - \eta_2' \eta_1, \ v_5(\boldsymbol{\eta}_6)=-\eta_3' \eta_2, \  v_6(\boldsymbol{\eta}_6)= \eta_3 \eta_4' - \eta_5' \eta_1 + \eta_3' \eta_2 \eta_1 
 \]
 and one easily checks that the coefficients are as in the statement of the lemma.
\end{example}

\begin{lemma}\label{relations_coefficients2}
We have  
 \begin{equation*}
  \mathrm{Ad}(\boldsymbol{u}(\boldsymbol{\eta}_m))(A_0^+)=  
  A_0^+ + \sum_{i=1}^l g_i(\boldsymbol{\eta}) H_i + \sum_{i=1}^m (\ell_i(\boldsymbol{\eta}_m)  + p_i(\boldsymbol{\eta}_m) ) X_{i}
 \end{equation*}
 where 
 \begin{enumerate}[label=(\alph*),ref=\thetheorem(\alph*)]
  \item \label{rel_coef2_point_a} $g_1(\boldsymbol{\eta}),\dots, g_l(\boldsymbol{\eta})  \in C[\eta_1, \dots,  \eta_l]$ are 
  nonzero $C$-linear independent homogeneous polynomials of degree one.
  \item \label{rel_coef2_point_b} $p_i(\boldsymbol{\eta}_m) \in C[\eta_1 ,\dots, \eta_{i_2}]$ with $i_2$  maximal such 
  that $r(i_2)=r(i)$ and each term of $p_i(\boldsymbol{\eta}_m)$ is of degree greater than one.
  \item \label{rel_coef2_point_c} $\ell_i(\boldsymbol{\eta}_m) \in C[\eta_{k_1}, \dots , \eta_{k_2}]$ 
  with $k_1$ minimal and $k_2$ maximal such that $r(k_1)=r(i)-1=r(k_2)$
  and $\ell_i(\boldsymbol{\eta}_m)$ is a homogeneous polynomial of degree one. 
  Let $i_1$ be minimal and 
  $i_2'$ maximal such that $r(i_1)=r(i)=r(i_2')$ and $i_2'$ does not 
  correspond to a complementary root. Then 
  \[
  \ell_{i_1}(\boldsymbol{\eta}_m)=0, \dots, \ell_{i_2'}(\boldsymbol{\eta}_m)=0 
  \]
  is a quadratic linear system in $\eta_{k_1}, \dots , \eta_{k_2}$ of full rank.
\end{enumerate}
\end{lemma}

\begin{proof}
For $j = 1, \dots , m $ we prove by induction that 
\begin{equation*}
\mathrm{Ad}(\boldsymbol{u}_j(\boldsymbol{\eta}_m) )(A_0^+)=A_0^+ + \sum_{i=1}^j \eta_i W_i + 
\sum_{i=1}^m \bar{p}_i(\boldsymbol{\eta}_m)  X_i
\end{equation*}
where $\bar{p}_i(\boldsymbol{\eta}_m)$ is zero or $\bar{p}_i(\boldsymbol{\eta}_m) \in C[\eta_1, \dots , \eta_{i_2}]$ with $i_2$ maximal such that $r(i_2)=r(i)$ and each term is 
of degree greater than one.

Since by definition $W_j= [X_j,A_0^+]$ we have 
$[\eta_j X_j,A_0^+]=\eta_j W_j$ and from
Remark~\ref{remark_adjoint_action} we obtain  
\begin{equation*}
 \mathrm{Ad}(u_{j}(\eta_j))(A_0^+) = A_0^+ + \eta_j W_j + \sum_{k=1}^l \sum_{s \geq 2} c_{\beta_j, \bar{\alpha}_k,s} \ \eta_j^s \ 
 X_{\bar{\alpha}_k + s \beta_j } 
\end{equation*}
where the coefficients of the double sum are either zero or trivially a
homogeneous polynomial of degree greater than one in $\eta_j$ depending 
whether $\bar{\alpha}_k + s \beta_j$ is a root or not. 

For $j=1$ the above observations yield  
\begin{equation*}
\mathrm{Ad}(u_1(\eta_1)  )(A_0^+)=A_0^+ + \eta_1 W_1 + \bar{p}_1(\boldsymbol{\eta}_m) X_1
\end{equation*}
where $ \bar{p}_1(\boldsymbol{\eta}_m) \in C[\eta_1]$ is a homogeneous polynomial 
of degree two.  

For $j > 1$ we obtain with the same observations that   
\begin{equation*}
\mathrm{Ad}(\boldsymbol{u}_j(\boldsymbol{\eta}_m) )(A_0^+) =
\mathrm{Ad}(\boldsymbol{u}_{j-1}(\boldsymbol{\eta}_m)) 
(A_0^+ + \eta_j W_j  +
 \sum_{i=j_1}^m  \bar{p}_i(\boldsymbol{\eta}_m) X_i)
\end{equation*}
where we take $j_1$ minimal such that $r(j_1)=r(j)$ and the polynomials  
$\bar{p}_i(\boldsymbol{\eta}_m)$ in $C[\eta_j]$ are either zero or 
homogeneous of degree greater than one.
The polynomial $ \bar{p}_i(\boldsymbol{\eta}_m)$ is the coefficient of the 
basis element which corresponds to the root $\beta_i =\bar{\alpha}_k + s \beta_j $ with $s \geq 2$ 
and so for $i_2$ maximal such that $r(i_2)=r(i)$
we have $ \bar{p}_i(\boldsymbol{\eta}_m) \in C[\eta_1, \dots, \eta_{i_2}]$. 
Since the adjoint action is linear we can consider the three images
on the right hand side of the last equation
individually and then combine our results. For the image 
of $A_0^+$ the induction assumption implies 
\begin{equation*}
\mathrm{Ad}(\boldsymbol{u}_{j-1}(\boldsymbol{\eta}_m ))(A_0^+) = A_0^+ + \sum_{i=1}^{j-1} \eta_i W_i + 
\sum_{i=1}^m \bar{p}_i(\boldsymbol{\eta}_m) X_i
\end{equation*}
where $\bar{p}_i(\boldsymbol{\eta}_m)$ is zero or $\bar{p}_i(\boldsymbol{\eta}_m) \in C[\eta_1, \dots , \eta_{i_2}]$ with $i_2$ maximal such that $r(i_2)=r(i)$ and each term of it is 
of degree greater than one. 

The vector $\eta_j W_j $ is contained in the sum of root spaces
corresponding to roots of height $r(j)+1$, that is in $\mathrm{g}^{(r(j)+1)}$, and  
$\eta_j$ lies in $ C [\eta_1, \dots, \eta_{j_2}]$ with $j_2$ maximal such that $r(j_2)=r(j)$. 
The image of $\eta_j W_j $ computes with the above observations iteratively as
\begin{equation*}
\mathrm{Ad}(\boldsymbol{u}_{j-1}(\boldsymbol{\eta}_m ))(\eta_j W_j ) = \eta_j W_j + 
\sum_{i=j_1}^m \bar{p}_i(\boldsymbol{\eta}_m) X_i 
\end{equation*}
where $j_1$ is minimal such that $r(j_1)=r(j)$. Since the entries of 
$\boldsymbol{u}_{j-1}(\boldsymbol{\eta}_m )$ are in 
$C[\eta_1, \dots, \eta_{j-1}]$ and 
$\eta_j \in C [\eta_1, \dots, \eta_{j_2}]$ with $j_2$ maximal such that 
$r(j_2)=r(j)$, we conclude that all $\bar{p}_i(\boldsymbol{\eta}_m)$ are 
contained in $C[\eta_1, \dots, \eta_{j_2}]$. Because in each 
iteration step $s\geq 1$ in the above formula and because we apply the adjoint action 
to basis elements with polynomial coefficients whose terms are of degree at least one, the 
degree of each term in $\bar{p}_i(\boldsymbol{\eta}_m)$ is greater than one. 
For $i \geq j_1$ we have that $i_2$ 
maximal such that $r(i_2)=r(i)$ satisfies $i_2 \geq j_2$ and so $\bar{p}_i$ 
also lies in the eventually larger ring $C[\eta_1, \dots, \eta_{i_2}]$. 

For the last image the above observations yield 
\begin{equation*}
 \mathrm{Ad}(\boldsymbol{u}_{j-1}(\boldsymbol{\eta}_m ))
 (\sum_{i=j_1}^m \bar{p}_i(\boldsymbol{\eta}_m) X_i)
 = \sum_{i=j_1}^m \bar{p}_i(\boldsymbol{\eta}_m)  X_i + 
 \sum_{i=j_1'}^m \hat{p}_{i}(\boldsymbol{\eta}_m) X_i 
\end{equation*}
where $j_1'$ is minimal such that $r(j_1')=r(j_1)-1$. 
For each $i \geq j_1'$ the index $i_2$ maximal such that $r(i_2)=r(i)$ satisfies $i_2 > j$, 
because $j_1$ is minimal such that $r(j_1)=r(j)$. Since the entries of 
$\boldsymbol{u}_{j-1}(\boldsymbol{\eta}_m )$ are elements of $C[\eta_1, \dots, \eta_{j-1}]$ 
and $\bar{p}_i(\boldsymbol{\eta}_m)$ lies in $C[\eta_j]$ and is of degree greater than one, 
we conclude that $\hat{p}_{i}(\boldsymbol{\eta}_m)$ is a polynomial of $C[\eta_1, \dots, \eta_{i_2}]$ 
where each term is of degree greater than one. Combining the three results completes the induction.
 
The claim yields for $j=m$ the equation 
\begin{equation}\label{eq:relations_coefficients3}
\mathrm{Ad}(\boldsymbol{u} (\boldsymbol{\eta}_m) )(A_0^+)=A_0^+ + \sum_{i=1}^m \eta_i W_i + 
\sum_{i=1}^m \bar{p}_i(\boldsymbol{\eta}_m) X_i
\end{equation}
where $\bar{p}_i(\boldsymbol{\eta}_m) \in C[\eta_1, \dots , \eta_{i_2}]$ with 
$i_2$ maximal such that $r(i_2)=r(i)$ is a polynomial whose terms are 
of degree greater than one.
Since $\{ W_i \mid r(i)=-1 \}$ is a basis of $\mathfrak{h}$, we obtain
that the coefficients of $H_1,\dots, H_l$ 
are nonzero $C$-linearly independent homogeneous polynomials 
of degree one which we denote by $g_i(\boldsymbol{\eta})$. 
Let $\beta_i$ be a non complementary root and let $k_1$ be minimal and 
$k_2$ be maximal such that $r(k_1)=r(i)-1=r(k_2)$. Since $\beta_i$ and 
at least one of the roots $\beta_{k_1}, \dots, \beta_{k_2}$ differ by a simple 
root, the coefficient of $X_i$ in \eqref{eq:relations_coefficients3}
has a non zero homogeneous part of degree one in 
$C[\eta_{k_1},\dots, \eta_{k_2} ]$  stemming from $W_{k_1},\dots, W_{k_2}$. 
We denote this part by $\ell_i(\boldsymbol{\eta}_m)$ and the remaining part of degree greater than one 
by $p_i(\boldsymbol{\eta}_m)$. We prove that 
\[
\ell_{i_1}(\boldsymbol{\eta}_m)=0, \dots, \ell_{i_2'}(\boldsymbol{\eta}_m)=0 
\]
with $i_1$ minimal and $i_2'$ maximal such that 
$r(i_1)=r(i)=r(i_2')$ and  $i_2'$ does not correspond to a complementary root 
form a linear system in $\eta_{k_1},\dots,\eta_{k_2}$ of full rank.
Indeed, the set 
$\{ W_{k_1},\dots, W_{k_2} \} \cup \{ X_{\gamma_k} \mid r(k)=r(i) \}$
is a basis of $\mathfrak{g}^{(r(i))}$ and so the coefficient matrix for 
these vectors with respect to the basis $X_{i_1},\dots,X_{i_2}$ with 
$i_2$ maximal such that $r(i)=r(i_2)$ has full rank. Since the two basis 
coincide for the indices $k_2 +1, \dots, i_2$, it is a block matrix 
whose last row consists of a corresponding zero and unit matrix. 
We conclude that the
first block, that is, the matrix without the rows and columns for the complementary roots has also full rank. 
Identifying $ W_{k_1},\dots, W_{k_2}$ with $\eta_{k_1},\dots, \eta_{k_2}$ shows that 
\[
\ell_{i_1}(\boldsymbol{\eta}_m)=0, \dots, \ell_{i_2'}(\boldsymbol{\eta}_m)=0 
\]
is a linear system in $\eta_{k_1},\dots,\eta_{k_2}$ of full rank.
\end{proof}

\begin{example}\label{ex:SL4Adjoint} 
We proceed with Example \ref{ex:SL4logderivative}. According to Lemma~\ref{relations_coefficients2} the image
of the matrix $A_0^+$
under the adjoint action with $\boldsymbol{u}(\boldsymbol{\eta}_6)$ can be written as 
\begin{gather*}
   \mathrm{Ad}(\boldsymbol{u}(\boldsymbol{\eta}_6))(A_0^+) = 
   A_0^+ + \sum_{i=1}^3 g_i(\boldsymbol{\eta}) H_i + \sum_{i=1}^6     (\ell_i(\boldsymbol{\eta}_6)+p_i(\boldsymbol{\eta}_6) ) X_i .
\end{gather*}
We determine the differential polynomials $g_i(\boldsymbol{\eta})$, $\ell_i(\boldsymbol{\eta}_6)$ and 
$p_i(\boldsymbol{\eta}_6) )$ and check that they have the described properties.
We have
\begin{eqnarray*}
\mathrm{Ad}(u_6(\eta_6))(A_0^+) &=& A_0^+ + \eta_6 W_6 =  A_0^+ - \eta_6 X_4 + \eta_6 X_5  \\
\mathrm{Ad}(u_5(\eta_5))(A_0^+) &=& A_0^+ + \eta_5 W_5 =  A_0^+ - \eta_5 X_2 + \eta_5 X_3 \\
\mathrm{Ad}(u_4(\eta_4))(A_0^+) &=& A_0^+ + \eta_4 W_4 =  A_0^+ - \eta_4 X_1 + \eta_4 X_2
\end{eqnarray*}
Indeed, there are no nonlinear parts, since it is not possible to obtain a root by adding any simple root to $2 \beta_i$ for $i=6,5,4$. The components of $W_i$ in the basis 
$\{ X_i \mid i=1,\dots,m \}$ correspond to those roots which are the sum of $\beta_i$ and 
a simple root. Further with Remark~\ref{remark_adjoint_action} we get 
\begin{eqnarray*}
\mathrm{Ad}(u_3(\eta_3))(A_0^+) &=& A_0^+ + \eta_3 W_3 - \eta_3^2 X_3  \\
\mathrm{Ad}(u_2(\eta_2))(A_0^+) &=& A_0^+ + \eta_2 W_2 - \eta_2^2 X_2 \\
\mathrm{Ad}(u_1(\eta_1))(A_0^+) &=& A_0^+ + \eta_1 W_1 - \eta_1^2 X_1
\end{eqnarray*}
where $W_i=-H_i$ for $i=3,2,1$. The $W_j$ with $j=1,\dots,6$ determine the terms in the  
homogeneous polynomials of degree one $g_i(\boldsymbol{\eta})$ and 
$\ell_i(\boldsymbol{\eta}_6)$. We have  
\begin{gather*}
g_1(\boldsymbol{\eta}) = -\eta_1, \ g_2(\boldsymbol{\eta})=-\eta_2 , \ g_3(\boldsymbol{\eta})=-\eta_3 
\end{gather*}
and they clearly satisfy Lemma~\ref{rel_coef2_point_a}. Further we read off 
\begin{gather*}
\ell_1(\boldsymbol{\eta}_6) = -\eta_4 , \ 
\ell_2(\boldsymbol{\eta}_6) = \eta_4 -\eta_5, \  
\ell_3(\boldsymbol{\eta}_6) = \eta_5, \ \\
\ell_4(\boldsymbol{\eta}_6) = -\eta_6, \ 
\ell_5(\boldsymbol{\eta}_6) = \eta_6 , \ 
\ell_6(\boldsymbol{\eta}_6)=0    
\end{gather*}
and they obviously satisfy the first part of statement \ref{rel_coef2_point_c}.
Since the complementary roots are 
$\beta_3$, $\beta_5$ and $\beta_6$ (see \cite[Lemma 7.1]{Seiss}), we have the two linear systems 
\[
\ell_1(\boldsymbol{\eta}_6) =0, \ \ell_2(\boldsymbol{\eta}_6)=0 \quad \mathrm{and} \quad \ell_4(\boldsymbol{\eta}_6)=0
\]
in the variables $\eta_4$, $\eta_5$ and 
in the variable $\eta_6$ respectively and both systems have full rank. 
In order to determine the non-linear parts $p_i(\boldsymbol{\eta}_6) $ 
one can use the recursion 
\begin{eqnarray*}
   r_0  &:=& \mathrm{Ad}(u_6(\eta_6))(A_0^+) \\
 r_{k} &:=& \mathrm{Ad}(u_{6-k}(\eta_{6-k}))(r_{k-1}) \ \mathrm{for} \ k=1,\dots,5
\end{eqnarray*}
which computes the whole image of $A_0^+$ under 
$\mathrm{Ad}(\boldsymbol{u}(\boldsymbol{\eta}_6))$. 
We sketch how to proceed. We already know that in the steps $k=3,4,5$ we 
obtain from $A_0^+$ in $r_{k-1}$ the term $\eta_{6-k}^2$ in the coefficient 
of $X_{6-k}$ in the linear representation of $r_k$. We get further nonlinear terms 
from the component in the 
Cartan subalgebra. Since it is only nonzero in the steps $4$ and $5$, we obtain 
the new nonlinear terms $\eta_2\eta_3$ and $\eta_2\eta_1$ in the coefficients of 
$X_2$ and $X_1$ in the linear representation of $r_4$ and $r_5$ respectively.
Finally in each step $k$ one obtains new 
nonlinear terms in the coefficient of $X_j$ in the representation of $r_k$ 
from multiplying the coefficient of
that $X_i$ in the representation of $r_{k-1}$ with $\eta_{6-k}$ for which $\beta_i + \beta_{6-k}$ is the root $\beta_j$.
One can check that the nonlinear parts compute as 
\begin{gather*}
p_1(\boldsymbol{\eta}_6) = -\eta_1^2 +\eta_2 \eta_1 , \  
p_2(\boldsymbol{\eta}_6) = -\eta_2^2+\eta_3 \eta_2 , \ 
p_3(\boldsymbol{\eta}_6) = -\eta_3^2  , \\
p_4(\boldsymbol{\eta}_6) = -\eta_2 \eta_4 + \eta_1 (\eta_2^2 -\eta_4  -\eta_3  \eta_2+\eta_5 ) , \ 
p_5(\boldsymbol{\eta}_6) = -\eta_5 \eta_2+\eta_3 ( \eta_4- \eta_5+\eta_2 \eta_3 ) , \\
p_6(\boldsymbol{\eta}_6) = -\eta_1 \eta_3 \eta_4-\eta_1 \eta_6-\eta_5 \eta_4+\eta_5 \eta_2 \eta_1-\eta_3^2   
    \eta_2 \eta_1+\eta_3 \eta_5 \eta_1-\eta_3 \eta_6 
\end{gather*}
and that they have the stated properties of Lemma~\ref{rel_coef2_point_b}.
\end{example}

In the next step we construct the Liouvillian extension of $C\langle \boldsymbol{\eta} \rangle$. For this purpose we fix a representative 
$n(\bar{w})$ in the normalizer of the torus for the Weyl group element $\bar{w}$ of maximal length. Since $\bar{w}$ sends $\{- \bar{\alpha}_1, \dots, -\bar{\alpha}_l \}$ bijectively to $\Delta$ and
the adjoint action of $n(\bar{w})$ maps $X_1, \dots, X_l $ bijectively to non zero multiples of
$X_{\bar{\alpha}_1}, \dots, X_{\bar{\alpha}_l}$, there are units 
$\boldsymbol{c}=(c_1,\dots,c_l)$ in $C$ such that 
\begin{equation*}
 \mathrm{Ad}(n(\bar{w}))(A_0^-(\boldsymbol{c})) =A_0^+. 
\end{equation*}
Further the adjoint action of $n(\bar{w})$ sends each basis element $H_i$ 
of the Cartan subalgebra to a nonzero multiple. Thus for $g_i(\boldsymbol{\eta})$ of Lemma~\ref{rel_coef2_point_a} there are nonzero $C$-linear 
independent homogeneous polynomials $\bar{g}_1(\boldsymbol{\eta}) ,\dots, \bar{g}_l(\boldsymbol{\eta})$ 
of degree one in $C[\boldsymbol{\eta}]$ such that 
\begin{equation*}
 \mathrm{Ad}(n(\bar{w}))( \sum_{i=1}^l \bar{g}_i(\boldsymbol{\eta}) H_i ) = 
 \sum_{i=1}^l -g_i(\boldsymbol{\eta}) H_i. 
\end{equation*}
With these definitions we set 
\begin{equation*}
 A_L(\boldsymbol{\eta})= \sum_{i=1}^l  \bar{g}_i(\boldsymbol{\eta}) H_i +A_0^-(\boldsymbol{c}) 
 \in \mathfrak{b}^-(C\langle \boldsymbol{\eta} \rangle)
\end{equation*}
and consider the matrix differential 
equation defined by $A_L(\boldsymbol{\eta})$.

\begin{proposition}\label{generic_liouville}
There is a Picard-Vessiot extension $E$ of $C \langle \boldsymbol{\eta} \rangle$ for 
$A_L(\boldsymbol{\eta})$ with the following properties:
\begin{enumerate}[label=(\alph*),ref=\thetheorem(\alph*)]
    \item \label{generic_liouville_point_a} The differential Galois group of $E$ over 
    $C \langle \boldsymbol{\eta} \rangle$ is $B^-(C)$.
    \item \label{generic_liouville_point_b} There is a fundamental solution matrix $Y_L \in B^-(E)$ and 
    there are elements $\boldsymbol{z}=(z_1,\dots,z_l)$ and $\boldsymbol{y}=(y_1,\dots,y_m)$ of $E$
    algebraically independent over $C \langle \boldsymbol{\eta} \rangle$ with 
    \begin{equation*}
     Y_L= \boldsymbol{t}(\boldsymbol{z}) \boldsymbol{u}(\boldsymbol{y}).
    \end{equation*}
    \item  \label{generic_liouville_point_c} The extension is Liouvillian and there is a tower 
    of differential fields 
    \begin{equation*}
    E_0=C \langle \boldsymbol{\eta} \rangle \subset E_1 \subset \dots \subset E_l \subset E_{l+1} 
    \subset \dots \subset E_{m+l} =E
    \end{equation*}
    where $E_i = E_{i-1}(z_i)$ with $z_i$ an exponential for $i=1, \dots,l$ and $E_i = E_{i-1}(y_{i-l})$ 
    with $y_{i-l}$ an integral for $i=l+1, \dots , m+l$. 
    \item \label{generic_liouville_point_d} The elements $\boldsymbol{z}$ and $\boldsymbol{y}$ have shape
  \begin{gather*}
   z_1= e^{ \int \bar{g}_1(\boldsymbol{\eta})}, \dots,  z_l=e^{ \int \bar{g}_l(\boldsymbol{\eta}) }, \\ 
   y_1 = \int c_1 \bar{\alpha}_1 (\boldsymbol{t}(\boldsymbol{z}))^{-1},  \dots ,  
   y_l= \int c_l \bar{\alpha}_l (\boldsymbol{t}(\boldsymbol{z}))^{-1}, \\
   y_{l+1} = \int -v_{l+1}(\boldsymbol{y}),  \dots,  
   y_m = \int -v_m(\boldsymbol{y}  ) \big) ,
  \end{gather*}
    where $\bar{g}_i(\boldsymbol{\eta})$ are as in the definition of 
    $A_L(\boldsymbol{\eta})$ and $v_{l+1}(\boldsymbol{y}),\dots , v_{m}(\boldsymbol{y})$ 
    are as in Lemma~\ref{relations_coefficients1}.
\end{enumerate}
\end{proposition}

\begin{proof}
(a) We show that for $n \in \mathbb{N}$ there are nonzero elements 
$\boldsymbol{\overline{c}}=(\bar{c}_1,\dots , \bar{c}_l)$ of $C$ such that the differential Galois group of a Picard-Vessiot extension 
for the matrix  
\[
A_L=\sum_{i=1}^l \bar{c}_i z^n H_i + c_i X_i 
\]
over the rational function field $C(z)$ with standard derivation is the group $B^-(C)$.
The statement then follows from \cite[Theorem 4.3]{Seiss} and \cite[Proposition 1.31]{P/S}.
Indeed, since $\bar{g}_1(\boldsymbol{\eta}), \dots, \bar{g}_l(\boldsymbol{\eta})$ are $C$-linearly independent homogeneous polynomials of 
degree one in $C[\eta_1, \dots, \eta_l]$, the corresponding coefficient 
matrix is invertible and so there are nonzero elements  
$\hat{c}_1, \dots , \hat{c}_l$ of $C$ such that 
\[
\bar{c}_i z^n = \bar{g}_i (\hat{c}_1 z^n, \dots , \hat{c}_l z^n).
\]
It follows that the differential $C$-Algebra homomorphism $\sigma:C\{ \boldsymbol{\eta} \} \rightarrow C[z]$ 
defined by $\sigma(\eta_i)=\hat{c}_i z^n $ satisfies $\sigma(A_L(\boldsymbol{\eta}))=A_L$ and is 
a surjective $R_1$-specialization.

In order to show that there exists $\boldsymbol{\overline{c}}$ such that $A_L$ has Galois group $B^-(C)$ 
we consider the quotient homomorphism 
\[
\pi : B^- \rightarrow B^- / [U^-,U^-]
\]
and the corresponding Lie algebra homomorphism 
$d \pi : \mathfrak{b}^- \rightarrow \mathfrak{b}^- / [ \mathfrak{u}^-,\mathfrak{u}^- ]$.
According to \cite[Proposition 10]{Kov} it is enough to prove that there exists $\boldsymbol{\overline{c}}$
such that the differential Galois group of the image $d \pi(A_L)$ 
is the full group $\pi(B^-)$. The first  
assumption of the proposition is satisfied by \cite[Corollary of Lemma 2]{Kov}.
For the remaining assumptions the standard construction method yields a Picard-Vessiot 
extension $E$ of $F$ with a fundamental solution matrix $Y \in B^-(E)$, 
since $A_L$ lies in $\mathfrak{b}^-$. 
By \cite[Proposition 3]{Kov} the logarithmic derivative of 
$\pi(Y)$ is $d \pi(A_L)$ and so $\pi(Y)$
is a fundamental solution matrix for 
$d \pi(A_L)$ and its entries generate a Picard-Vessiot extension. 
The last condition is satisfied by construction.

A basis of the Lie algebra $d\pi (\mathfrak{b}^-)$ is given by 
the images $\overline{H}_1, \dots, \overline{H}_l$ of the basis $H_1, \dots, H_l$ of the 
Cartan subalgebra and the images $\overline{X}_{1}, \dots,\overline{X}_{l} $ of the basis 
elements $X_{1}, \dots, X_l$ corresponding to the negative simple roots and so  
\[
d\pi (A_L)= \sum_{i=1}^l \bar{c}_i z^n \overline{H}_i  + c_i \overline{X}_i = \overline{A}_L .
\]
We follow the argumentation of \cite[Chapter III]{Kov} and use the notation.
We need to show that there are $\boldsymbol{\overline{c}}$ such that 
the images of 
\[
\boldsymbol{\overline{c}}_z=(\overline{c}_1 z^n,\dots,\overline{c}_l z^n) 
\]
in the quotient 
$C (z)/ \ell \delta(C (z)^*)$ are linearly independent over $\mathbb{Z}$ and that $c_i$ does not 
reduce to zero modulo $L_{\boldsymbol{\overline{c}}_z, \bar{\alpha}_i}(C (z))$. Indeed, if these statements are fulfilled, it follows then from \cite[Poposition 16]{Kov} together 
with \cite[Poposition 15]{Kov} and \cite[Lemma 5]{Kov} that the 
differential Galois group of $\overline{A}_L$ is $\pi(B^-)$.

For the first statement \cite[Lemma 6]{Kov} yields that for $z^n$ there are infinitely many choices of 
coefficients $\boldsymbol{\overline{c}}$ in $C$ such that $\boldsymbol{\overline{c}}_z$ are 
$\mathbb{Z}$-linearly independent modulo $\ell \delta ( C(z)^*)$. 
For the second statement we need to show for $i=1,\dots,l$ that the differential 
equations 
\[
L_{\boldsymbol{\overline{c}}_z, \bar{\alpha}_i}(\zeta)= \zeta' - \sum_{j=1}^l \langle \bar{\alpha}_j, \bar{\alpha}_i \rangle \overline{c}_j z^n \zeta =c_i.
\]
have no solutions in $C(z)$. 
According to \cite[Exercise 1.36 4.]{P/S} for any nonzero $\bar{c} \in C$ and 
$n \in \mathbb{N}$ the differential Galois group of a Picard-Vessiot extension of $C(z)$ 
for the differential equation defined by the matrix 
\[
 \begin{pmatrix} 0 & 1 \\ \bar{c}^2 z^{2n}/4 + \bar{c} n z^{n-1}/2 & 0 \end{pmatrix} 
\]
is conjugate to the Borel subgroup of $\mathrm{SL}_2$. It can be easily checked that 
this matrix is gauge equivalent to 
\[
  \begin{pmatrix} \bar{c} z^n/2 & c_i \\ 0 & -\bar{c} z^n/2 \end{pmatrix}
\]
and so there are no solutions in $C(z)$ for differential equations of shape 
\[
\zeta' -\bar{c} z^n \zeta =c_i.    
\]

(b) We construct a Picard-Vessiot extension $E$ of $C \langle \boldsymbol{\eta} \rangle$ for $A_L(\boldsymbol{\eta})$ such that there is a 
fundamental solution matrix $Y_L $ in $B^-(E)$ and elements 
$\boldsymbol{z}=(z_1, \dots, z_l)$ and $\boldsymbol{y}=(y_1,\dots, y_m)$ of $E$ 
such that $Y_L$ can be written as a product as stated where we use the notation of 
Chapter~\ref{Connecting the structure of reductive groups with Picard-Vessiot extensions}. 
Since $A_L(\boldsymbol{\eta})$ lies in the Lie algebra of $B^-$ the defining ideal of $B^-$ in $C[\mathrm{GL}_n]$ extends to a differential 
ideal in $C \langle \boldsymbol{\eta} \rangle[\mathrm{GL}_n]$ where the derivation 
on $X$ is defined by 
multiplication with $A_L(\boldsymbol{\eta}) $. 
It follows from (a)   
that it is a maximal differential ideal and so the quotient ring is a Picard-Vessiot 
ring $R$. By construction the matrix $Y_L:=  \overline{X}$ is a 
fundamental solution matrix and lies in $ B^-(R)$. Since $C[B^-]$ embeds canonically 
into $R$, Lemma~\ref{decomp_for_matrix_of_B} yields that there are over $C \langle \boldsymbol{\eta} \rangle$ algebraically independent elements 
$\boldsymbol{z}=(z_1,\dots, z_l)$ and $\boldsymbol{y}=(y_1,\dots,y_m)$ of $R$ such that 
\[
Y_L= \boldsymbol{t}(\boldsymbol{z})  \boldsymbol{u}(\boldsymbol{y}).
\]
The field of fractions $E$ of $R$ has then the desired properties.

(c) The extension is clearly a Liouvillian extension, since its 
differential Galois group is solvable. We prove that there is a tower of 
fields as stated for the elements $\boldsymbol{z} $ and  
$\boldsymbol{y} $ of (b) and that they are respectively exponentials and 
integrals. The proof is similar to the one of Proposition \ref{Liouvillian_part_general}. 
Recall from Chapter~\ref{The Structure of the Classical Groups} that we have 
a descending chain of normal subgroups for $B^-$ as in 
\eqref{chainofsubgroups}. For $i=0, \dots, l-2$ the quotients 
of $\bar{B}^-_i$ by $\bar{B}^-_{i+1}$ and $\bar{B}^-_{l-1}$ by $\bar{U}^-_0$ are isomorphic to 
$\mathrm{G}_m$. Hence the corresponding inclusions of fixed fields 
\[
E^{\bar{B}^-_{i+1}} \supset E^{\bar{B}^-_{i}} \quad \mathrm{and} \quad 
E^{\bar{U}^-_{0}} \supset E^{\bar{B}^-_{l-1}}
\]
are Picard-Vessiot extensions of transcendence degree one and they are 
generated by an exponential. Corollary~\ref{cor_decomp_for_matrix_of_G_point_b} and \ref{cor_decomp_for_matrix_of_G_point_c} (with $i=0$) implies after multiplication with the inverse
 of $\boldsymbol{u}(\boldsymbol{x})n(\bar{w})$ for $i=1, \dots, l-1$
the inclusions 
\[
C\langle \boldsymbol{\eta} \rangle (z_1, \dots, z_i) \subseteq E^{\bar{B}^-_i}
\quad \mathrm{and} \quad 
C\langle \boldsymbol{\eta} \rangle (z_1, \dots, z_l) \subseteq E^{\bar{U}^-_0}.
\]
We can prove inductively by comparing the 
transcendence degree of the respective fields that 
\[
E^{\bar{B}^-_i}= E^{\bar{B}^-_{i-1}}(z_i) \quad \mathrm{and} \quad 
E^{\bar{U}^-_0}= E^{\bar{B}^-_{l-1}}(z_l)
\]
and it follows with the above that $z_i$ an exponential.

For $i=0,\dots, m-1$ the quotients $\bar{U}^-_{i}$ by $\bar{U}^-_{i+1}$ are 
isomorphic to $\mathbb{G}_a$ and so the corresponding inclusions 
of fixed fields 
\[
E^{\bar{U}^-_{i+1}} \supset E^{\bar{U}^-_{i}} 
\]
are Picard-Vessiot extensions of transcendence degree one and the extensions 
are generated by an integral. We already know that 
$E^{\bar{U}^-_{0}}= C\langle \boldsymbol{\eta} \rangle (\boldsymbol{z}) $
and that $E^{\bar{U}^-_{m}}= C\langle \boldsymbol{\eta} \rangle (\boldsymbol{z},\boldsymbol{y})$. For $i=1, \dots, m-1$ we conclude
as above with Corollary~\ref{cor_decomp_for_matrix_of_G_point_c} that
$\boldsymbol{z},y_1, \dots, y_i$ are left fixed by $\bar{U}^-_{i}$, that is we have
\[
C\langle \boldsymbol{\eta} \rangle (\boldsymbol{z},y_1, \dots, y_i)
\subseteq E^{\bar{U}^-_{i}}.
\]
Again one proves inductively by comparing the transcendence degrees of the 
respective fields that 
the inclusions are actually equalities. Summing up it follows that 
$E^{\bar{U}^-_{i}}= E^{\bar{U}^-_{i-1}}(y_i)$ with $y_i$ an integral.

(d) We determine the shape of the exponentials $\boldsymbol{z}$ and the 
integrals $\boldsymbol{y}$. Using the product rule  
the logarithmic derivative of the fundamental matrix 
$Y_L=\boldsymbol{t}(\boldsymbol{z}) \boldsymbol{u}(\boldsymbol{y})$
computes as
\begin{equation} \label{eq_shape_A_L}
 \ell  \delta  (Y_L) = \ell \delta (\boldsymbol{t}(\boldsymbol{z}) ) + 
 \mathrm{Ad}(\boldsymbol{t}(\boldsymbol{z}) ) 
(\ell \delta (\boldsymbol{u}(\boldsymbol{y}) ) ) =A_L.
\end{equation}
By Proposition~\ref{log_derivate} the logarithmic derivative of 
$\boldsymbol{t}(\boldsymbol{z})$ and $\boldsymbol{u}(\boldsymbol{y})$
are elements of $\mathfrak{h}$ and $\mathfrak{u}^-$ respectively.
Moreover the adjoint action of the torus stabilizes the root spaces and 
so \eqref{eq_shape_A_L} implies 
\[
 \ell \delta (\boldsymbol{t}(\boldsymbol{z}))= \sum_{i=1}^l \bar{g}_i(\boldsymbol{\eta})  H_i.
\]
The logarithmic derivative of $t_i(z_i)$ computes as $z_i'/z_i  H_i$. 
Moreover with the product rule and the fact that the adjoint action 
of the torus stabilizes the Cartan algebra we conclude that 
$\bar{g}_i(\boldsymbol{\eta})=z_i'/z_i$ and so we have 
\begin{equation*}
z_i= e^{ \int \bar{g}_i(\boldsymbol{\eta}) }.
\end{equation*} 
It is left to determine the integrals. From \eqref{eq_shape_A_L} 
we know that 
\[
\mathrm{Ad}(\boldsymbol{t}(\boldsymbol{z}) ) 
(\ell \delta (\boldsymbol{u}(\boldsymbol{y}) ) ) =A^-_0(\boldsymbol{c}).
\]
Since the torus acts on $X_i$ by the roots we conclude with 
Lemma~\ref{relations_coefficients1} that the 
coefficients of basis elements corresponding to the negative simple roots 
satisfy $c_i= \bar{\alpha}_i( \boldsymbol{t}(\boldsymbol{z})) y_i'$. Thus for 
$i=1, \dots, l$ the integral $y_i$ is 
\[
 y_i = \int c_i \bar{\alpha}_i(\boldsymbol{t}(\boldsymbol{z}))^{-1}.
\]
For the remaining indices that is for the roots $\beta_i$ with 
$r(i)\leq -2$ the component of the 
logarithmic derivative of $\boldsymbol{u}(\boldsymbol{y})$ in 
$\mathfrak{g}_i$ has to be zero. It follows then from Lemma~\ref{relations_coefficients1} that for  $l+1\leq i \leq m$ we have 
\[
 y_{i} = \int -v_i(\boldsymbol{y}) 
\] 
where $v_i(\boldsymbol{y})$ only depends on integrals $y_1,\dots,y_j$ with $j$ maximal such 
that $r(j)=r(i)+1$.
\end{proof}

\begin{example}\label{ex:LiouvillePart1}
We continue with the example for $\mathrm{SL}_4(C)$. 
For the Weyl group element of maximal length we choose the representative 
\[
 n(\bar{w}) = E_{14} - E_{23} + E_{32} - E_{41}
\]
in the normalizer of the torus of $\mathrm{SL}_4(C)$. 
With $\bar{g}_1(\boldsymbol{\eta})= - \eta_3$, $\bar{g}_2(\boldsymbol{\eta})=- \eta_2$, 
$\bar{g}_3(\boldsymbol{\eta})=- \eta_1$ and $c_i= -1$
the defining matrix of the Liouvillian extension becomes
\[
A_L(\boldsymbol{\eta}) = -A_0^- - \eta_3 H_1 - \eta_2 H_2 -\eta_1 H_3 
\]
and it satisfies 
\[
\mathrm{Ad}(n(\bar{w}))(A_L(\boldsymbol{\eta})) = A_0^+ -g_1(\boldsymbol{\eta}) H_1 - g_2(\boldsymbol{\eta}) H_2 - g_3(\boldsymbol{\eta}) H_3
\]
where $g_i(\boldsymbol{\eta})$ is as in Example~\ref{ex:SL4Adjoint}.
With $\bar{g}_i(\boldsymbol{\eta})$ we define the exponentials 
\begin{gather*}
z_1 =  e^{\int - \eta_3}, \ z_2 = e^{\int - \eta_2}, \ z_3 = e^{\int - \eta_1} 
\end{gather*}
and obtain for the roots of height $-1$ the integrals
\begin{gather*}
y_1 = - \int e^{ \int (-2\eta_3 + \eta_2)}, \
y_2 = - \int e^{ \int (-2\eta_2 + \eta_1 + \eta_3)}, \
y_3 = - \int e^{ \int (-2 \eta_1 +  \eta_2)} .
\end{gather*}
  Furthermore with $v_i(\boldsymbol{\eta}_6)$
as in Example~\ref{ex:SL4logderivative} we compute for the remaining roots the integrals   
 \begin{eqnarray*}
     y_4 &=& \int -v_4(y_1,y_2,y_3) = \int \Big(\int e^{\int( -2 \eta_3 + \eta_2)} \Big) e^{\int (-2 \eta_2 + \eta_1 + \eta_3)} , \\
     y_5 &=& \int -v_5(y_1,y_2,y_3) = \int \Big( \int e^{\int (-2 \eta_2 +\eta_1+\eta_3)} \Big) e^{\int (-2 \eta_1+ \eta_2)} , \\
     y_6 &=& \int -v_6(y_1,\dots,y_5) = \int \Big(\int e^{\int (-2 \eta_1 + \eta_2)} \Big) \Big(\int e^{\int ( -2 \eta_3 +\eta_2)} \Big) e^{\int (-2 \eta_2 +\eta_1 +\eta_3)}.
 \end{eqnarray*}
By construction  
$Y_L=\boldsymbol{t}(\boldsymbol{z}) \boldsymbol{u}(\boldsymbol{y})$ satisfies 
$\ell \delta(Y_L)=A_L(\boldsymbol{\eta})$.
\end{example}
 
\begin{lemma}\label{lemma_relations_myshape}
Let $\boldsymbol{z}, \, \boldsymbol{y}$ be as in Proposition \ref{generic_liouville} 
and define 
\begin{equation*}
Y= \boldsymbol{u}(\boldsymbol{\eta}_m)  n(\bar{w}) 
 \boldsymbol{t}(\boldsymbol{z}) \boldsymbol{u}(\boldsymbol{y}).
\end{equation*}
We have 
\begin{equation*}
\ell\delta(Y)= A_0^+ + h_1(\boldsymbol{\eta}_m) X_{1} + \dots + h_m(\boldsymbol{\eta}_m) X_{m}
\end{equation*}
with coefficients
\begin{equation*}
  h_i(\boldsymbol{\eta}_m) = \eta_i'+ \ell_i(\boldsymbol{\eta}_m)  +  q_i(\boldsymbol{\eta}_m) 
 \end{equation*}
 where $\ell_i(\boldsymbol{\eta}_m)$ is as in Lemma~\ref{relations_coefficients2}
and $q_i(\boldsymbol{\eta}_m)$ lies in $C \{ \eta_1 ,\dots , \eta_{s_2} \}[\eta_{s_2 +1},\dots \eta_{i_2}]$ with $s_2$ and $i_2$ maximal such that $r(s_2)=r(i)+1$ and $r(i_2)=r(i)$. Moreover each 
term of $q_i(\boldsymbol{\eta}_m)$ is of degree greater than one and the derivatives appearing are at most of order one.
\end{lemma}

\begin{proof}
With the product rule and $\ell \delta(n(\bar{w}))=0$ the logarithmic derivative of $Y$ computes as
 \[
  \ell\delta(Y) = \ell\delta(\boldsymbol{u}(\boldsymbol{\eta}_m) ) + 
   \boldsymbol{u}(\boldsymbol{\eta}_m) n(\bar{w}) \  
   \ell\delta(\boldsymbol{t}(\boldsymbol{z}) \boldsymbol{u}(\boldsymbol{y})) \ 
   ( \boldsymbol{u}(\boldsymbol{\eta}_m)n(\bar{w}))^{-1} .
 \] 
We look at the two summands individually and then combine our results. 
From Lemma~\ref{relations_coefficients1} we obtain that the first summand 
is an element of $\mathfrak{u}^-(C\{ \boldsymbol{\eta}\})$ and that if 
we represent it as a linear combination of the basis elements $X_1, \dots , X_m$ then 
the coefficient of $X_{i}$ is $\eta_i'+v_i(\boldsymbol{\eta}_m)$ where 
$v_i(\boldsymbol{\eta}_m) \in C \{ \eta_1 , \dots , \eta_{s_2} \}$ with $s_2$ maximal such that 
$r(s_2)=r(i)+1$ and all its terms are of order one and of degree greater than one. 

Next we consider the second summand.  
Since $\boldsymbol{z}$ and $\boldsymbol{y}$ are as in Proposition~\ref{generic_liouville} we have 
\begin{equation*}
\ell\delta( \boldsymbol{t}(\boldsymbol{z}) \boldsymbol{u}(\boldsymbol{y}))= A_L(\boldsymbol{\eta})=
A_0^-(\boldsymbol{c}) + \sum_{i=1}^l \bar{g}_i(\boldsymbol{\eta})  H_i .
\end{equation*}
The adjoint action of $n(\bar{w})$ maps $A_L(\boldsymbol{\eta})$ by construction to  
\[
  A_0^+  + \sum_{i=1}^l -g_i(\boldsymbol{\eta})  H_i 
\]
where $g_i(\boldsymbol{\eta})$ is as in Lemma~\ref{rel_coef2_point_a}. 
Moreover the linearity of the adjoint action yields 
\begin{equation}\label{eqn1_lemma_my_relations}
\mathrm{Ad}( \boldsymbol{u}(\boldsymbol{\eta}_m)) (A_0^+) + 
\sum_{i=1}^l - g_i(\boldsymbol{\eta}) \mathrm{Ad}( \boldsymbol{u}(\boldsymbol{\eta}_m))(H_i) \, .
\end{equation}
It follows from Lemma~\ref{relations_coefficients2} that the first summand of 
\eqref{eqn1_lemma_my_relations} computes as 
\[
 \mathrm{Ad}( \boldsymbol{u}(\boldsymbol{\eta}_m))(A_0^+) = A_0^+ + 
  \sum_{i=1}^l g_i(\boldsymbol{\eta})  H_i + \sum_{i=1}^m ( \ell_i(\boldsymbol{\eta}_m)  + p_i(\boldsymbol{\eta}_m) ) X_{i}
 \]
 where $g_i(\boldsymbol{\eta})  \in C[\eta_1, \dots,  \eta_l]$ are nonzero $C$-linear 
 independent homogeneous polynomials of degree one, 
 $\ell_i(\boldsymbol{\eta}_m) \in C[\eta_{k_1}, \dots , \eta_{k_2}]$  with $k_1$ minimal and $k_2$ 
 maximal such that $r(k_1)=r(i)-1=r(k_2)$ and $\ell_i(\boldsymbol{\eta}_m)$ is a homogeneous 
 polynomial of degree one and $p_i(\boldsymbol{\eta}_m)  \in C[\eta_1 ,\dots, \eta_{i_2}]$ with
 $i_2$ maximal such that $r(i_2)=r(i)$ and each term of $p_i(\boldsymbol{\eta}_m)$ is of degree greater 
 than one. By the second formula of Remark \ref{remark_adjoint_action} 
 the adjoint action of $u_{k}(\eta_k)$ maps $H_i$ to 
 \[
 H_i + \langle \bar{\alpha}_i, \beta_k \rangle \ \eta_k \  X_k 
 \]
 and so the second summand of ($\ref{eqn1_lemma_my_relations}$) becomes
\[
 \sum_{i=1}^l - g_i(\boldsymbol{\eta}) \mathrm{Ad}(\boldsymbol{u}(\boldsymbol{\eta}_m))(H_i) = 
 \sum_{i=1}^l - g_i(\boldsymbol{\eta})  H_i  + \sum_{i=1}^m a_i(\boldsymbol{\eta}_m)  X_{i} 
 \]
 where $a_i(\boldsymbol{\eta}_m)$ is a polynomial of $C[\eta_1, \dots, \eta_{i_2}]$ with $i_2$ maximal 
 such that $r(i_2)=r(i)$ and each term of $a_i(\boldsymbol{\eta}_m)$ is of degree greater than one.
Indeed, this follows from the fact that $a_i(\boldsymbol{\eta}_m)$ is the product of 
$\eta_i \in C[\eta_1, \dots, \eta_{i_2}]$ 
and the homogeneous polynomial of degree one 
$g_i(\boldsymbol{\eta})  \in C[\eta_1, \dots, \eta_l]$.

Combining the results one checks that the coefficient of each $H_i$ in the linear combination of $\ell\delta(Y)$ 
vanishes and that the coefficient of each $X_{i}$ is as described in the statement.
\end{proof}

\begin{example}\label{ex:h1...hm}
We proceed with the example for $\mathrm{SL}_4(C)$. 
The logarithmic derivative of 
\[
Y= \boldsymbol{u}(\boldsymbol{\eta}_6)n(\bar{w}) \boldsymbol{t}(\boldsymbol{z}) \boldsymbol{u}(\boldsymbol{y}),
\]
where $\boldsymbol{z}$ and $\boldsymbol{y}$ are as in 
Example~\ref{ex:LiouvillePart1}, computes as 
\begin{eqnarray*}
\ell \delta (Y) &=& \ell \delta (\boldsymbol{u}(\boldsymbol{\eta}_6)) + 
\mathrm{Ad}(\boldsymbol{u}(\boldsymbol{\eta}_6) n(\bar{w})) (\ell \delta (\boldsymbol{t}(\boldsymbol{z}) \boldsymbol{u}(\boldsymbol{y}))) \\
 &=&  \ell \delta (\boldsymbol{u}(\boldsymbol{\eta}_6)) + 
\mathrm{Ad}(\boldsymbol{u}(\boldsymbol{\eta}_6)) ( A_0^+ - g_1(\boldsymbol{\eta}) H_1 - g_2(\boldsymbol{\eta}) H_2 -g_3(\boldsymbol{\eta}) H_3 ).
\end{eqnarray*}
Combining the results of Example~\ref{ex:SL4logderivative} and 
Example~\ref{ex:SL4Adjoint} we obtain 
\[
\ell \delta (Y) = A_0^+ + h_1(\boldsymbol{\eta}_6) X_1 + \cdots + h_6(\boldsymbol{\eta}_6) X_6
\]
with coefficients 
\begin{align*}
& h_1(\boldsymbol{\eta}_6)  =\eta_1' - \eta_4  + q_1(\boldsymbol{\eta}_6) ,    & &
h_2(\boldsymbol{\eta}_6)  =\eta_2' + \eta_4 - \eta_5 + q_2(\boldsymbol{\eta}_6),  \\  
& h_3(\boldsymbol{\eta}_6) = \eta_3' +  \eta_5 + q_3(\boldsymbol{\eta}_6), & & 
h_4(\boldsymbol{\eta}_6) = \eta_4' - \eta_6  +q_4(\boldsymbol{\eta}_6), \\
& h_5(\boldsymbol{\eta}_6) = \eta_5' + \eta_6 + q_5(\boldsymbol{\eta}_6), & &
h_6(\boldsymbol{\eta}_6) =  \eta_6' + q_6(\boldsymbol{\eta}_6),
\end{align*}
where 
\[
q_i(\boldsymbol{\eta}_6)= v_i(\boldsymbol{\eta}_6) + p_i(\boldsymbol{\eta}_6) +a_i(\boldsymbol{\eta}_6)  
\]
and $a_i(\boldsymbol{\eta}_6)$ is as in the proof of 
Lemma~\ref{lemma_relations_myshape}. By computation we obtain 
\[
q_1(\boldsymbol{\eta}_6) = \eta_1^2 , \
q_2(\boldsymbol{\eta}_6) =   \eta_2(\eta_2 - \eta_1),   \
q_3(\boldsymbol{\eta}_6) =   \eta_3 ( \eta_3 - \eta_2)  
\]
and 
\begin{eqnarray*}
q_4(\boldsymbol{\eta}_6) &= & \eta_5 \eta_1  - ( \eta_2 
(\eta_2 - \eta_1)) \eta_1 - \eta_3\eta_4     + v_4(\boldsymbol{\eta}_6), \\
q_5(\boldsymbol{\eta}_6) &=&  \eta_3 \eta_4 - \eta_5 \eta_1  
- ( \eta_3 (\eta_3 -\eta_2  )) \eta_2  + v_5(\boldsymbol{\eta}_6),\\
q_6(\boldsymbol{\eta}_6) &= &    \eta_3^2 (\eta_1 \eta_2 - \eta_4 )
+ \eta_3 ( \eta_4 \eta_2-  \eta_1 \eta_2^2 ) +\eta_5 ( \eta_1^2 - \eta_4) + v_6(\boldsymbol{\eta}_6).
\end{eqnarray*}
  Recall that the indices $1$, $2$, $3$ correspond to the roots of height $-1$, 
  the indices $4$, $5$ belong to the roots of height $-2$ and the root belonging to the index 
  $6$ is the only root of height $-3$. It is easy to check that the linear and non-linear 
  parts of the coefficients satisfy the statement of the lemma.
\end{example}

\begin{lemma}\label{relations_for_indeterminates}
Let 
\[ 
\{ i_1, \dots, i_{m-l} \mid \beta_{i_s} \notin \{ \gamma_1, \dots, \gamma_l  \} \} \cup \{ j_1, \dots, j_l \mid \beta_{j_s} \in \{ \gamma_1, \dots, \gamma_l  \}  \}
\]
be the partition of $\{1, \dots, m \}$ into indices corresponding to non-complementary and complementary roots and let 
$h_1(\boldsymbol{\eta}_m), \dots, h_m(\boldsymbol{\eta}_m)$ be the 
differential polynomials of Lemma~\ref{lemma_relations_myshape}. 
Then the system of equations 
 \begin{equation*}
  h_{i_1}(\boldsymbol{\eta}_m) =0, \dots, h_{i_{m-l}}(\boldsymbol{\eta}_m) =0 
 \end{equation*} 
 is equivalent to the system 
 \begin{equation}\label{EquivalentSystem}
 \eta_{l+1} = \bar{\ell}_{l+1}(\boldsymbol{\eta}) + \bar{p}_{l+1}(\boldsymbol{\eta}), \dots, \eta_{m} = \bar{\ell}_m(\boldsymbol{\eta})  +  \bar{p}_m(\boldsymbol{\eta})
 \end{equation}
 where $\bar{\ell}_{i}(\boldsymbol{\eta})$ and $\bar{p}_{i}(\boldsymbol{\eta})$ have the following properties: 
 \begin{enumerate}[label=(\alph*),ref=\thetheorem(\alph*)] 
     \item \label{rel_indet_point_a} The differential polynomial 
     $\bar{\ell}_{i}(\boldsymbol{\eta}) \in C\{ \eta_1, \dots,\eta_{j} \}$ is homogeneous of degree one 
     and of order $|r(i)+1|$ with $j$ maximal such that $r(j)=-1$ and $j$ does not
     correspond to a complementary root.
     \item \label{rel_indet_point_b} For the heights $q=-2, \dots, r(m)$ let $i_1$ be minimal and $i_2$ 
     be maximal such that $r(i_1)=q=r(i_2)$. Then the coefficient matrix of the linear system 
     \[
      \bar{\ell}_{i_1}(\boldsymbol{\eta})=0,  \dots  ,\bar{\ell}_{i_2}(\boldsymbol{\eta})=0 
     \]
     in the variables $\eta_1^{(|q+1|)}, \dots,\eta_{j}^{(|q+1|)}$ 
     has full rank. In case $q \leq -3$ the linear system defined by this coefficient matrix is equivalent 
     to a subsystem obtained form the corresponding coefficient matrix for $q+1$ by canceling those 
     rows whose indices belong to complementary roots.
     \item \label{rel_indet_point_c} Each term in the differential polynomial $\bar{p}_i(\boldsymbol{\eta}) \in C\{ \boldsymbol{\eta} \}$ is of degree greater than two and the derivatives appearing are at most of order $|r(i)+2|$.
 \end{enumerate}
\end{lemma}

\begin{proof}
We fix the following notation. For a height $q=-1,\dots, r(m)$
let $i_1$ be minimal and $i_2$ and $i_2'$ be maximal such that 
$r(i_1)=q=r(i_2)=r(i_2')$ and the index $i_2'$ does not correspond to a 
complementary root. Further let $k_1$ be minimal and $k_2$ be maximal 
such that $r(k_1)=q-1=r(k_2)$. If not otherwise stated the index $i$ runs between $i_1 \leq i \leq i_2'$
and the index $k$ between $k_1\leq k \leq k_2$.

We use induction on the height $q=-1,\dots, r(m)$. Let $q=-1$.
Note that $i_1=1$, $i_2'=j$ with $j$ as in (a) and $k_1 = l+1$.
By assumption  we have 
\begin{equation*}
  h_{1}(\boldsymbol{\eta}_m) = \eta_{1}'+ \ell_{i}(\boldsymbol{\eta}_m)  +  
  q_{1}(\boldsymbol{\eta}), \dots, h_{i_2'}(\boldsymbol{\eta}_m) = \eta_{i_2'}'+ 
  \ell_{i_2'}(\boldsymbol{\eta}_m)  +  q_{i_2'}(\boldsymbol{\eta})   
 \end{equation*}
 with $\ell_i(\boldsymbol{\eta}_m)  \in C[\eta_{l+1}, \dots , \eta_{k_2}]$ homogeneous of degree one 
 and $q_{i}(\boldsymbol{\eta}) \in C [ \eta_1 ,\dots , \eta_l ]$ 
 such that each term is of degree greater than one.
 Setting $h_i (\boldsymbol{\eta}_m)$
 to zero and then solving for $\ell_{i}(\boldsymbol{\eta}_m)$ we obtain the system 
  \[ 
   \ell_{1}(\boldsymbol{\eta}_m)  = - \eta_{1}' - q_{1}(\boldsymbol{\eta}),  \dots,  
   \ell_{i_2'}(\boldsymbol{\eta}_m)  = - \eta_{i_2'}' - q_{i_2'}(\boldsymbol{\eta})    
 \]
 where by assumption the left hand sides of the system 
 define a quadratic linear system 
 \[
 \ell_{1}(\boldsymbol{\eta}_m)=0,\dots,\ell_{i_2'}(\boldsymbol{\eta}_m)=0  
 \]
 of full rank in the variables 
 $\eta_{l+1},\dots, \eta_{k_2}$. Note that the right hand sides depend only on the variables
 $\eta_{1},\dots, \eta_{l}$ and the first order derivatives of $\eta_1,\dots, \eta_{i'_2}$.
 Applying Gaussian elimination we obtain the equations 
 \[
 \eta_{l+1} = \bar{\ell}_{l+1}(\boldsymbol{\eta}) + \bar{p}_{l+1}(\boldsymbol{\eta}), \ \dots, \
 \eta_{k_2} = \bar{\ell}_{k_2}(\boldsymbol{\eta}) + \bar{p}_{k_2}(\boldsymbol{\eta})
 \]
 where $\bar{\ell}_k(\boldsymbol{\eta})$ and $\bar{p}_k(\boldsymbol{\eta})$ represent the corresponding results obtained from applying the 
 same row operations to $-\eta_i'$ and to $q_i(\boldsymbol{\eta})$ 
 respectively. Clearly $\bar{\ell}_k(\boldsymbol{\eta})$ satisfies statement (a) and 
 \[
 \bar{\ell}_{l+1}(\boldsymbol{\eta})=0, \dots, \bar{\ell}_{k_2}(\boldsymbol{\eta})=0
 \]
 is a linear system of full rank in $\eta_1',\dots , \eta_{i'_2}'$. Since the properties of 
 $q_i$ also hold for their $C$-linear combinations, we conclude that
 $\bar{p}_{k}(\boldsymbol{\eta})$ satisfies (c).  
 
Let $q \leq -2$. 
By the induction assumption we have the equations 
\begin{gather}
 \label{eq:relations_for_indeterminates43} 
  \eta_{l+1} = \bar{\ell}_{l+1}(\boldsymbol{\eta}) + \bar{p}_{l+1}(\boldsymbol{\eta}),  \dots, 
 \eta_{i_1-1} = \bar{\ell}_{i_1 -1}(\boldsymbol{\eta}) + \bar{p}_{i_1-1}(\boldsymbol{\eta}), \\
 \label{eq:relations_for_indeterminates44} 
 \eta_{i_1} = \bar{\ell}_{i_1}(\boldsymbol{\eta}) + \bar{p}_{i_1}(\boldsymbol{\eta}),  \dots, 
 \eta_{i_2} = \bar{\ell}_{i_2}(\boldsymbol{\eta}) + \bar{p}_{i_2}(\boldsymbol{\eta}) ,
\end{gather}
where $\bar{\ell}_i(\boldsymbol{\eta})$ and $\bar{p}_i(\boldsymbol{\eta})$ with $i=l+1,\dots,i_2$ satisfy (a), (b) and (c) respectively.
We obtained these equations from solving the equations   
$h_i(\boldsymbol{\eta}_m)=0$ with $i= 1, \dots, i_1-1$ and $i$ does not correspond 
to a complementary root. Thus we did not yet use the differential polynomials
\[
h_{i_1}(\boldsymbol{\eta}_m), \dots,h_{i'_2}(\boldsymbol{\eta}_m) .   
\]
 Setting then to zero and solving for $\ell_{i}(\boldsymbol{\eta}_m) $
 we obtain the system of equations
 \begin{equation}\label{eq:relations_for_indeterminates45}
   \ell_{i_1}(\boldsymbol{\eta}_m) =  - \eta_{i_1}' - q_{i_1}(\boldsymbol{\eta}_m),  \dots,    
   \ell_{i'_2}(\boldsymbol{\eta}_m)  =  - \eta_{i'_2}' - q_{i'_2}(\boldsymbol{\eta}_m)  
 \end{equation}
 where by assumption the left hand sides define a quadratic linear system
 of full rank in the indeterminates $\eta_{k_1}, \dots , \eta_{k_2}$ and the right hand sides 
 depend on the differential polynomials 
 \[
 q_{i}(\boldsymbol{\eta}_m) \in 
 C \{ \eta_1 ,\dots , \eta_{i_1 -1} \}[\eta_{i_1},\dots, \eta_{i_2}]
 \]
 whose terms are of degree greater than one and only contain derivatives of order at most  
 one. Substituting \eqref{eq:relations_for_indeterminates44} into $- \eta_{i_1}', \dots,- \eta_{i'_2}'$
 as well as \eqref{eq:relations_for_indeterminates43} and \eqref{eq:relations_for_indeterminates44} into 
 $- q_{i}(\boldsymbol{\eta}_m) $ the equations in
 \eqref{eq:relations_for_indeterminates45} become
  \begin{equation}\label{eq:relations_for_indeterminates46}
   \ell_{i_1}(\boldsymbol{\eta}_m) =  - \bar{\ell}_{i_1}(\boldsymbol{\eta})' - \bar{q}_{i_1}(\boldsymbol{\eta}),  \dots,   
   \ell_{i'_2}(\boldsymbol{\eta}_m)  =  - \bar{\ell}_{i'_2}(\boldsymbol{\eta})' - \bar{q}_{i'_2}(\boldsymbol{\eta})  
 \end{equation}
 where in $\bar{q}_i(\boldsymbol{\eta})$ we collected the terms stemming from $-\bar{p}_{i}(\boldsymbol{\eta})'$ and $-q_{i}(\boldsymbol{\eta}_m)$. 
 Since the linear system 
 \begin{equation}\label{eq:relations_for_indeterminates47}
    \bar{\ell}_{i_1}(\boldsymbol{\eta})=0, \dots,\bar{\ell}_{i_2}(\boldsymbol{\eta})=0 
 \end{equation}
 in the variables $\eta_1^{(|q+1|)}, \dots,\eta_j^{(|q+1|)}$ 
 has full rank, the linear system 
 \[
 \bar{\ell}_{i_1}(\boldsymbol{\eta})'=0, \dots, \bar{\ell}_{i'_2}(\boldsymbol{\eta})'=0  
 \]
 in the variables $\eta_1^{(|q|)}, \dots,\eta_j^{(|q|)}$ has full rank
 and is equal to the subsystem of \eqref{eq:relations_for_indeterminates47}
 obtained by canceling those equations whose index correspond to a complementary root. 
 The properties of $\bar{p}_{i}(\boldsymbol{\eta})$ and $q_{i}(\boldsymbol{\eta})$ imply that 
 $\bar{q}_i(\boldsymbol{\eta})$ satisfies (c).
 We apply now Gaussian elimination to \eqref{eq:relations_for_indeterminates46} and obtain 
 \[
 \eta_{k_1} = \bar{\ell}_{k_1}(\boldsymbol{\eta}) + \bar{p}_{k_1}(\boldsymbol{\eta}), \dots ,\eta_{k_2} = \bar{\ell}_{k_2}(\boldsymbol{\eta}) + \bar{p}_{k_2}(\boldsymbol{\eta})
 \]
 where $ \bar{\ell}_{k}(\boldsymbol{\eta})$ and $\bar{p}_{k}(\boldsymbol{\eta})$ 
 are obtained by applying the same row operations to 
 $-\bar{\ell}_{i}(\boldsymbol{\eta})' $ and $- \tilde{q}_{i}(\boldsymbol{\eta})$
 respectively. Since the above mentioned properties of the last differential polynomials 
 do not change under 
 row operations, they also hold for $\bar{\ell}_{k}(\boldsymbol{\eta})$ and $\bar{p}_{k}(\boldsymbol{\eta})$.
\end{proof}

\begin{example}\label{ex:sl4_noncomplementaryroots}
We continue with the example for $\mathrm{SL}_4$. The equations corresponding to non-complementary roots of height $-1$ are 
\[
h_1(\boldsymbol{\eta}_6) = \eta_1' - \eta_4  + \eta_1^2 =0, \quad  
h_2(\boldsymbol{\eta}_6) = \eta_2' + \eta_4 - \eta_5 + \eta_2(\eta_2-\eta_1 ) =0.
\]
Solving for their linear parts yields 
\[
\eta_4 = \eta_1'   + \eta_1^2, \quad  
- \eta_4 + \eta_5 = \eta_2'  + \eta_2( \eta_2-\eta_1) 
\]
and if we add the first to the second equation we obtain 
\begin{equation}\label{eq1:example123}
\eta_4 = \eta_1'   +  \bar{p}_4 , \quad   
\eta_5 = \eta_2' + \eta_1' + \bar{p}_5 
\end{equation}
where 
\[
\bar{p}_4(\boldsymbol{\eta}) = \eta_1^2 , \quad \bar{p}_5(\boldsymbol{\eta}) = \eta_1^2 + \eta_2(\eta_2-\eta_1).
\]
One easily checks that the linear and non-linear parts in \eqref{eq1:example123} 
satisfy respectively the statements of Lemma~\ref{rel_indet_point_a} and \ref{rel_indet_point_c}.
Furthermore the linear system 
\[
\bar{\ell}_4(\boldsymbol{\eta}) = \eta_1' =0 , \  \bar{\ell}_5(\boldsymbol{\eta})=\eta_2' + \eta_1' = 0 
\]
fulfills the first part of \ref{rel_indet_point_b}.
For the non-complementary root of height $-2$ we consider the equation 
\[
h_4(\boldsymbol{\eta}_6) = \eta_4' - \eta_6 - ( \eta_2 
( \eta_2 - \eta_1)) \eta_1 - \eta_3\eta_4    + \eta_5 \eta_1  - \eta_2' \eta_1 =0
\]
which we solve for its linear part. We obtain 
\[
\eta_6 = \eta_4'  - ( \eta_2 
( \eta_2 - \eta_1)) \eta_1 - \eta_3\eta_4    + \eta_5 \eta_1  - \eta_2' \eta_1 .
\]
We use the expressions in \eqref{eq1:example123} for the substitution of 
$\eta_4$ and $\eta_5$ and obtain 
\[
\eta_6 =  \eta_1'' + \bar{p}_6(\boldsymbol{\eta})
\]
where 
\begin{equation*}
 \bar{p}_6(\boldsymbol{\eta}) = 3 \eta_1  \eta_1'  +  \eta_1^3 -  \eta_3  \eta_1' 
 - \eta_3  \eta_1^2  .
\end{equation*}
One sees that the linear and non-linear part satisfy \ref{rel_indet_point_a} and \ref{rel_indet_point_c} 
respectively. The linear system
$\bar{\ell}_6(\boldsymbol{\eta})= \eta_1''=0$ clearly satisfies the first part of \ref{rel_indet_point_b} and canceling the second equation of the system 
\[
\bar{\ell}_4(\boldsymbol{\eta})=0, \ \bar{\ell}_5(\boldsymbol{\eta}) = 0 ,
\]
which corresponds to a complementary root, yields the second part of \ref{rel_indet_point_b}. 
The unique root of height $-3$ is a complementary root and so there is no equation to solve.
\end{example}

\begin{lemma}\label{lem:fullrankprolongedlinearsystem}
We keep the notation of Lemma~\ref{relations_for_indeterminates}. Using the equivalent 
 system~\eqref{EquivalentSystem} the differential polynomials 
 $h_{j_1}(\boldsymbol{\eta}_m), \dots , h_{j_l}(\boldsymbol{\eta}_m)$ whose indices correspond 
 to the complementary roots reduce to differential polynomials 
\[
 h_{j_1}(\boldsymbol{\eta})= \hat{\ell}_{j_1}(\boldsymbol{\eta}) + \hat{p}_{j_1}(\boldsymbol{\eta}),  \dots,  
 h_{j_l}(\boldsymbol{\eta})= \hat{\ell}_{j_l}(\boldsymbol{\eta}) + \hat{p}_{j_l}(\boldsymbol{\eta})
\]
where $\hat{\ell}_{j_i}(\boldsymbol{\eta})$ and $\hat{p}_{j_i}(\boldsymbol{\eta})$ have the following properties:
\begin{enumerate}[label=(\alph*),ref=\thetheorem(\alph*)] 
    \item \label{fullrankprollinsys_point_a} The differential polynomial $\hat{\ell}_{j_i}(\boldsymbol{\eta}) \in C\{ \boldsymbol{\eta} \}$ is homogeneous of degree one and of order $| r(j_i) |$.
    \item \label{fullrankprollinsys_point_b} The matrix formed by the coefficients of the linear system 
    \[
    \hat{\ell}_{j_1}(\boldsymbol{\eta})=0,  \dots ,  \hat{\ell}_{j_l}(\boldsymbol{\eta})=0 
    \]
    in the variables $\eta_1,\dots, \eta_l$ obtained by ignoring the derivatives 
    is quadratic and has full rank.
    \item \label{fullrankprollinsys_point_c} The derivatives appearing in $\hat{p}_{j_i}(\boldsymbol{\eta}) \in C\{ \boldsymbol{\eta} \}$ are at most of order $|r(j_i)+1|$ and each term is of 
    degree greater than one.
\end{enumerate}
\end{lemma}

\begin{proof}
We fix the following notation. For a height $q=-1,\dots, r(m)$ let $j_{r_1}$ and $j_{r_2}$ be the minimal and 
maximal index among the indices of the complementary roots $ \{ j_1, \dots , j_l \}$ such that 
$q \leq r(j_{r_1})=r(j_{r_2})$ and let $i_1$ and $k_1$ be minimal and 
$i_2$ and $k_2$ be maximal such that $r(i_1)=q=r(i_2)$ and $r(k_1)=q-1=r(k_2)$. Furthermore in the 
following the index $j$ (resp.~$i$ and $k$) runs between 
$j_{r_1}\leq j \leq j_{r_2}$ (resp.~$i_1\leq i \leq i_2$ and $k_1 \leq k \leq k_2 $). If not otherwise stated 
we ignore the derivatives of the variables when we consider a linear system.

We prove by induction on the height $q=-1,\dots, r(m)$ that using the equations 
\begin{equation*} 
    \eta_{l+1} = \bar{\ell}_{l+1}(\boldsymbol{\eta}) + \bar{p}_{l+1}(\boldsymbol{\eta}),    \dots ,  \eta_{m} = \bar{\ell}_{m}(\boldsymbol{\eta}) + \bar{p}_{m}(\boldsymbol{\eta})
\end{equation*}
of Lemma~\ref{relations_for_indeterminates} the differential polynomials 
$h_{j}(\boldsymbol{\eta}_m)$ reduce to  
\[
 h_{j_1}(\boldsymbol{\eta})= \hat{\ell}_{j_1}(\boldsymbol{\eta}) + \hat{p}_{j_1}(\boldsymbol{\eta}),  \dots, 
 h_{j_2}(\boldsymbol{\eta})= \hat{\ell}_{j_2}(\boldsymbol{\eta}) + \hat{p}_{j_2}(\boldsymbol{\eta})
\]
where $\hat{\ell}_{j}(\boldsymbol{\eta})$ satisfy (a), $\hat{p}_{j}(\boldsymbol{\eta})$ fulfill (c) and the linear system 
\[
\bar{\ell}_{k_1}(\boldsymbol{\eta})=0,  \dots,  \bar{\ell}_{k_2}(\boldsymbol{\eta})=0, \
\hat{\ell}_{j_1}(\boldsymbol{\eta})=0,  \dots,  \hat{\ell}_{j_{r_2}}(\boldsymbol{\eta})=0
\]
has full rank.

Let $q=-1$. In this case $j_{r_1}=j_1$ and $k_1=l+1$. We have the differential polynomials 
\begin{equation}\label{eq:finitdiffdegzweite}
   h_{j_1}(\boldsymbol{\eta}_m)=  
   \eta_{j_1}'+ \ell_{j_1}(\boldsymbol{\eta}_m)  +  q_{j_1}(\boldsymbol{\eta}),  \dots,  
  h_{j_2}(\boldsymbol{\eta}_m)=  
   \eta_{j_2}'+ \ell_{j_2}(\boldsymbol{\eta}_m)  +  q_{j_2}(\boldsymbol{\eta})
\end{equation}
where according to Lemma~\ref{lemma_relations_myshape} the element $q_{j}(\boldsymbol{\eta})$ lies in 
$C[\eta_{1},\dots, \eta_{l}]$ and each of its terms 
is of degree greater than one and $\ell_{j}(\boldsymbol{\eta}_m)$ is 
homogeneous of degree one in the variables $\eta_{l+1},\dots ,\eta_{k_2}$.
From Lemma~\ref{relations_for_indeterminates} we obtain the equations    
\begin{equation}\label{eq:finitdiffdeg3}
\eta_{l+1} = \bar{\ell}_{l+1}(\boldsymbol{\eta}) + \bar{p}_{l+1}(\boldsymbol{\eta}) ,
 \dots, 
\eta_{k_2} = \bar{\ell}_{k_2}(\boldsymbol{\eta}) + \bar{p}_{k_2}(\boldsymbol{\eta})
\end{equation}
where 
\[
\bar{\ell}_{l+1}(\boldsymbol{\eta})=0,\dots,\bar{\ell}_{k_2}(\boldsymbol{\eta})=0 
\]
is a linear system in the variables 
$\eta_1', \dots , \eta_{j_1 -1}'$ of full rank and each term of 
$\bar{p}_{k}(\boldsymbol{\eta})$ which lies in $C[\eta_1,\dots,\eta_l ]$ 
is of degree greater than one.
We substitute the expressions of \eqref{eq:finitdiffdeg3} into $\ell_{j}(\boldsymbol{\eta}_m)$ and 
$q_{j}(\boldsymbol{\eta})$ of the right hand sides of \eqref{eq:finitdiffdegzweite} and obtain 
\begin{gather*}
   h_{j}(\boldsymbol{\eta})=  \eta_{j}'+ \ell_{j}(\bar{\ell}_{l+1}(\boldsymbol{\eta}), \dots, 
   \bar{\ell}_{k_2}(\boldsymbol{\eta})) + 
   \ell_{j}(\bar{p}_{l+1}(\boldsymbol{\eta}) , \dots, \bar{p}_{k_2}(\boldsymbol{\eta}) ) +  q_{j}(\boldsymbol{\eta}) 
\end{gather*}
where $j_1 \leq j \leq j_{r_2}$. Since the variables $ \eta_{j_1}',\dots, \eta_{j_{r_2}}'$ do not appear among  
the variables $\eta_1', \dots , \eta_{j_1 -1}'$, we conclude that the linear system 
\begin{gather*}
\bar{\ell}_{l+1}(\boldsymbol{\eta})=0, \dots, \bar{\ell}_{k_2}(\boldsymbol{\eta})=0,  \ 
\hat{\ell}_{j}(\boldsymbol{\eta}) = \eta_{j}'+ \ell_{j}(\bar{\ell}_{l+1}(\boldsymbol{\eta}),\dots,\bar{\ell}_{k_2}(\boldsymbol{\eta}))=0 
\end{gather*}
with $j_1 \leq j \leq j_{r_2}$ has full rank and all variables 
in $\hat{\ell}_{j}(\boldsymbol{\eta})$ have order one. 
Since $\ell_{j}(\boldsymbol{\eta}_m)$ is linear in the variables $\eta_{k_1},\dots, \eta_{k_2}$ and 
$\bar{p}_{k}(\boldsymbol{\eta})$ as well as $q_{j}(\boldsymbol{\eta})$ are polynomials where each 
term is of degree greater than one, the same holds for
\[
 \hat{p}_{j}(\boldsymbol{\eta})= \ell_{j}(\bar{p}_{l+1}(\boldsymbol{\eta}), \dots, \bar{p}_{k_2} (\boldsymbol{\eta})) +  q_{j}(\boldsymbol{\eta}) 
\]
with $j_1 \leq j \leq j_{r_2}$.

Let $q \leq -2$. We distinguish between the cases when there are complementary roots of 
height $q$ and when there are not. 
In the first case we have the equations  
\begin{equation}\label{eq:finitdiffdeg5}
  h_{j}(\boldsymbol{\eta}_m)=  \eta_{j}'+ \ell_{j}(\boldsymbol{\eta}_m) + 
   q_{j}(\boldsymbol{\eta}_m) 
\end{equation}
with $j_1 \leq j \leq j_{r_2}$
where the non-linear parts $q_{j}(\boldsymbol{\eta}_m)$ are as in Lemma~\ref{lemma_relations_myshape} 
and the linear parts form the linear system
\[
 \ell_{j_{r_1}}(\boldsymbol{\eta}_m)=0, \dots , \ell_{j_{r_2}}(\boldsymbol{\eta}_m)=0
\]
of full rank in the variables $\eta_{k_1},\dots, \eta_{k_2}$.  
From Lemma~\ref{relations_for_indeterminates} we obtain the two systems of equations   
$\eta_{i} = \bar{\ell}_{i}(\boldsymbol{\eta}) + \bar{p}_{i}(\boldsymbol{\eta})$ and     
$\eta_{k} = \bar{\ell}_{k}(\boldsymbol{\eta}) + \bar{p}_{k}(\boldsymbol{\eta})$ 
with $i_1 \leq i \leq i_2$ and $k_1 \leq k \leq k_2$.
The corresponding linear systems 
\[
\bar{\ell}_{i_1}(\boldsymbol{\eta})=0, \dots, \bar{\ell}_{i_2}(\boldsymbol{\eta})=0 \quad \mathrm{and} \quad  
 \bar{\ell}_{k_1}(\boldsymbol{\eta})=0, \dots, \bar{\ell}_{k_2}(\boldsymbol{\eta})=0 
\]
in the variables $\eta_1^{(|q+1|)}, \dots,\eta_{j'}^{(|q+1|)}$ and 
$\eta_1^{(|q|)}, \dots,\eta_{j'}^{(|q|)}$ respectively have full rank where $j'$ is as $j$ in Lemma~\ref{relations_for_indeterminates}. 
We substitute the corresponding expressions among $\eta_{i_1},\dots, \eta_{i_2}$ into 
$ \eta_{j_{r_1}}',\dots , \eta_{j_{r_2}}'$ and the expressions for $\eta_{k_1},\dots,\eta_{k_2}$ into 
$ \ell_{j}(\boldsymbol{\eta}_m)$ of the right hand sides of the equations in   \eqref{eq:finitdiffdeg5}. 
We obtain  
\begin{gather*}
h_{j}(\boldsymbol{\eta}_m) =  \hat{\ell}_{j}(\boldsymbol{\eta}) + \bar{p}_{j}(\boldsymbol{\eta})'+ \ell_{j}(\bar{p}_{k_1}(\boldsymbol{\eta}) , \dots, \bar{p}_{k_2}(\boldsymbol{\eta}) ) +  q_{j}(\boldsymbol{\eta}_m) 
\end{gather*}
with $ j_{r_1} \leq j \leq j_{r_2}$ where the homogeneous linear parts 
\[
\hat{\ell}_{j}(\boldsymbol{\eta}) = \bar{\ell}_{j}(\boldsymbol{\eta})' + \ell_{j}(\bar{\ell}_{k_1}(\boldsymbol{\eta}),\dots,\bar{\ell}_{k_2}(\boldsymbol{\eta}))
\]
are in variables $\eta_1^{(|q|)}, \dots , \eta_{l}^{(|q|)}$.
By induction assumption the linear system
\[
\bar{\ell}_{i_1}(\boldsymbol{\eta})=0,  \dots,  \bar{\ell}_{i_2}(\boldsymbol{\eta})=0, \ \hat{\ell}_{j_1}(\boldsymbol{\eta})=0, \dots,  \hat{\ell}_{j_{r_1-1}}(\boldsymbol{\eta})=0
\]
has full rank and by Lemma~\ref{rel_indet_point_b} the linear 
system 
\[
\bar{\ell}_{k_1}(\boldsymbol{\eta})=0, \dots, \bar{\ell}_{k_2}(\boldsymbol{\eta})=0
\]
is equivalent to the subsystem 
\[
\bar{\ell}_{i_1}(\boldsymbol{\eta})=0, \dots,  \bar{\ell}_{j_{r_1}-1}(\boldsymbol{\eta})=0
\]
of the system
\[
\bar{\ell}_{i_1}(\boldsymbol{\eta})=0, \dots, \bar{\ell}_{i_2}(\boldsymbol{\eta})=0
\]
obtained by canceling the
rows which correspond to complementary roots of height $q$.
We conclude that the linear system 
\begin{gather*}
\bar{\ell}_{k_1}(\boldsymbol{\eta})=0, \dots, \bar{\ell}_{k_2}(\boldsymbol{\eta})=0, \\
\hat{\ell}_{j_1}(\boldsymbol{\eta})=0,\dots, \hat{\ell}_{j_{r_1 -1}}(\boldsymbol{\eta})=0, \ 
\hat{\ell}_{j_{r_1}}(\boldsymbol{\eta}) = 0,  \dots, 
\hat{\ell}_{j_{r_2}}(\boldsymbol{\eta}) = 0
\end{gather*}
has full rank. It follows from the order of the variables in $\bar{\ell}_{i}(\boldsymbol{\eta}) $ 
and $\bar{\ell}_{k}(\boldsymbol{\eta}) $ that the variables in 
$\hat{\ell}_{j}(\boldsymbol{\eta})$ have order $|q|$. 
This proves (a) and (b).

We substitute the expression for $\eta_{l+1},\dots,\eta_{i_2}$ which we obtain from 
Lemma~\ref{relations_for_indeterminates} into the non-linear parts 
\[ 
\hat{p}_{j}= \bar{p}_{j}(\boldsymbol{\eta})' + \ell_{j}(\bar{p}_{k_1}(\boldsymbol{\eta}) , \dots, \bar{p}_{k_2}(\boldsymbol{\eta}) ) +  q_{j}(\boldsymbol{\eta}_m) 
\]
with $j_{r_1} \leq j \leq j_{r_2}$.
The elements $q_{j}(\boldsymbol{\eta}_m)$ lie  in 
$C\{ \eta_{1},\dots, \eta_{i_1-1} \}[\eta_{i_1  },\dots ,\eta_{i_2}]$ according to Lemma~\ref{lemma_relations_myshape} and each 
term of $q_{j}(\boldsymbol{\eta}_m)$ is of degree greater than one and the derivatives 
appearing are at most of order one.
Since the derivatives appearing in the expressions of $\eta_{i_1},\dots, \eta_{i_2}$ and of 
$\eta_{1},\dots, \eta_{i_1-1}$ are at most of order $|q+1|$ and $|q+2|$ respectively, the derivatives 
appearing in $q_{j}(\boldsymbol{\eta}_m)$ are at most of order $|q+1|$. The derivatives appearing in 
$\bar{p}_{k}(\boldsymbol{\eta})$ and $\bar{p}_{j}(\boldsymbol{\eta})$ are at most of 
order $|q+1|$ and $|q+2|$ respectively. Since $\ell_{j}(\boldsymbol{\eta}_m)$ are homogeneous 
of degree one in the variables $\eta_{k_1},\dots,\eta_{k_2}$, we conclude that the derivatives appearing 
in $\hat{p}_{j}(\boldsymbol{\eta})$ are at most of order $|q+1|$. 

Assume that we are now in the case when there are no complementary roots of height $q$.
By induction assumption we have that the linear system
\[
\bar{\ell}_{i_1}(\boldsymbol{\eta})=0, \dots, \bar{\ell}_{i_2}(\boldsymbol{\eta})=0, \ \hat{\ell}_{j_1}(\boldsymbol{\eta})=0,\dots, \hat{\ell}_{j_{r_2}}(\boldsymbol{\eta})=0
\]
has full rank. Since the two systems 
\[
\bar{\ell}_{i_1}(\boldsymbol{\eta})=0, \dots, \bar{\ell}_{i_2}(\boldsymbol{\eta})=0 \quad \mathrm{and} \quad  
\bar{\ell}_{k_1}(\boldsymbol{\eta})=0, \dots, \bar{\ell}_{k_2}(\boldsymbol{\eta})=0
\]
are equivalent by Lemma~\ref{rel_indet_point_b}, it follows that 
\[
\bar{\ell}_{k_1}(\boldsymbol{\eta})=0,  \dots,  \bar{\ell}_{k_2}(\boldsymbol{\eta})=0, \ \hat{\ell}_{j_1}(\boldsymbol{\eta})=0,  \dots,  \hat{\ell}_{j_{r_2}}(\boldsymbol{\eta})=0
\]
also has full rank.
\end{proof}

\begin{example}
We continue with the example for $\mathrm{SL}_4$. 
From Example~\ref{ex:h1...hm} we obtain that
the coefficients in the linear representation of 
$\ell \delta (Y)$ 
of the basis elements for the non-complementary roots are
\begin{equation}\label{eq1:exampleSL4NonComp}
h_3(\boldsymbol{\eta}_6) = \eta_3' +  \eta_5 + q_3(\boldsymbol{\eta}_6), \
h_5(\boldsymbol{\eta}_6) = \eta_5' + \eta_6 + q_5(\boldsymbol{\eta}_6), \
h_6(\boldsymbol{\eta}_6) = \eta_6' +  q_6(\boldsymbol{\eta}_6)  .
\end{equation}
Example \ref{ex:sl4_noncomplementaryroots} provides the expressions
\begin{equation*}\label{eq2:exampleSL4NonComp}
\eta_4 = \eta_1'   +  \bar{p}_4(\boldsymbol{\eta}) , \   
\eta_5 = \eta_2' + \eta_1' + \bar{p}_5(\boldsymbol{\eta}), \ 
\eta_6 =  \eta_1'' + \bar{p}_6(\boldsymbol{\eta}) .
\end{equation*}
We substitute these accordingly into the differential polynomials of   \eqref{eq1:exampleSL4NonComp} 
and obtain 
\begin{equation*}\label{eq3:exampleSL4NonComp}
h_3(\boldsymbol{\eta}) =  \eta_3' +  \eta_2' + \eta_1' + \hat{p}_3(\boldsymbol{\eta}) , \
h_5(\boldsymbol{\eta}) =  \eta_2'' + 2 \eta_1'' + \hat{p}_5(\boldsymbol{\eta}) , \ 
h_6(\boldsymbol{\eta}) =  \eta_1''' +\hat{p}_6(\boldsymbol{\eta}) 
\end{equation*}
where 
\[
\hat{p}_3(\boldsymbol{\eta}) = \bar{p}_5(\boldsymbol{\eta}) + q_3(\boldsymbol{\eta}), \ 
\hat{p}_5(\boldsymbol{\eta})= \bar{p}_5(\boldsymbol{\eta})' + \bar{p}_6(\boldsymbol{\eta}) + q_5(\boldsymbol{\eta}) , \  
\hat{p}_6(\boldsymbol{\eta})=\bar{p}_6(\boldsymbol{\eta})' +  q_6(\boldsymbol{\eta}).
\]
When we consider in the following a linear systems, we ignore the derivatives of the 
variables, that is we understand the systems to be in the respective variables 
$\boldsymbol{\eta}$. We follow the argumentation of the proof of 
Lemma~\ref{lem:fullrankprolongedlinearsystem} to show that the linear system 
\[
\hat{\ell}_3(\boldsymbol{\eta})=\eta_3' + \eta_2' + \eta_1'=0 ,\ \hat{\ell}_5(\boldsymbol{\eta}) = \eta_2''+2\eta_1''=0, \ 
\hat{\ell}_6(\boldsymbol{\eta})= \eta_1'''=0
\]
has full rank. 
The linear system 
\[
\bar{\ell}_4(\boldsymbol{\eta}) =\eta_1'=0, \ \bar{\ell}_5(\boldsymbol{\eta})= \eta_2' + \eta_1'=0 
\]
which corresponds to the roots of 
height $-2$ has full rank. It follows that the linear system 
\[
\bar{\ell}_4(\boldsymbol{\eta})=0, \ \bar{\ell}_5(\boldsymbol{\eta}) =0, \  \hat{\ell}_3(\boldsymbol{\eta})=\eta_3' + \eta_2' + \eta_1'=0 
\]
has full rank, since the variable $\eta_3'$ only 
appears in $\hat{\ell}_3(\boldsymbol{\eta})$.
For the root of height $-3$ the linear system $\bar{\ell}_6(\boldsymbol{\eta})= \eta_1''=0 $
is equivalent to the subsystem $\bar{\ell}_4(\boldsymbol{\eta}) =0$ of 
\[
\bar{\ell}_4 (\boldsymbol{\eta})=0, \  \bar{\ell}_5(\boldsymbol{\eta})=0.
\]
It follows that the linear system 
\[
\bar{\ell}_6(\boldsymbol{\eta})=0 ,\ \hat{\ell}_5(\boldsymbol{\eta}) = \eta_2''+2\eta_1''=0, \ \hat{\ell}_3(\boldsymbol{\eta})=\eta_3' + \eta_2' + \eta_1'=0
\]
has full rank, since $\hat{\ell}_5(\boldsymbol{\eta}) =0$
is obtained from adding $\bar{\ell}_6(\boldsymbol{\eta})=0$ to $\bar{\ell}_5(\boldsymbol{\eta})'=0$.
Because the linear system $\hat{\ell}_6(\boldsymbol{\eta})= \eta_1'''=0 $
is equal to the linear system $\bar{\ell}_6(\boldsymbol{\eta})'=0$ we obtain that the linear system
\[
\hat{\ell}_6(\boldsymbol{\eta})= \eta_1'''=0, \ \hat{\ell}_5(\boldsymbol{\eta})=\eta_2''+2\eta_1''=0 , \ \hat{\ell}_3(\boldsymbol{\eta})=\eta_3' + \eta_2' + \eta_1'=0
\]
has full rank.

It is left to check that $\hat{p}_j(\boldsymbol{\eta})$ fulfills  ~\ref{fullrankprollinsys_point_c} for $j=3,5,6$.  
Since all terms in $\bar{p}_5(\boldsymbol{\eta})$, $\bar{p}_6(\boldsymbol{\eta})$, $q_3(\boldsymbol{\eta})$, $q_5(\boldsymbol{\eta})$ and $q_6(\boldsymbol{\eta})$ are of degree 
greater than one, the same holds for $\hat{p}_j(\boldsymbol{\eta})$. We determine the maximal order of the 
derivatives appearing in $\hat{p}_j(\boldsymbol{\eta})$. Since $\bar{p}_5(\boldsymbol{\eta})$ and $q_3(\boldsymbol{\eta})$ are polynomials, so 
is $\hat{p}_3(\boldsymbol{\eta})$. The derivatives appearing in $\bar{p}_6(\boldsymbol{\eta})$ and $q_5(\boldsymbol{\eta})$ are at most of order one and so the
same holds for $\hat{p}_5(\boldsymbol{\eta})$. Finally the derivatives appearing in $q_6(\boldsymbol{\eta})$ are at most of order one and therefore the derivatives appearing in $\hat{p}_6(\boldsymbol{\eta})$ are at most of order two.

\end{example}

\begin{theorem}\label{theorem_noetherstyl}
Let $E=C\langle \boldsymbol{\eta}\rangle (\boldsymbol{z},\boldsymbol{y})$ 
and the elements $\boldsymbol{h}=(h_{j_1}(\boldsymbol{\eta}), \dots , h_{j_l}(\boldsymbol{\eta}))$
be as in Lemma~\ref{generic_liouville}  
and~\ref{lem:fullrankprolongedlinearsystem}. 
Then $E$ 
is a Picard-Vessiot extension of $C\langle \boldsymbol{h} \rangle$ for 
\begin{equation*}
 A_G (\boldsymbol{h} ) = A_0^+ + \sum_{k=1}^l h_{j_k}(\boldsymbol{\eta})  X_{j_k}   
\end{equation*} 
with differential Galois group $G(C)$. The differential field  
$C\langle \boldsymbol{h}\rangle$ is a purely differential transcendental extension of $C$ of degree $l$. 
\end{theorem}

\begin{proof} 
  We prove that the differential field extension 
 $ C \langle \boldsymbol{h} \rangle$ of $C$
 is a purely differential transcendental extension of degree $l$. 
 Clearly $C \langle \boldsymbol{\eta} \rangle$ is a purely differential transcendental extension of $C$
 of degree $l$ and we have the tower of differential fields 
 \[ 
  C \subset C \langle \boldsymbol{h} \rangle 
  \subset C \langle \boldsymbol{\eta} \rangle. 
 \]
 From Lemma~\ref{lem:fullrankprolongedlinearsystem} we obtain that 
\[
 h_{j_1}(\boldsymbol{\eta})= \hat{\ell}_{j_1}(\boldsymbol{\eta}) + \hat{p}_{j_1}(\boldsymbol{\eta})
 , \dots,  
 h_{j_l}(\boldsymbol{\eta})= \hat{\ell}_{j_l}(\boldsymbol{\eta}) + \hat{p}_{j_l}(\boldsymbol{\eta}).
\]
Moreover the properties of 
$\hat{\ell}_{j_1}(\boldsymbol{\eta}), \dots, \hat{\ell}_{j_l}(\boldsymbol{\eta})$ and 
$\hat{p}_{j_1}(\boldsymbol{\eta}), \dots, \hat{p}_{j_l}(\boldsymbol{\eta})$ 
imply that after prolonging 
$h_{j_1}(\boldsymbol{\eta}),\dots, h_{j_l}(\boldsymbol{\eta})$ to the order $|r(j_l)|=|r(m)|$
the prolongations of $\hat{\ell}_{j_1}(\boldsymbol{\eta}),\dots,\hat{\ell}_{j_l}(\boldsymbol{\eta})$ 
form a linear system in $\eta_1^{|r(m)|} ,\dots, \eta_l^{|r(m)|}$ of full rank 
and the derivatives appearing in the terms 
in the prolongations of the non-linear parts 
$\hat{p}_{j_1}(\boldsymbol{\eta}),\dots, \hat{p}_{j_l}(\boldsymbol{\eta})$ 
are at most of order $|r(m)+1|$.
We conclude that the transcendence degree of $C \langle \boldsymbol{\eta} \rangle$ over
 $C \langle  \boldsymbol{h} \rangle $ is finite. 
 It follows that its differential transcendence degree is zero and so the differential transcendence
 degree of $C \langle  \boldsymbol{h}  \rangle $ 
 over $C$ is $l$ by \cite[II, 9, Corollary 2]{Kolchin}. Since the extension is differentially
 generated by $l$ elements, it is a purely differential transcendental extension.
 
 We show that 
 $E$ 
 is Picard-Vessiot extension of $C \langle  \boldsymbol{h}\rangle$ and that its differential Galois group is $G(C)$.
 Since $E$ is a Picard-Vessiot extension
 of $C \langle \boldsymbol{\eta} \rangle$ by 
 Proposition~\ref{generic_liouville}, its field of constants is $C$. 
 Moreover, by construction   
 \[ 
  Y= \boldsymbol{u}\big((\eta_1, \dots,\eta_l,f_{l+1}(\boldsymbol{\eta}),
  \dots, f_{m}(\boldsymbol{\eta})) \big)
  n(\bar{w}) 
 \boldsymbol{t}(\boldsymbol{z}) \boldsymbol{u}(\boldsymbol{y})
 \]
 is a fundamental solution matrix where $f_{l+1}(\boldsymbol{\eta}),\dots ,f_{m}(\boldsymbol{\eta}) $ are as 
 in Lemma~\ref{relations_for_indeterminates}. We need to show that the 
 entries of $Y$ generate $E$ as field over 
 $C \langle  \boldsymbol{h}\rangle$. Clearly $E$
 contains $ C \langle \boldsymbol{h} \rangle (Y_{ij})$ and we need to 
 prove the other inclusion. Since $Y$ lies in $G$ we can compute 
 its Bruhat decomposition over $ C \langle \boldsymbol{h} \rangle (Y_{ij})$. Using the uniqueness of the elements of the Bruhat decomposition we get that $\boldsymbol{\eta}$, $f_{l+1}(\boldsymbol{\eta}), \dots, f_{m}(\boldsymbol{\eta})$, 
 $\boldsymbol{z}$ and $\boldsymbol{y}$ can be represented as rational expressions in the entries of $Y$. Since $ C \langle \boldsymbol{h} \rangle (Y_{ij})$
 is a differential field it follows that it is equal to $E$ and so $E$ is generated as a field by the entries of $Y$ over $C \langle \boldsymbol{h} \rangle$.
 We conclude that $E$ is a Picard-Vessiot extension of $C \langle \boldsymbol{h}) \rangle $.
 We proved above that $\boldsymbol{h}$ are differentially algebraically independent over $C$ and so it follows from Proposition 5.3, Lemma 6.8 and Theorem 4.3 of 
 \cite{Seiss} applied to $A_G(\boldsymbol{h})$ that the differential Galois group is $G(C)$.
 \end{proof}

\begin{remark}\label{corollary_noetherstyl}
Theorem~\ref{theorem_noetherstyl} shows that for the differential field 
$E=C\langle \boldsymbol{\eta}\rangle( \boldsymbol{z},\boldsymbol{y}) $
and the group $G(C)$ a construction as explained in the 
introduction is possible. Let 
 \[
  Y =\boldsymbol{u}\big( (\boldsymbol{\eta}, f_{l+1}(\boldsymbol{\eta})
  , \dots, f_{m}(\boldsymbol{\eta})) \big)  n(\bar{w}) 
 \boldsymbol{t} (\boldsymbol{z}) \boldsymbol{u}(\boldsymbol{y})
 \]
and for $g \in G(C)$ consider the Bruhat decomposition of $Y g$.
Let $(\boldsymbol{\bar{x}},\boldsymbol{\bar{z}},\boldsymbol{\bar{y}})$
be the corresponding normal form coefficients. Then the map 
\[ 
\varphi_g: E  \rightarrow  E , \ 
 ((\eta_1, \dots, \eta_l, f_{l+1}(\boldsymbol{\eta})
  , \dots, f_{m}(\boldsymbol{\eta})) , \boldsymbol{z},\boldsymbol{y}) \mapsto 
 (\boldsymbol{\bar{x}},\boldsymbol{\bar{z}},\boldsymbol{\bar{y}})
\]
 is a differential $C$-automorphism and the group $G(C)$ acts on $E$ by 
\[  
G(C) \times E \rightarrow E, \ (g,x) \mapsto \varphi_g(x)
\]
where the field of invariants is $E^G = C\langle \boldsymbol{h} \rangle$. 
Indeed, it follows from Theorem~\ref{theorem_noetherstyl} that for $g \in G$ the map $\varphi_g$ is actually the Galois automorphism induced by $Yg$ of the 
Picard-Vessiot extension $E$ of $C \langle \boldsymbol{h} \rangle$, 
since the normal form coefficients of the Bruhat decompositions of $Y$ and $Yg$ can be 
written as rational expression in the entries of $Y$ and $Y g$ respectively.

\end{remark}

\begin{example}
In this example we consider the group $G$ and its Lie algebra $\mathfrak{g}$ of type $G_2$. We compute the general extension field $E$ with its Liouvillian solutions and the invariants generating the fixed field. 
Let $\Delta = \{ \bar{\alpha}_1, \bar{\alpha}_2 \}$ be simple roots and 
we number the six negative roots according to Chapter~\ref{The Structure of the Classical Groups} as 
\begin{gather*}
\beta_1= -\bar{\alpha}_1, \ 
\beta_2= -\bar{\alpha}_2, \ 
\beta_3= -\bar{\alpha}_1 - \bar{\alpha}_2, \ 
\beta_4= -2\bar{\alpha}_1 - \bar{\alpha}_2, \ 
\beta_5= -3\bar{\alpha}_1 - \bar{\alpha}_2 \\ \mathrm{and} \quad  
\beta_6= -3\bar{\alpha}_1 - 2 \bar{\alpha}_2.
\end{gather*}
The complementary roots are $\{ \beta_2, \beta_6 \}$. We take the 
representation of $\mathfrak{g}$ as in \cite[Section 11]{Seiss}.
From the Chevalley basis presented there we compute the root group 
elements 
$u_1(\eta_1), \dots, u_6(\eta_6)$ belonging to the negative roots for 
indeterminates $\boldsymbol{\eta}_6=(\eta_1, \dots, \eta_6)$ and the torus elements  
$t_1(z_1)$, $t_2(z_2)$ for indeterminates $\boldsymbol{z}=(z_1, z_2)$. 
For the simple reflection $w_{\bar{\alpha}_1}$ and $w_{\bar{\alpha}_2}$
we choose the representatives 
\begin{eqnarray*}
n(w_{\bar{\alpha}_1})&=& -E_{11}+E_{72}-E_{63}-E_{54}+E_{45}-E_{36}-E_{27}   \\
n(w_{\bar{\alpha}_2})&=& \ \  E_{11}+ E_{22} -E_{34}+E_{43}+E_{55}-E_{67}+E_{76}
\end{eqnarray*}
and obtain the representative $n(\bar{w})=(n(w_{\bar{\alpha}_2})n(w_{\bar{\alpha}_1}))^3$ for the 
longest Weyl group element $\bar{w}$. With further indeterminates 
$\boldsymbol{y}=(y_1, \dots, y_6)$ we define 
\begin{equation*}
  Y=u_1(\eta_1) \cdots u_{6}(\eta_6) n(\bar{w})t_1(z_1)t_2(z_2)u_1(y_1) \cdots u_{6}(y_6)
\end{equation*}
and compute its logarithmic derivative
\begin{equation*}\label{exampleG1eqn1}
\ell \delta (Y)= \sum_{i=1}^6 a_{i}X_{i} +d_1 H_1+d_2 H_2 +
\sum_{i=1}^6 a_{-\beta_i}X_{-\beta_i}
\end{equation*} 
where the coefficients are differential 
rational expressions in the indeterminates  
$\boldsymbol{\eta}_6$, $\boldsymbol{z}$ and $\boldsymbol{y}$.
In \cite[Section 11]{Seiss} the parameter differential equation with Galois group $G$ is defined by the matrix 
$A_0^+ + t_1 X_2 + t_2 X_6$ with parameters $t_1$, $t_2$ and so  
we determine in the following $\boldsymbol{\eta}_6$, $\boldsymbol{z}$ and $\boldsymbol{y}$ such that the logarithmic derivative of $Y$ satisfies  
\begin{equation*}
\ell \delta (Y) = A_0^+ + h_2(\boldsymbol{\eta}) X_{2}+ h_6(\boldsymbol{\eta}) X_{6} 
\end{equation*}
where $\boldsymbol{\eta}=(\eta_1, \eta_2)$. In a first step we require that 
\begin{equation*}
a_{-\beta_6}=0, \dots, a_{-\beta_3}=0, \ a_{-\beta_2}=1, \ a_{-\beta_1}=1,  \ d_1=0, \ d_2 =0.
\end{equation*}
It is possible to solve these equations successively for $y_6', \dots, y_1'$ and $z_1'/z_1,z_2'/z_2$. Indeed, one can check that 
in each step we can solve the equation for the first derivative or the logarithmic derivative of the corresponding variable and 
substitute the obtained expression in the 
remaining equations. This process yields the solutions 
 \begin{gather*}
  z_1= e^{ \int \eta_1  }, \ z_2=e^{ \int - \eta_2 }  , \
y_1=\int \frac{-1}{z_1^2 z_2}, \
y_2 =\int - z_1^3 z_2^2, \ 
y_3 = \int y_1 y_2', \\
y_4= \int -3 y_1 y_5'-3 y_1^2 y_3'+ y_1^3 y_2', \ 
y_5= \int -2 y_1 y_3' + y_1^2 y_2' , \ 
y_6 = \int -3 y_3 y_5' - y_2 y_4'  .
\end{gather*}
The differential field extension $E=C\langle \boldsymbol{\eta} \rangle(\boldsymbol{z},\boldsymbol{y})$
of $C\langle \boldsymbol{\eta}  \rangle$ is a Picard-Vessiot extensions with 
differential Galois group $B^-$ according to Proposition~\ref{generic_liouville} and it is our general extension field.
It is left to determine the two generators of the field of invariants.
If we substitute these solutions into the logarithmic derivative of $Y$ we obtain   
 \begin{equation*}
  l\delta(Y)=A_0^+ + h_1(\boldsymbol{\eta}_6) X_{1}+ \dots + h_6(\boldsymbol{\eta}_6) X_{6}
 \end{equation*} 
with 
 \begin{eqnarray*}
h_{ 1}(\boldsymbol{\eta}_6) &=& \eta_1'-\eta_3+\eta_1^2  ,  \\ 
h_{ 2}(\boldsymbol{\eta}_6)& =& \eta_2' +3 \eta_3+\eta_2^2-3 \eta_1 \eta_2    \\
h_{3}(\boldsymbol{\eta}_6) & =& \eta_3'- h_{2}(\boldsymbol{\eta}_6) \eta_1 -2 \eta_5 - \eta_1 \eta_3 , \\
h_{ 4}(\boldsymbol{\eta}_6) &=& \eta_5'+ \eta_1 \eta_5 -\eta_4 + \eta_3^2 +  h_{2}(\boldsymbol{\eta}_6) \eta_1^2  \\
 h_{5}(\boldsymbol{\eta}_6) &=& \eta_4' - h_{ 2}(\boldsymbol{\eta}_6) \eta_1^3 - \eta_6 + \eta_4 (3 \eta_1 - \eta_2)\\
 h_{6}(\boldsymbol{\eta}_6) &=& \eta_6' + h_{2}(\boldsymbol{\eta}_6) (\eta_1^3\eta_2 - 3\eta_1^2 \eta_3) + 
\eta_2 \eta_6 + 3 \eta_5^2 - 2 \eta_3^3  .
 \end{eqnarray*}
Next we require that the coefficients of basis elements corresponding to the non-complementary roots vanish. This leads to the system of equations 
$h_1(\boldsymbol{\eta}_6)=0$, $h_3(\boldsymbol{\eta}_6)=0$, $h_4(\boldsymbol{\eta}_6)=0$ and $h_5(\boldsymbol{\eta}_6)=0$.  
These equations depend linearly on the 
variables $\eta_3$, $\eta_5$, $\eta_4$ and $\eta_6$. 
Successively we solve the 
equations for the corresponding variables and substitute the obtained 
expressions for the variables into the remaining equations. This process yields the equations 
\[ 
\eta_3= f_3(\boldsymbol{\eta}), \ 
\eta_5= f_5(\boldsymbol{\eta}), \ 
\eta_4= f_4(\boldsymbol{\eta}), \ 
\eta_6= f_6(\boldsymbol{\eta})
\]
and the invariants
 \begin{eqnarray*}
 h_{1}(\boldsymbol{\eta}) &=& \eta_2' +3 (\eta_1' + \eta_1^2) +\eta_2^2-3 \eta_1 \eta_2 \\
 h_6(\boldsymbol{\eta}) &= &-1/4((2 \eta_1'''+4 \eta_1 \eta_1''+6 \eta_1'^2+(4 \eta_1^2-2 h_1) \eta_1'-2 \eta_1 h_1'+2 \eta_1^4+2 h_1 \eta_1^2) \eta_2'\\
 &&-2 \eta_1^{(5)}-10 \eta_1 \eta_1^{(4)}+(-26 \eta_1'+2 \eta_2^2-6 \eta_1 \eta_2-16 \eta_1^2+2 h_1) \eta_1'''-19 \eta_1''^2\\
&&+(6 h_1'-90 \eta_1 \eta_1'+4 \eta_1 \eta_2^2-12 \eta_1^2 \eta_2-14 \eta_1^3+8 h_1 \eta_1) \eta_1''-18 \eta_1'^3 +(6 \eta_2^2\\ 
&&-18 \eta_1 \eta_2-39 \eta_1^2+2 h_1) \eta_1'^2+ (6 h_1''+10 \eta_1 h_1'+12 \eta_1^2 h_1+(4 \eta_1^2-2 h_1) \eta_2^2\\
&&+(6 h_1 \eta_1-12 \eta_1^3) \eta_2) \eta_1'+2 \eta_1 h_1'''+4 \eta_1^2 h_1''+(-2 \eta_1 \eta_2^2+6 \eta_1^2 \eta_2-2 \eta_1^3) h_1' \\ 
&& +(2 \eta_1^4+2 h_1 \eta_1^2) \eta_2^2+(-6 \eta_1^5-6 h_1 \eta_1^3) \eta_2+5 \eta_1^6 + 6 h_1 \eta_1^4-3 h_1^2 \eta_1^2)
\end{eqnarray*} 
where we write shortly $h_1$ for $h_1(\boldsymbol{\eta})$. 
\end{example}

\section{Full $G$-Primitive Picard-Vessiot Extensions and \\the General Extension field} 
\label{Full $G$-Primitive Picard-Vessiot Extensions and the General Extension field} 
Let $E$ be a full $G$-primitive Picard-Vessiot extension of a differential field $F$ with defining matrix $A \in \mathfrak{g}(F)$. In this 
chapter we pursue the question whether there exist a specialization of the parameters
\[
\boldsymbol{h}=(h_{j_1},\dots,h_{j_l}) \mapsto \boldsymbol{f}=(f_1,\dots,f_l)
\]
to elements of $F$ such 
that $E$ is a Picard-Vessiot extension of $F$ for the specialized matrix  $A_G(\boldsymbol{f})$.
If $A$ is gauge equivalent to some $A_G(\boldsymbol{f})$ by an element of $G(F)$, 
then the corresponding specialization has trivially the desired properties. 
We develop in this 
chapter a criterion 
depending on the differential subfield $F(\boldsymbol{x})$ of $E$, where $\boldsymbol{x}$ 
are the first 
$m$ normal form coefficients of a fundamental matrix for $A$, to decide whether the orbit of 
$A$ under gauge transformation intersects the set of matrices of shape $A_G(\boldsymbol{f})$.

The orbit of a matrix in $\mathfrak{g}(F)$ under $G(F)$ depends on the differential field $F$. 
In the special case when $F$ is equal to the field constants $C$, the gauge transformation 
reduces to the adjoint action. The orbit structure in the algebraic case was studied by B.~Kostant 
in \cite{Kost_Rep}. According to \cite[Remark 19' and Theorem 8]{Kost_Rep} 
the orbits of the matrices of shape $A_G(\boldsymbol{f})$ are in 
bijective correspondence to the orbits of maximal dimension in $\mathfrak{g}(F)$. 
Moreover in the special 
case when $F$ is differentially closed the logarithmic derivative from $G(F)$ to 
$\mathfrak{g}(F)$ is surjective and so any two matrices are gauge equivalent, that is there is 
only a single orbit. In both special cases we have a description of all matrices which are gauge 
equivalent to matrices of shape $A_G(\boldsymbol{f})$ and they show how the orbit structure 
varies with the differential field $F$.
 
For an algebraically closed differential field $F $ the following lemma shows that matrices of 
a specific subset of $\mathfrak{g}(F)$ are gauge equivalent to matrices of shape $A_G(\boldsymbol{f})$.
\begin{lemma}[Transformation Lemma]\label{Prop_tans_lemma}
 Let $F= \overline{F}$ and $A \in A_0^+(\boldsymbol{s}) +  \mathfrak{b}^-(F)$ with nonzero $\boldsymbol{s}=(s_1,\dots,s_l)$
 in $F$. Then $A$ is gauge equivalent by an element of $  B^-(F)$ to a matrix of shape 
 $A_G(\boldsymbol{f})$.  
\end{lemma}
For proof of Lemma~\ref{Prop_tans_lemma} see \cite[Lemma 6.8]{Seiss} 
\begin{lemma}\label{lem:extend_transf_lemma_general}
Let $A$ be an element of an orbit of a matrix of shape $A_G(\boldsymbol{f})$.
Then there are $w \in \mathcal{W}$, $u \in (U^{-})'_{w}$ and 
nonzero $\boldsymbol{s}=(s_1,\dots ,s_l)$ of $F$ such that $A$ is an element of the plane
\[
\mathrm{Ad}(u n(w)) (A_0^+(\boldsymbol{s}) + \mathfrak{b}^-)  + \ell \delta (u) 
\]
where $n(w)$ is a representative for $w$ and $(U^{-})'_{w}$ is as in Theorem~\ref{Bruhat1} 
with the roles of the positive and negative roots interchanged. 
\end{lemma}
\begin{proof}
Since $A$ is an element of the orbit of $A_G(\boldsymbol{f})$, there is $g \in G(F)$ such that 
\[
A= \mathrm{Ad}(g)(A_G(\boldsymbol{f}) ) + \ell \delta (g).
\]
By Theorem~\ref{Bruhat1}, where we interchange the role of the positive and negative roots, there exists $w \in \mathcal{W}$, $u \in (U^{-})'_w(F)$, $t \in T(F)$ 
and $\bar{u} \in U^-(F)$ such that $g=u n(w) t \bar{u} $ with $n(w)$ a 
representative of $w$ in the normalizer of $T$. It follows that  
\[
A = \mathrm{Ad}( u n(w)) \big( \mathrm{Ad}(t \bar{u})( A_G(\boldsymbol{f})) + \ell \delta (t \bar{u}) 
\big) + \ell \delta (u n(w)) .
\]
Finally Remark~\ref{remark_adjoint_action} and Proposition~\ref{log_derivate} imply 
that the gauge-transform of $A_G(\boldsymbol{f})$ with $t \bar{u}$ lies in the plane 
$A_0^+(\boldsymbol{s}) + \mathfrak{b}^-(F)$ for some nonzero $\boldsymbol{s}=(s_1,\dots ,s_l)$
of $F$.
\end{proof}

\begin{lemma}\label{eqtransformationweylgrp}
 Let $w \in W$ and fix a representative $n(w)$ in the normalizer of $T$. Moreover let  
 $C(\boldsymbol{x}, \boldsymbol{z},\boldsymbol{y}) $ be the rational function field in the indeterminates $\boldsymbol{x}=(x_1,\dots , x_m)$, $\boldsymbol{z}=(z_1,\dots,z_l)$ and 
 $\boldsymbol{y}=(y_1,\dots,y_m)$.
 Then there are $\boldsymbol{\bar{x}}=(\bar{x}_1,\dots , \bar{x}_m)$ in $C(\boldsymbol{x})$ 
 and $\boldsymbol{\bar{v} }=(\bar{v}_1,\dots , \bar{v}_m)$ and nonzero 
 $\boldsymbol{\bar{z}}=(\bar{z}_1,\dots , \bar{z}_l)$ in $C(\boldsymbol{x},\boldsymbol{z})$
 such that 
 \[ 
 n(w) \cdot \big( \boldsymbol{u}(\boldsymbol{x}) n(\bar{w}) \boldsymbol{t}(\boldsymbol{z})\boldsymbol{u}(\boldsymbol{y}) \big)
 = \boldsymbol{u}(\boldsymbol{\bar{x}}) n(\bar{w}) \boldsymbol{t}(\boldsymbol{\bar{z}})
 \boldsymbol{u}(\boldsymbol{\bar{v}}) \boldsymbol{u}(\boldsymbol{y}).
 \]
\end{lemma}

\begin{proof}
 For a simple root $\bar{\alpha}_i$ let 
 $G_{i}$ be the centralizer of $\mathrm{kern}(\bar{\alpha}_i) \subseteq T$ in $G$. Then 
 $G_{i}$ is of semisimple rank $1$ and it is generated according to 
 \cite[Chapter 26.2, Corollary B]{HumGroups} by the subgroups 
 $U_{\bar{\alpha}_i}$, $U_{-\bar{\alpha}_i}$ and $T$. 
 We prove that a specific relation holds in $G_{i}$. Let 
 \begin{equation*}
 \bar{u}_{\alpha}(x)=\begin{pmatrix} 1 & x \\ 0 & 1 \end{pmatrix}, \
 \bar{u}_{- \alpha}(x)=\begin{pmatrix} 1 & 0 \\ x & 1 \end{pmatrix}, \
 \bar{t}(x)=\begin{pmatrix} x & 0 \\ 0 & 1/x \end{pmatrix} ,\   
 \bar{n} =\begin{pmatrix} 0 & 1 \\ -1 & 0 \end{pmatrix} 
 \end{equation*}
 be the standard representation of $\mathrm{SL}_2(C)$ where $\alpha$ is the unique simple root and 
 $\bar{n}$ a representative of the unique nontrivial Weyl group element.
 Then a simple computation shows that
  \[
  \bar{n} \bar{u}_{-\alpha}(x) =  
  \bar{u}_{- \alpha}(-1/x) \bar{t}(x) \bar{u}_{ \alpha}(1/x).
 \]
 Applying the injective group homomorphism from $\mathrm{SL}_2(C)$ to $G_{i}$ which is 
 defined by 
 \[
 \bar{u}_{\alpha}(x) \mapsto u_{\bar{\alpha}_i}(x), \ \bar{u}_{-\alpha}(x)  \mapsto u_{-\bar{\alpha}_i}(x), \ 
 \bar{t}(x) \mapsto t_i(x) , \  \bar{n} \mapsto n(w_{\bar{\alpha}_i}) 
 \] 
 to the last equation we obtain that the corresponding relation is also true in $G_{i}$.
 
 For a simple root $\bar{\alpha}_i $  let $\varPsi= \Phi^- \setminus \{ - \bar{\alpha}_i \}$. 
 By \cite[Chapter 14.12, (2) (iii)]{Borel} the root groups $U_{\beta_j}$ with $\beta_j \in \varPsi$ 
 directly span a subgroup $U_{\varPsi}$ which is normalized by $G_{i}$ and satisfies
 $U^-= U_{-\bar{\alpha}_i} U_{\varPsi} = U_{\varPsi}  U_{-\bar{\alpha}_i}$. We conclude that 
 \[
 n(w_{\bar{\alpha}_i}) \boldsymbol{u}(\boldsymbol{x}) = 
 n(w_{\bar{\alpha}_i}) \boldsymbol{u}_{\varPsi}(\boldsymbol{\bar{x}}_{\varPsi}) 
 u_{-\bar{\alpha}_i}(\bar{x}_i) =
 \boldsymbol{u}_{\varPsi}(\boldsymbol{\bar{x}}_{\varPsi}) 
 n(w_{\bar{\alpha}_i}) u_{-\bar{\alpha}_i}(\bar{x}_i)
 \]
 where $\bar{x}_i$ and $\boldsymbol{\bar{x}}_{\varPsi}=(\bar{x}_1,\dots, \bar{x}_{i-1}, \bar{x}_{i+1},\dots, \bar{x}_m)$ are elements in $C(\boldsymbol{x})$ and 
 \[
 \boldsymbol{u}_{\varPsi}(\boldsymbol{\bar{x}}_{\varPsi}) =u_1(\bar{x}_i) \cdots u_{i-1}(\bar{x}_{i-1})u_{i+1}(\bar{x}_{i+1})\cdots u_{m}(\bar{x}_{m}) .
 \]
Since $n(w_{\bar{\alpha}_i}) u_{-\bar{\alpha}_i}(\bar{x}_i) \in G_{i}$ we obtain from 
the above relation that  
\begin{gather*}
 \boldsymbol{u}_{\varPsi}(\boldsymbol{\bar{x}}_{\varPsi}) 
 n(w_{\bar{\alpha}_i}) u_{-\bar{\alpha}_i}(\bar{x}_i) = 
 \boldsymbol{u}_{\varPsi}(\boldsymbol{\bar{x}}_{\varPsi}) 
 u_{-\bar{\alpha}_i}(-1/\bar{x}_i) t_i(\bar{x}_i)u_{\bar{\alpha}_i}(1/\bar{x}_i) = \\
 \boldsymbol{u}( \boldsymbol{\bar{x}}) 
 t_i(\bar{x}_i)u_{\bar{\alpha}_i}(1/\bar{x}_i)
\end{gather*} 
with suitable $\boldsymbol{\bar{x}}$ in $C(\boldsymbol{x})$. The torus is normalized by 
 $n(\bar{w})$ and the adjoint action of $n(\bar{w})$ maps the positive root groups 
 bijectively to the negative root groups. So we have  
 \[ 
 \boldsymbol{u}(\boldsymbol{\bar{x}}) 
 t_{i}(\bar{x}_i)u_{\bar{\alpha}_i}(1/\bar{x}_i)
 n(\bar{w}) \boldsymbol{t}(\boldsymbol{z})\boldsymbol{u}(\boldsymbol{y}) = 
 \boldsymbol{u}(\boldsymbol{\bar{x}}) n(\bar{w})
 \boldsymbol{t}(\boldsymbol{\hat{z}}) \boldsymbol{u}(\boldsymbol{\hat{v}})
  \boldsymbol{t}(\boldsymbol{z})\boldsymbol{u}(\boldsymbol{y})
 \]
with nonzero $\boldsymbol{\hat{z}}$ and $\boldsymbol{\hat{v}}$ in $ C(\boldsymbol{x})$. 
Finally, since the torus normalizes $U^-$, we obtain 
  \[
  \boldsymbol{u}(\boldsymbol{\bar{x}}) n(\bar{w})
 \boldsymbol{t}(\boldsymbol{\hat{z}}) \boldsymbol{u}(\boldsymbol{\hat{v}})
  \boldsymbol{t}(\boldsymbol{z})\boldsymbol{u}(\boldsymbol{y})=
 \boldsymbol{u}(\boldsymbol{\bar{x}}) n(\bar{w})
 \boldsymbol{t}(\boldsymbol{\bar{z}}) 
 \boldsymbol{u}(\boldsymbol{\bar{v}}) \boldsymbol{u}(\boldsymbol{y})
 \]
 with nonzero $\boldsymbol{\bar{z}}$ and $\boldsymbol{\bar{v}}$ in $C(\boldsymbol{x},\boldsymbol{z})$ which completes the proof for $w_{\bar{\alpha}_i} \in \mathcal{W}$.
 
 Now let $w \in \mathcal{W}$ be arbitrary. 
  By \cite[Theorem 10.3]{HumLie} it is the product of simple reflections 
 and so the statement of the lemma follows from applying successively the above result. 
 
\end{proof}

\begin{corollary}\label{proposition_equiv_transformation}
Let $E/F$ be a full $G$-primitive Picard-Vessiot extension with matrix $Y \in G(E)$ 
and normal form decomposition 
\[ 
 Y=\boldsymbol{u}(\boldsymbol{x}) n(\bar{w}) \boldsymbol{t}(\boldsymbol{z})\boldsymbol{u}(\boldsymbol{y}).
\]
Moreover let $u \in U^-(F)$ and $w \in \mathcal{W}$ with representative $n(w)$. 
Then there are $\boldsymbol{\bar{x}}=(\bar{x}_1,\dots, \bar{x}_m)$ in $F(\boldsymbol{x})$
and nonzero $\boldsymbol{\bar{z}}=(\bar{z}_1,\dots, \bar{z}_l)$ and 
$\boldsymbol{\bar{v}}=(\bar{v}_1,\dots,\bar{v}_l)$ in $F(\boldsymbol{x},\boldsymbol{z})$
such that
 \begin{equation*}
 ( n(w) u ) \cdot \big(  \boldsymbol{u}(\boldsymbol{x}) n(\bar{w}) \boldsymbol{t}(\boldsymbol{z})\boldsymbol{u}(\boldsymbol{y}) \big)
 =  \boldsymbol{u}\boldsymbol{(\bar{x}}) n(\bar{w}) \boldsymbol{t}(\boldsymbol{\bar{z}})
 \boldsymbol{u}(\boldsymbol{\bar{v}}) \boldsymbol{u}(\boldsymbol{y}) .
 \end{equation*}
\end{corollary}
\begin{proof}
 Since $u$ and $\boldsymbol{u}(\boldsymbol{x})$ are elements of $U^-(F(\boldsymbol{x}))$, 
 there are $\boldsymbol{\hat{x}}=(\hat{x}_1,\dots,\hat{x}_m)$ in $F(\boldsymbol{x})$
 such that $u \cdot  \boldsymbol{u}(\boldsymbol{x})=\boldsymbol{u}(\boldsymbol{\hat{x}})$.
 Then the statement follows from applying Lemma~\ref{eqtransformationweylgrp} to
 \[  
 n(w) \boldsymbol{u}(\boldsymbol{\hat{x}}) n(\bar{w}) \boldsymbol{t}(\boldsymbol{z})\boldsymbol{u}(\boldsymbol{y}) .
 \] 
\end{proof}

\begin{remark}\label{lemmaeqtransformation}
Let $E$ be a full $G$-primitive Picard-Vessiot extension of $F$ with normal 
form decomposition
\[ 
 Y=\boldsymbol{u}(\boldsymbol{x}) n(\bar{w}) \boldsymbol{t}(\boldsymbol{z})\boldsymbol{u}(\boldsymbol{y}).
\]
The fundamental theorem yields that $E$ is a Picard-Vessiot extension of 
$E^{U^-}$ with differential Galois group $U^-(C)$ and it follows from 
Proposition~\ref{proposition_equiv_transformation} that $E^{U^-}$ is  
$F(\boldsymbol{x},\boldsymbol{z})$. Since 
$\boldsymbol{u}(\boldsymbol{x}) n(\bar{w}) \boldsymbol{t}(\boldsymbol{z})$ lies in 
$G(F(\boldsymbol{x},\boldsymbol{z}))$, the matrix $\boldsymbol{u}(\boldsymbol{y})$ is 
a fundamental matrix for the extension $E$ of $F(\boldsymbol{x},\boldsymbol{z})$.
Clearly the defining matrix $\ell \delta (\boldsymbol{u}(\boldsymbol{y}))$ for the 
extension lies in the 
subalgebra $\mathfrak{u}^-(F(\boldsymbol{x},\boldsymbol{z}))$. We call 
$\boldsymbol{u}(\boldsymbol{y})$ the 
unipotent part of the normal form for $Y$. 
\end{remark}

\begin{definition}
 A gauge transformation of the logarithmic derivative of the unipotent part by an element $\boldsymbol{u}(\boldsymbol{\bar{v}}) \in U^-(F(\boldsymbol{x},\boldsymbol{z}))$ 
 as in Corollary~\ref{proposition_equiv_transformation}
 is called an equivariant transformation of the unipotent part.
\end{definition}

\begin{theorem}\label{proposition_transf_principal nilpotent}
Let $E$ be a full $G$-primitive Picard-Vessiot extension of $F$ with normal 
form decomposition
\[ 
 Y=\boldsymbol{u}(\boldsymbol{x}) n(\bar{w}) \boldsymbol{t}(\boldsymbol{z})\boldsymbol{u}(\boldsymbol{y}).
\]
Then $\ell \delta (Y)$ is gauge equivalent by an element of $G(F)$ to a matrix 
of shape $A_G(\boldsymbol{f})$ if and only if there is an equivariant transformation 
of the unipotent part to a principal nilpotent matrix in normal form $A^-_0(\boldsymbol{s})$ with 
$\boldsymbol{s}=(s_1,\dots,s_l)$ in $F(\boldsymbol{x},\boldsymbol{z})$.
\end{theorem}

\begin{proof} 
First assume that there is $g \in G(F)$ and $\boldsymbol{f}=(f_1,\dots,f_m)$ in $F$ such that 
\[ 
 \ell \delta (Y)= \mathrm{Ad}(g)( A_G(\boldsymbol{f}) ) + \ell \delta (g).
\]
We conclude with Theorem~\ref{Bruhat1} that there are $w \in \mathcal{W}$, $u \in U^-$ and 
nonzero $\boldsymbol{\bar{s}}=(\bar{s}_1,\dots,\bar{s}_l)$ in $F$ such that 
\begin{equation*}
 \mathrm{Ad}(n(w) u) (\ell \delta(Y)) + \ell \delta (n(w) u)= \bar{A}
 \in A_0^+ (\boldsymbol{\bar{s}}) + \mathfrak{b}^-(F)
\end{equation*}
where $n(w)$ is a representative of $w$. 
The last equation implies that $n(w) u Y $ is a fundamental solution matrix for $\bar{A}$. 
Moreover from Corollary~\ref{proposition_equiv_transformation} we obtain that 
there are $\boldsymbol{\bar{x}}=(\bar{x}_1,\dots, \bar{x}_m)$ in $F(\boldsymbol{x})$
and nonzero $\boldsymbol{\bar{z}}=(\bar{z}_1,\dots, \bar{z}_l)$ and 
$\boldsymbol{\bar{v}}=(\bar{v}_1,\dots,\bar{v}_l)$ in $F(\boldsymbol{x},\boldsymbol{z})$
such that
\[ 
n(w)u  Y= (n(w) u) \cdot \big( \boldsymbol{u}(\boldsymbol{x}) n(\bar{w}) \boldsymbol{t}(\boldsymbol{z})\boldsymbol{u}(\boldsymbol{y}) \big)
= \boldsymbol{u}(\boldsymbol{\bar{x}}) n(\bar{w}) \boldsymbol{t}(\boldsymbol{\bar{z}})
 \boldsymbol{u}(\boldsymbol{\bar{v}}) \boldsymbol{u}(\boldsymbol{y})
\]
and so $\boldsymbol{u}(\boldsymbol{\bar{v}})$ is an equivariant 
transformation of the unipotent part $\boldsymbol{u}(\boldsymbol{y})$. Further the 
logarithmic derivative of the fundamental matrix $n(w)u Y$ computes as 
\[
\bar{A} = \ell \delta (\boldsymbol{u}(\boldsymbol{\bar{x}})) + \mathrm{Ad}(\boldsymbol{u}(\boldsymbol{\bar{x}}) n(\bar{w})) ( \ell \delta (\boldsymbol{t}(\boldsymbol{\bar{z}})) ) + \mathrm{Ad}(\boldsymbol{u}(\boldsymbol{\bar{x}}) n(\bar{w}) \boldsymbol{t}(\boldsymbol{\bar{z}}) ) (\ell \delta (\boldsymbol{u}(\boldsymbol{\bar{v}}) \boldsymbol{u}(\boldsymbol{y})) ).
\]
The first two summands lie in $\mathfrak{b}^-$. Since $\bar{w}$ sends the negative roots 
bijectively to the positive roots and in particular the negative simple roots to the simple roots, we 
conclude that there are $\boldsymbol{s}=(s_1,\dots,s_l)$ in $F(\boldsymbol{x},\boldsymbol{z})$ 
such that 
\[ 
\ell \delta (\boldsymbol{u}(\boldsymbol{\bar{v}}) \boldsymbol{u}(\boldsymbol{y}))=A^-_0 (\boldsymbol{s}).
\]
Hence the logarithmic derivative of 
$ \boldsymbol{u}(\boldsymbol{y}) $ is gauge equivalent by
$\boldsymbol{u}(\boldsymbol{\bar{v}})$ to a principal nilpotent matrix in normal form.

Next we assume that there is an equivariant transformation of the logarithmic 
derivative of 
$\boldsymbol{u}(\boldsymbol{y})$ to a matrix of shape $A^-_0 (\boldsymbol{s})$
with nonzero $\boldsymbol{s}=(s_1,\dots,s_l)$ in $F(\boldsymbol{x},\boldsymbol{z})$.
This means that there are $w \in \mathcal{W}$, $u \in U^-(F)$, 
$\boldsymbol{\bar{x}}=(\bar{x}_1,\dots, \bar{x}_m)$ in $F(\boldsymbol{x})$
and nonzero $\boldsymbol{\bar{z}}=(\bar{z}_1,\dots, \bar{z}_l)$ and 
$\boldsymbol{\bar{v}}=(\bar{v}_1,\dots,\bar{v}_l)$ in $F(\boldsymbol{x},\boldsymbol{z})$
such that
 \begin{equation*}
 ( n(w) u ) \cdot \big(  \boldsymbol{u}(\boldsymbol{x}) n(\bar{w}) \boldsymbol{t}(\boldsymbol{z})\boldsymbol{u}(\boldsymbol{y}) \big)
 =  \boldsymbol{u}\boldsymbol{(\bar{x}}) n(\bar{w}) \boldsymbol{t}(\boldsymbol{\bar{z}})
 \boldsymbol{u}(\boldsymbol{\bar{v}}) \boldsymbol{u}(\boldsymbol{y}) 
 \end{equation*}
and such that 
\[
\ell \delta (\boldsymbol{u}(\boldsymbol{\bar{v}}) \boldsymbol{u}(\boldsymbol{y}))=A^-_0 (\boldsymbol{s}).
\]
The logarithmic derivative of the right hand side of the last equation 
is
\[
\ell \delta (\boldsymbol{u}(\boldsymbol{\bar{x}})) + \mathrm{Ad}(\boldsymbol{u}(\boldsymbol{\bar{x}}) n(\bar{w})) ( \ell \delta (\boldsymbol{t}(\boldsymbol{\bar{z}})) ) + \mathrm{Ad}(\boldsymbol{u}(\boldsymbol{\bar{x}}) n(\bar{w}) \boldsymbol{t}(\boldsymbol{\bar{z}}) ) (\ell \delta (\boldsymbol{u}(\boldsymbol{\bar{v}}) \boldsymbol{u}(\boldsymbol{y})) ).
\]
As above the first two summands lie in the subalgebra  $\mathfrak{b}^-$.
Since the logarithmic derivative of $\boldsymbol{u}(\boldsymbol{\bar{v}})\boldsymbol{u}(\boldsymbol{y})$ is
$A^-_0 (\boldsymbol{s})$, it is send by the adjoint action with $n(\bar{w}) \boldsymbol{t}(\boldsymbol{\bar{z}})$ to a matrix of shape $A^+_0 (\boldsymbol{\bar{s}})$ with 
nonzero $\boldsymbol{\bar{s}}=(\bar{s}_1,\dots,\bar{s}_l)$ in $F(\boldsymbol{x},\boldsymbol{z})$.
Further the adjoint action with $\boldsymbol{u}(\boldsymbol{\bar{x}})$ maps $A^+_0 (\boldsymbol{\bar{s}})$
to the plane $ A^+_0 (\boldsymbol{\bar{s}}) + \mathfrak{b}^-$ and so combining all results 
we get that the logarithmic derivative of $n(w) uY$ lies in the plane  
 $A_0^+(\boldsymbol{\bar{s}}) + \mathfrak{b}^-(F)$ and that $\boldsymbol{\bar{s}}=(\bar{s}_1,\dots, \bar{s}_l) $ are actually in $F$. The statement then
 follows from applying Lemma~\ref{Prop_tans_lemma}.
\end{proof}

\begin{definition}
 A full $G$-primitive Picard-Vessiot extension such that there is an 
 equivariant transformation of the unipotent part to a principal nilpotent matrix in normal form is 
 called a full $G$-primitive Picard-Vessiot extension with normalisable unipotent part.
\end{definition}

\begin{theorem}\label{main_theorem_generic}
The Picard-Vessiot extension $C \langle \boldsymbol{\eta} \rangle (\boldsymbol{z}, \boldsymbol{y})$
of $C\langle \boldsymbol{h} \rangle$ and the defining matrix $A_G(\boldsymbol{h})$ of 
Theorem~\ref{theorem_noetherstyl} 
are generic for every full $G$-primitive Picard-Vessiot extension $L$ of an 
algebraically closed differential field $F$ with normalisable unipotent part
in the following sense:
\begin{enumerate}[label=(\alph*)]
\item \label{main_theorem_generic_point1}
There exists a specialization $\sigma : C \{ \boldsymbol{h} \} \rightarrow F$ such 
that $L$ is a Picard-Vessiot extension of $F$ for 
$ A_G(\sigma(\boldsymbol{h}))$ and the defining matrix of $L$ is 
gauge equivalent to $A_G(\sigma(\boldsymbol{h}))$. 
\item \label{main_theorem_generic_point2}   
There exists a specialization $\sigma: C\{ \boldsymbol{\eta} \}[\boldsymbol{z},\boldsymbol{z}^{-1} ,\boldsymbol{y}] \rightarrow L$ such that 
\[
L=F(\sigma(Y))
\]
where $Y$ is the fundamental matrix for $ A_G(\boldsymbol{h})$.  
\item \label{main_theorem_generic_point3}
For every specialization $\sigma : C \{ \boldsymbol{h} \} \rightarrow F$ the differential 
Galois group of a Picard-Vessiot extension for $ A_G(\sigma(\boldsymbol{h}))$ is a subgroup of $G(C)$.
\end{enumerate}

\end{theorem}

\begin{proof}
Let $\bar{Y} \in G(L)$ be a fundamental solution matrix for the extension $L$ of $F$. Since the extension has a normalisable unipotent part 
and $F$ is algebraically closed it follows from 
Theorem~\ref{proposition_transf_principal nilpotent} 
that there are $g \in G(F)$ and $\boldsymbol{f}=(f_1,\dots,f_l)$ in $F$ such that 
$\ell \delta (\bar{Y})$ can be gauge transformed by $g$ to $A_G(\boldsymbol{f})$. 
This means that $L/F$ is a full $G$-primitive Picard-Vessiot extension for $A_G(\boldsymbol{f})$ with 
 fundamental solution matrix $g\bar{Y} \in G(L)$. Furthermore we have a normal form decomposition
 \[
 g \bar{Y}=\boldsymbol{u}(\boldsymbol{\bar{x}}_m) n(\bar{w})\boldsymbol{t}(\boldsymbol{\bar{z}})
 \boldsymbol{u}(\boldsymbol{\bar{y}})
 \]
 where  $\boldsymbol{\bar{x}}_m=(\bar{x}_1,\dots,\bar{x}_m)$, 
 $\boldsymbol{\bar{z}}=(\bar{z}_1,\dots,\bar{z}_l)$ and 
 $\boldsymbol{\bar{y}}=(\bar{y}_1,\dots,\bar{y}_m)$ 
 in $L$ are the normal form coefficients. For the first 
 $l$ entries $\boldsymbol{\bar{x}}=(\bar{x}_1,\dots, \bar{x}_l)$ in $\boldsymbol{\bar{x}}_m$ we define the specialization
 \[
 \sigma : C \{ \boldsymbol{\eta} \} \rightarrow L, \ \boldsymbol{\eta} \mapsto 
\boldsymbol{\bar{x}}.
 \]
 Since the logarithmic derivative of $g\bar{Y}$ is $A_G(\boldsymbol{f})$ it follows 
 from Theorem~\ref{theorem_noetherstyl} that 
 $h_{j_i}(\boldsymbol{\bar{x}})=f_i$
 and so the restriction   
  \[
 \sigma \mid_{C\{ \boldsymbol{h} \}} : C \{ \boldsymbol{h} \} \rightarrow F, \ \boldsymbol{h} \mapsto 
\boldsymbol{f} 
 \]
 satisfies statement \ref{main_theorem_generic_point1}.  
 Furthermore by the same argument we have that
  \[ 
  \bar{x}_{l+1} = f_{l+1}(\boldsymbol{\bar{x}}) , \dots ,  
  \bar{x}_{m} = f_{m}(\boldsymbol{\bar{x}})
 \]
 and that $\ell \delta(\boldsymbol{t}(\boldsymbol{\bar{z}})\boldsymbol{u}(\boldsymbol{\bar{y}}))=A_L(\boldsymbol{\bar{x}})$. Thus we can extend $\sigma$ to a 
 specialization 
 \[
 \sigma : C \{ \boldsymbol{\eta} \}[\boldsymbol{z},\boldsymbol{z}^{-1},\boldsymbol{y}] 
 \rightarrow L , \
  (\boldsymbol{\eta} , \boldsymbol{z}, \boldsymbol{y})    \mapsto 
  (\boldsymbol{\bar{x}}, \boldsymbol{\bar{z}}, \boldsymbol{ \bar{y}})  
 \]
 which satisfies  
  \begin{gather*}
 \sigma(Y)=\boldsymbol{u}\big(\sigma(\boldsymbol{\eta},f_{l+1}(\boldsymbol{\eta}), \dots ,
 f_{m}(\boldsymbol{\eta}))\big) n(\bar{w}) \boldsymbol{t}\big(\sigma(\boldsymbol{z})\big)
 \boldsymbol{u}\big(\sigma(\boldsymbol{y})\big) = 
 \boldsymbol{u}(\boldsymbol{\bar{x}})n(\bar{w})  \boldsymbol{t}( \boldsymbol{\bar{z}}) 
  \boldsymbol{u}( \boldsymbol{\bar{y}}).  
 \end{gather*}
 This shows the second statement of the theorem. Finally the third statement 
 follows from \cite[Poposition 1.31]{P/S}, since for every specialization 
 $\sigma: C\{ \boldsymbol{h} \} \rightarrow F$ the matrix 
 $A_G(\sigma(\boldsymbol{h}))$ lies in $\mathfrak{g}$. 

\end{proof}

\begin{remark}\label{rem:determinedbylelements}
The proof of Theorem~\ref{main_theorem_generic} shows that every full $G$-primitive 
Picard-Vessiot extension with normalisable unipotent part is determined by 
$l=\mathrm{rank}(G)$ many elements of $L$. More precisely, the elements $\boldsymbol{\bar{x}}=(\bar{x}_1,\dots,\bar{x}_l)$ 
 differentially generate the intermediate extension $L^{B^-}/F$ and, since $L/F$ has normalisable unipotent 
 part, they determine the Liouvillian extension $L/L^{B^-}$ by $A_L(\boldsymbol{\bar{x}})$.
\end{remark}


\bibliographystyle{plain}
\bibliography{main}

\end{document}